\documentclass[12pt,reqno]{amsart}
\usepackage{hyperref}
\usepackage{amscd,amssymb,amsfonts,amsbsy,amsmath,verbatim,color, mathrsfs}
\usepackage{pst-node}
\usepackage{tikz-cd}
\usepackage{graphicx,cite}
\usepackage{subfigure}
\usepackage{float}
\usepackage{tikz,ifthen,bm}
\usepackage{verbatim,mathdots,caption,epstopdf}
\usetikzlibrary{intersections}
\usetikzlibrary{calc}
\usepackage{amsmath}

\newtheorem{thm}{Theorem}[section]

\newtheorem{lem}[thm]{Lemma}

\newtheorem{prop}[thm]{Proposition}
\newtheorem{exam}[thm]{Example}

\theoremstyle{definition}

\newtheorem{defi}[thm]{Definition}
\newtheorem{rema}[thm]{Remark}
\numberwithin{equation}{section}

\newcommand{\E}{{\mathcal E}}

\newcommand{\Real}{\mathbb{R}}

\newcommand{\R}{\Real}
\newcommand{\Z}{{\mathbb Z}}

\def\supp{{\textup{\textrm{supp}}}}

\DeclareMathOperator{\esssup}{ess\,sup}

\makeatletter

\newcommand{\Rmnum}[1]{\expandafter\@slowromancap\romannumeral #1@}
\makeatother

\begin{document}
	\baselineskip=17pt
	\setcounter{figure}{0}
	
	\title[Differential equations]
	{Differential equations defined by Kre\u{\i}n-Feller operators on Riemannian manifolds}
	\date{\today}
	\author[S.-M. Ngai]{Sze-Man Ngai}
	\address{Beijing Institute of Mathematical Sciences and Applications, Huairou District, 101400, Beijing,  China,  and Key Laboratory of High Performance Computing and Stochastic
		Information Processing (HPCSIP) (Ministry of Education of China), College of
		Mathematics and Statistics, Hunan Normal University, Changsha, Hunan
		410081,  China.}
	\email{ngai@bimsa.cn}
	
	\author[L. Ouyang]{Lei Ouyang}
	\address{Key Laboratory of High Performance Computing and Stochastic Information
		Processing (HPCSIP) (Ministry of Education of China, College of Mathematics and Statistics, Hunan Normal University, Changsha, Hunan 410081, China.}
	\email{ouyanglei@hunnu.edu.cn}


	\thanks{The authors are supported in part by the National Natural Science Foundation of China, grant 12271156, and Construct Program of the Key Discipline in Hunan Province. The second author was supported in part by Beijing Institute of Mathematical Sciences and Applications (BIMSA)}

	\begin{abstract}
	We study linear and semi-linear wave, heat, and Schr\"odinger equations defined by Kre\u{\i}n-Feller operator $-\Delta_\mu$ on a complete Riemannian $n$-manifolds $M$, where $\mu$ is a finite positive  Borel measure on a bounded open subset $\Omega$ of $M$ with support contained in $\overline{\Omega}$. Under the assumption that $\underline{\operatorname{dim}}_{\infty}(\mu)>n-2$, we prove that for a linear or semi-linear equation of each of the above three types, there exists a unique weak  solution. We study the crucial condition $\dim_(\mu)>n-2$ and provide examples of measures on $\mathbb{S}^2$ and $\mathbb{T}^2$ that satisfy  the condition. We also study weak solutions of linear equations of the above three classes by using examples on $\mathbb{S}^1$.
		\tableofcontents
	\end{abstract}
	
	\maketitle
	\section{Introduction}\label{S:IN}
	\setcounter{equation}{0}
The Kre\u{\i}n-Feller operator, studied independently by Kre\u{\i}n and Feller, is of particular interest in fractal geometry. It is a natural generalization of the Laplace or Laplace-Beltrami operator, and has been applied successfully to study analytic properties of certain fractal measures. Let $\mu$ be a finite positive  Borel measure on a bounded open subset $\Omega$ of $\R^n$ with support contained in $\overline{\Omega}$.  If $\mu$ satisfies the Poincar\'e inequality on $\Omega$, then there exists a  Laplacian $\Delta_\mu$ defined by $\mu$. $\Delta_\mu$ is also called a Kre\u{\i}n-Feller operator (see definition in Section \ref{S:Pre}) and  has been studied extensively in connection with fractal geometry, such as existence of an orthonormal basis of eigenfunctions, spectral dimension and spectral asymptotics, eigenvalues and eigenfunctions, eigenvalue estimates, wave equations and wave speed, heat equation and heat kernel estimates, Schr\"odinger equation (see \cite{Hu-Lau-Ngai_2006, Freiberg_2003, Freiberg_2005, Freiberg_2011, Freiberg-Zahle_2002, Gu-Hu-Ngai_2020, Naimark-Solomyak_1994, Naimark-Solomyak_1995, Ngai_2011, Ngai-Tang-Xie_2018, Ngai-Xie_2020, Ngai-Xie_2021, Bird-Ngai-Teplyaev_2003,  Deng-Ngai_2015, Pinasco-Scarola_2019, Chan-Ngai-Teplyaev_2015, Pinasco-Scarola_2021, Kessebohmer-Niemann_2022-2,Tang-Ngai_2022,Kessebohmer-Niemann_2022-3} and references therein).  
	Chan {\em et al.}\cite{Chan-Ngai-Teplyaev_2015} studied approximations of the solution of the wave equation defined by a one-dimensional fractal Laplacian. Tang and Wang \cite{Tang-Wang_2022} proved the existence and uniqueness of weak solution of the strong damping linear wave equation.
		Tang and Ngai \cite{Tang-Ngai_2022} studied the heat equation on a bounded open set $U\subseteq\R^n$ supporting a Borel measure and
		obtained asymptotic bounds for the solution.
		For a Schr\"odinger operator defined by a fractal measure with a continuous
		potential and a coupling parameter, Ngai and Tang \cite{Ngai-Tang_2019}  obtained an analog of a semiclassical asymptotic formula for the number of bound states as the parameter tends
		to infinity. 
		Under the assumption that the Laplacian has compact resolvent, Ngai and Tang \cite{Ngai-Tang_2023} proved that there exists a unique weak solution for a linear Schr\"odinger equation, and obtained the existence and uniqueness of a weak solution of a semi-linear Schr\"odinger equation.

Fractal phenomena on manifolds have been observed by physicists  (see, e.g., \cite{Ambjorn-Jurkiewicz-Loll_2005,Calcagni_2010, Calcagni_2011,Benedetti-Henson_2009}); this partly motivated our work.
	In this paper, we let $M$ be a complete oriented smooth Riemannian $n$-manifold, let  $\Omega\subseteq M$ be a bounded open set, and let $\mu$ be a finite positive  Borel measure on $M$ such that ${\rm supp}(\mu) \subseteq \overline{\Omega}$ and $\mu(\Omega)>0$. Under the assumption $\underline{\operatorname{dim}}_{\infty}(\mu)>n-2$ (see definition in \eqref{eq:dim})
	the
	authors \cite{Ngai-Ouyang_2023,Ngai-Ouyang_2024} proved that a Kre\u{\i}n-Feller operator defined on $M$ has compact resolvent so that there exists an orthonormal basis of $L^{2}(\Omega, \mu)$ consisting of  eigenfunctions of $\Delta_{\mu}$. They also proved the Hodge theorem concerning the eigenvalues and eigenfunctions of the Kre\u{\i}n-Feller operator and generalization to the space of differential forms. The first objective of this paper is to study the solution of the following semi-linear wave equation defined by a Kre\u{\i}n-Feller operator on $M$, subject to the specified initial and boundary conditions:
	\begin{align}\label{eq:wnonlinear}
		\left\{
		\begin{array}{lll}
			\partial_{tt}u-\Delta_\mu u=F(u)\quad &\text{on}\,\,\,\Omega\times[0,T],\\
			u=0\quad  &\text{on}\,\,\,\partial\Omega\times[0,T],\\
			u=g,\,\,\partial_{t}u=h\quad  &\text{on}\,\,\,\Omega\times\{t=0\},
		\end{array}
		\right.
	\end{align}
	where $u:=u(t)$ is an  $L^2([0,T],\operatorname{dom}\mathcal{E})$-valued function of $t$ and $F(\cdot)\in {\rm Lip}(\operatorname{dom}\mathcal{E})$. The notations $\operatorname{dom}\mathcal{E}$, ${\rm Lip}(\operatorname{dom}\mathcal{E})$, and $L^2([0,T], \operatorname{dom}\mathcal{E})$ are defined in Section \ref{S:Pre}. Throughout this paper, if $\partial\Omega=\emptyset$, it is understood that the condition $u=0$ on $\partial\Omega\times[0,T]$ imposes no restriction on the solution. The existence and uniqueness of weak solution of equation \eqref{eq:wnonlinear} will be proved in Theorem \ref{thm:wmain2}.

	The second objective of this paper is to study the solution of the following semi-linear heat equation defined by $\Delta_\mu$:
	\begin{align}\label{eq:hnonlinear}
		\left\{
		\begin{array}{lll}
			\partial_tu-\Delta_\mu u=F(u)\quad &\text{on}\,\,\,\Omega\times[0,T],\\
			u=0\quad  &\text{on}\,\,\,\partial\Omega\times[0,T],\\
			u=g\quad  &\text{on}\,\,\,\Omega\times\{t=0\}.
		\end{array}
		\right.
	\end{align}
The existence and uniqueness of weak solution of this  equation will be described in Theorem \ref{thm:hmain2}.

	The third objective of this paper is to study the solution of the following semi-linear Schr\"odinger equation defined by $\Delta_\mu$:
	\begin{align}\label{eq:snonlinear}
			\left\{
			\begin{array}{lll}
				i\partial_tu+\Delta_\mu u=F(u)\quad &\text{on}\,\,\,\Omega\times[0,T],\\
				u=0\quad  &\text{on}\,\,\,\partial\Omega\times[0,T],\\
				u=g\quad  &\text{on}\,\,\,\Omega\times\{t=0\}.
			\end{array}
			\right.
	\end{align}
The existence and uniqueness of weak solution of  equation \eqref{eq:snonlinear} will be described in Theorem \ref{thm:smain2}.

	We also give two classes of measures that satisfy the condition $\underline{\operatorname{dim}}_{\infty}(\mu)>n-2$.  They are the invariant measures $\mu$ of iterated function systems (IFS) on $\mathbb{S}^2$ consisting of bi-Lipschitz mappings (Example \ref{exam1}), and graph iterated function systems (GIFS) of similitudes on the flat torus $\mathbb{T}^2$ (Example \ref{exam2}).  We also study weak solutions of linear equations of the above three classes by using examples  constructed on $\mathbb{S}^1$.
	
	This paper is organized as follows. In Section \ref{S:Pre}, we summarize some definitions and results that will be needed throughout the paper.  The
		existence and uniqueness of weak solutions of the wave, heat, and Schr\"odinger equations are studied in Sections \ref{S:wave}, \ref{S:heat}, and \ref{S:sch}, respectively.  In Section \ref{S:exam1}, we give an example of the invariant measures $\mu$ of iterated function systems (IFS) consisting bi-Lipschitz mappings on $\mathbb{S}^2$ that satisfy $\underline{\operatorname{dim}}_{\infty}(\mu)>0$. In Section \ref{S:exam2}, we give another class of the invariant measures $\mu$ of graph iterated function systems (GIFS) of similitudes on the torus $\mathbb{T}^2$ that satisfy $\underline{\operatorname{dim}}_{\infty}(\mu)>0$. Section \ref{S:So} illustrates weak solutions of the linear wave, heat, and Schr\"odinger equations by using examples.

	\section{Preliminaries}\label{S:Pre}
	\setcounter{equation}{0}
	In this section, we summarize some notation, definitions, and  preliminary results used
	throughout the rest of the paper. 
	\subsection{Notation}\label{s:no}

	\begin{defi}\label{defi:diffe}
		Let $X$ be a Banach space, $u: (a,b)\subseteq\R \to X$, and $t_0\in(a,b)$. We say that
			$u$ is {\em differentiable} at $t_0$ in the norm $\|\cdot\|_X$ if there exists $v_0\in X$ such that
		\begin{align*}
			\lim_{h\to 0}	\Big\|\frac{u(t_0+h)-u(t_0)}{h}-v_0\Big\|_{X}=0.
		\end{align*}
		$v_0$ is called the {\em derivative} of $u$ at $t_0$, and we write
		\begin{align*}
			v_0:=\partial_tu(t_0):=\lim_{h\to 0}	\frac{u(t_0+h)-u(t_0)}{h}.
		\end{align*}
		Similarly, the {\em second-order derivative} of $u$ at $t_0$ denoted $v_1$, is defined as
		\begin{align*}
			v_1:=\partial_{tt}u(t_0):=\lim_{h\to 0}	\frac{\partial_tu(t_0+h)-\partial_tu(t_0)}{h}.
		\end{align*}
	\end{defi}

	\begin{defi}\label{defi:l^P}(see \cite{Chan-Ngai-Teplyaev_2015})
		Let $X$ be a separable Banach space with norm $\|\cdot\|_X$. Let 
			$L^p([0,T],X)$ be the space of all measurable functions $u :[0, T ]\to X$ satisfying
		\begin{enumerate}
			\item[(1)] $\|u\|_{L^p([0,T],X)}:=\Big(\int_0^T\|u(t)\|_X^p\,dt\Big)^{1/p}< \infty$, if $1\leq p<\infty$, and
			\item[(2)] $\|u\|_{L^\infty([0,T],X)}:= \mathop{\esssup}\limits_{0\leq t\leq T}\|u(t)\|_X< \infty$, if $p=\infty$.
		\end{enumerate}	
	When no confusion is possible, we abbreviate these norms as $\|u\|_{p,X}$  and $\|u\|_{\infty,X}$.
	\end{defi}

	\begin{rema}\label{rema:bana}
		For each $1 \leq p \leq \infty$,	$L^p([0,T],X)$ is a Banach space. Moreover, if $0 \leq p_1 \leq p_2 \leq \infty$, then
		$L^{p_1}([0,T],X)\subseteq L^{p_2} ([0, T],X)$.  If $(X,\langle\cdot,\cdot\rangle_X )$ is a separable Hilbert space, then $L^2([0,T], X)$
		is a Hilbert space with the inner product
		\begin{align*}
			\langle u,v\rangle_{L^2([0,T], X)}:=\int_0^T\langle u(t), v(t)\rangle_X\,dt
		\end{align*}
		(see, e.g., \cite{Adams_1975,Wloka_1987}).
	\end{rema}
	
	\begin{defi}\label{defi:c}
		Let $X$ be a Banach space. We define $C([0,T],X)$ as {\em the vector space of all continuous functions} $u:[0, T ]\to X$ such that
		\begin{align*}
			\|u\|_{C([0,T],X)}:=\max_{0\leq t\leq T}\|u(t)\|_X<\infty.
		\end{align*}
		Similarly, we define $C^1([0,T],X)$ to be {\em the vector space of all continuous differentiable functions} $u:[0, T ]\to X$ such that
		\begin{align*}
			\|u\|_{C^1([0,T],X)}:=\max_{0\leq t\leq T}\big(\|u(t)\|_X+\|\partial_tu(t)\|_X\big)<\infty.
		\end{align*}
	\end{defi}
	\begin{defi}\label{defi:c}
		Let $\mathcal{H}$ be a Hilbert space with norm $\|\cdot\|_{\mathcal{H}}$. We say that a map $F:\mathcal{H}\to \mathcal{H}$ is {\em Lipschitz continuous} on $\mathcal{H}$ if there exists some constant $c>0$ such that 
		\begin{align*}
			\|F(u)-F(v)\|_{\mathcal{H}}\leq c	\|u-v\|_{\mathcal{H}} \quad \text{for all}\,\,u,v\in \mathcal{H}.
		\end{align*}
	We denote {\em the space of Lipschitz continuous function} on $\mathcal{H}$ by ${\rm Lip}(\mathcal{H})$.
	\end{defi}
	
	Throughout this paper, we assume that a Riemannian manifold is smooth and oriented. Also, whenever the Riemannian
	distance function is involved, we assume that the manifold is connected. Let $(M,g)$ be a Riemannian $n$-manifold with Riemannian metric $g$. Let $\nu$ be the Riemannian volume measure on $M$, i.e.,
	$$\,d\nu=\sqrt{\det g_{ij}}\, dx,$$
	where $g_{ij}$ are the components of $g$ in a coordinate chart, and $dx$ is the Lebesgue measure on $\R^n$. For any $F\subseteq M$, we let $\overline{F}$, $\partial F$, and $F^{\circ}$ denote, respectively, the closure, boundary, and interior of $F$. For a bounded open set $\Omega\subseteq M$, $C_c(\Omega)$, $C^{\infty}(\Omega)$, and $C_c^{\infty}(\Omega)$ denote, respectively, the following spaces of functions on $\Omega$: continuous functions with compact support,  $C^{\infty}$ functions, and $C^{\infty}$ functions with compact support. For $u\in C^{\infty}(\Omega)$, in local coordinates, $\nabla u=g^{ij}\partial_i u \partial_j$, where $g^{ij}=(g_{ij})^{-1}$ and $\partial_j=\partial/\partial x^j$. Let $\mu$ be a finite positive Borel measure on $M$ with ${\rm supp}(\mu)\subseteq\overline{\Omega}$. Let $L^2(\Omega,\mu):=L^2(\Omega,\mu,\mathbb{C})$ be the space of measurable functions $u\in \Omega\to \mathbb{C}$ such that 
	$$\|u\|_\mu:=\Big(\int_\Omega |u|^2\,d\mu\Big)^{1/2}<\infty.$$
	We regard $L^2(\Omega,\mu)$ as a real Hilbert space with the scalar product $\langle\cdot,\cdot\rangle_\mu$ associated to $\|\cdot\|_{\mu}$ defined as
	$$\langle u, v\rangle_\mu:={\rm re}\int_\Omega u\bar{v}\,d\mu.$$
	Let $W^{1,2}(\Omega)$ be the real Hilbert space equipped with the norm
	\begin{align}\label{eq:H}
		\|u\|_{W^{1,2}(\Omega)}:=\Big(\int_{\Omega}|u|^2\,d\nu+\int_{\Omega}|\nabla u|^2\,d\nu\Big)^{\frac{1}{2}}.
	\end{align}
	The scalar product $\langle\cdot,\cdot\rangle_{W^{1,2}(\Omega)}$ associated to $\|\cdot\|_{W^{1,2}(\Omega)}$ is defined as
	$$ \left\langle u, v\right\rangle_{W^{1,2}(\Omega)} :={\rm re}\int_{\Omega}u\bar{v} \,d\nu +{\rm re}\int_{\Omega}\langle\nabla u, \nabla \bar{v}\rangle_g \,d\nu,$$
	where $\langle\cdot,\cdot\rangle_g= g(\cdot,\cdot)$.  Let $W^{1,2}_0(\Omega)$ denote the closure of $C^\infty_c(\Omega)$ in the $W^{1,2}(\Omega)$ norm.

	We denote the {\em Euclidean distance} by $d_{\mathbb{E}}(\cdot, \cdot)$. For a connected Riemannian $n$-manifold $M$, we denote the {\em Riemannian distance} by $d_M(\cdot,\cdot)$. Let
	\begin{align*}
		&B(x,r):=\{y\in \R^n: d_{\mathbb{E}}(x,y)<r\}, \quad x\in \R^n,\\
		&B^M(p,r):=\{q\in M: d_M(p,q)<r\},\quad p\in M.
	\end{align*}
	Let $\mu$ be a finite positive Borel measure on $M$.  Recall that the {\em lower} and {\em upper $L^{\infty}$-dimensions of $\mu$}
	are defined respectively as
	\begin{align}\label{eq:dim}
		\underline{\dim}_\infty(\mu):=\displaystyle{\varliminf_{\delta\to 0^+}}\frac{\ln (\sup_x \mu(B^{M}(x,\delta)))}{\ln \delta}
	\end{align}
	where the supremum is taken over all $x\in{\rm supp}(\mu)$. Similarly, one can define $\overline{\dim}_\infty(\mu)$.
	If the limit exists, we denote the common value by $\dim_\infty(\mu)$.
	
	\subsection{ Laplacian defined by a measure}\label{s:lap}
	Let $M$ be a complete oriented smooth Riemannian $n$-manifold. Let  $\Omega\subseteq M$ be a bounded open set and let $\mu$ be a finite positive Borel measure on $M$ with ${\rm supp}(\mu)\subseteq\overline{\Omega}$. Throughout this paper, we assume $\mu(\Omega)>0$. We introduce the following {\em Poincar\'e inequalities for a measure $\mu$:}
\begin{enumerate}
			\item[(MPID)] (Dirichlet boundary condition)  There exists a constant $C>0$ such that for all $u\in C^{\infty}_c(\Omega)$,
			\begin{align}\label{eq:PI}
				\int_\Omega |u|^2\,d\mu\leq C\int_\Omega |\nabla u|^2\,d\nu.
			\end{align}
			\item[(MPIE)] ($\partial\Omega=\emptyset$) There exists a constant $C>0$ such that for all $u\in C^{\infty}_c(\Omega)$,
			\begin{align}\label{eq:PIN}
				\int_\Omega |u|^2\,d\mu\leq C\Big(\int_\Omega |\nabla u|^2\,d\nu+\int_\Omega |u|^2\,d\nu\Big).
			\end{align}
		\end{enumerate}
		(MPID) (resp. (MPIE)) implies that each equivalence class $u \in W^{1,2}_0(\Omega)$ (resp. $u \in W^{1,2}(\Omega)$) contains a unique (in $L^{2}(\Omega, \mu)$ sense) member $\overline{u}$ that belongs to $L^{2}(\Omega, \mu)$ and satisfies both conditions below:
		\begin{enumerate}
			\item[(1)] There exists a sequence $\left\{u_{k}\right\}$ in $C_{c}^{\infty}(\Omega)$ such that $u_{k} \rightarrow \overline{u}$ in $W^{1,2}_0(\Omega)$ (resp. $W^{1,2}(\Omega)$) and $u_{k} \rightarrow \overline{u}$ in $L^{2}(\Omega, \mu)$;
			\item[(2)] $\overline{u}$ satisfies the inequality in \eqref{eq:PI} (resp. \eqref{eq:PIN}).
		\end{enumerate}
		
		We call $\overline{u}$ the {\em $L^{2}(\Omega, \mu)$-representative} of $u$. Assume  $\mu$ satisfies (MPID) (resp. (MPIE)) on $\Omega$ and define a mapping $I_D: W^{1,2}_0(\Omega) \rightarrow L^{2}(\Omega, \mu)$ (resp. $I_E: W^{1,2}(\Omega) \rightarrow L^{2}(\Omega, \mu)$) by
		$$
		I_D(u)=\overline{u},\qquad(\text{resp.}\,\,I_E(u)=\overline{u}).
		$$
		Notice that $I_D$ and $I_E$ are bounded linear operators but are not necessarily injective. Hence we consider a subspace $\mathcal{N}_D$ of $W^{1,2}_0(\Omega)$  defined as
		$$
		\mathcal{N}_D:=\left\{u \in W^{1,2}_0(\Omega):\|I_D(u)\|_{\mu}=0\right\}.$$
		Similarly, we can define $\mathcal{N}_E$. Since $\mu$ satisfies (MPID) (resp. (MPIE)) on $\Omega$, $\mathcal{N}_D$ (resp. $\mathcal{N}_E$) is a closed subspace of $W^{1,2}_0(\Omega)$ (resp. $W^{1,2}(\Omega)$). Let $\mathcal{N}_D^{\perp}$ (resp. $\mathcal{N}_E^{\perp}$) be the orthogonal complement of $\mathcal{N}_D$ in $W^{1,2}_0(\Omega)$ (resp. $\mathcal{N}_E$ in $W^{1,2}(\Omega)$). Then $I_D: \mathcal{N}_D^{\perp} \rightarrow L^{2}(\Omega, \mu)$ (resp. $I_E: \mathcal{N}_E^{\perp} \rightarrow L^{2}(\Omega, \mu)$) is injective. Throughout this paper, if no confusion is possible, we will denote $\overline{u}$ simply by $u$.
		
		Now, we consider non-negative bilinear forms $\mathcal{E}_D(\cdot, \cdot)$ and $\mathcal{E}_E(\cdot, \cdot)$ on $L^{2}(\Omega, \mu)$ given by
		\begin{align}\label{eq(1.1)}
			\mathcal{E}_E(u, v)=\mathcal{E}_D(u, v):=\int_{\Omega} \langle\nabla u, \nabla v\rangle\,  d \nu
		\end{align}
		with $\operatorname{dom}(\mathcal{E}_D)=\mathcal{N}_D^{\perp}$ and $\operatorname{dom}(\mathcal{E}_E)=\mathcal{N}_E^{\perp}$.
		$(\mathcal{E}_D, \operatorname{dom}(\mathcal{E}_D))$ and $(\mathcal{E}_E, \operatorname{dom}(\mathcal{E}_E))$ are closed quadratic forms on $L^{2}(\Omega,\mu)$, and therefore there exists a non-negative definite self-adjoint operator $-\Delta_\mu^D$ on $L^{2}(\Omega,\mu)$ such that $\operatorname{dom}\mathcal{E}_D=\operatorname{dom}((\Delta_\mu^D)^{1/2})$ and
		$$\mathcal{E}_D(u,v)=\langle (-\Delta_\mu^D)^{1/2}u, (-\Delta_\mu^D)^{1/2}u \rangle_{\mu} \quad \text{for all}\,\, u, v\in \operatorname{dom}\mathcal{E}_D.$$
		We call $-\Delta_\mu^D$ the \textit{Dirichlet Laplacian with respect to $\mu$}. We also call $-\Delta_{\mu}^D$ a {\em Kre\u{\i}n-Feller operator}.  Similarly, we can define $\Delta_\mu^E$. In this paper, if no confusion is possible, we will denote  $\Delta_\mu^D$ and $\Delta_\mu^E$  simply by $\Delta_\mu$. Similarly, we denote $\operatorname{dom}(\mathcal{E}_D)$ and $\operatorname{dom}(\mathcal{E}_E)$ simply by $\operatorname{dom}(\mathcal{E})$.
	
	It is known that $u \in \operatorname{dom}(\Delta_\mu)$ if and only if $u \in \operatorname{dom}\mathcal{E}$ and there exists $f \in L^{2}(\Omega, \mu)$ such that $\mathcal{E}(u, v)=\langle f, v\rangle_{\mu}$ for all $v \in \operatorname{dom}\mathcal{E}$ (see, e.g., \cite{Davies_1995}). In this paper, we let $$\|\cdot\|_{\operatorname{dom}\mathcal{E}}:=\sqrt{\mathcal{E}(\cdot, \cdot)}.$$ 
Note that $\|\cdot\|_{\operatorname{dom}\mathcal{E}}$ is a norm on $\mathcal{E}/ \{\text{constants}\}$.

	The authors proved the following results (\cite[Theorem 2.2]{Ngai-Ouyang_2023}): Let $n\geq 1$, $M$ be a complete connected Riemannian $n$-manifold, and $\Omega \subseteq M$ be a bounded open set. Let $\mu$ be a finite positive  Borel measure on $M$ such that $\operatorname{supp}(\mu) \subseteq \overline{\Omega}$ and $\mu(\Omega)>0$. Assume that $\underline{\operatorname{dim}}_{\infty}(\mu)>n-2$.   Then there exists an
	orthonormal basis $\left\{\varphi_{k}\right\}_{k=0}^{\infty}$ of $L^{2}(\Omega, \mu)$ consisting of  eigenfunctions of $-\Delta_{\mu}.$ The eigenvalues $\left\{\lambda_{k}\right\}_{k=0}^{\infty}$ satisfy $0=\lambda_0<\lambda_{1} \leq \lambda_{2} \leq \cdots$,  where $\lambda_{0}=0$ only in the $\partial\Omega=\emptyset$ case. If $\dim(\operatorname{dom}\mathcal{E}) =\infty$, then $\lim _{k \to \infty} \lambda_{k}=\infty$. Hence we have
\begin{align}\label{eq:dome}
			\operatorname{dom}\mathcal{E}&=\Big\{\sum_{k=0}^\infty a_k\varphi_k:\sum_{k=0}^\infty|a_k|^2\lambda_k<\infty\Big\}\qquad \text{and}\\
			\operatorname{dom}(\Delta_\mu)&=\Big\{\sum_{k=0}^\infty a_k\varphi_k:\sum_{k=0}^\infty|a_k|^2\lambda_k^2<\infty\Big\}.
	\end{align}

	By the Lax-Milgram theorem, for any $w\in (\operatorname{dom}\mathcal{E})'$, there exists a unique $u\in \operatorname{dom}\mathcal{E}$ such that
	$$	\mathcal{E}(u, v):=\langle w,  v\rangle\qquad \text{for all}\,\,v\in \operatorname{dom}\mathcal{E},$$
	 where throughout this paper, $\langle \cdot,\cdot\rangle$ denotes the pairing between $(\operatorname{dom}\mathcal{E})'$ and $\operatorname{dom}\mathcal{E}$. 
Hence we can define a bijective operator $L$ from $\operatorname{dom}\mathcal{E}$ to $(\operatorname{dom}\mathcal{E})'$ by 
	\begin{align}\label{eq:L}
		Lu=w,
	\end{align}
	and equip $(\operatorname{dom}\mathcal{E})'$ with the scalar product
	$$\langle u,v\rangle_{(\operatorname{dom}\mathcal{E})'}:=\mathcal{E}(L^{-1}u, L^{-1}v)$$
	with the norm
		\begin{align*}
			\|w\|_{(\operatorname{dom}\mathcal{E})'}:=\|L^{-1}w\|_{\operatorname{dom}\mathcal{E}}\qquad \text{for}\,\,w\in (\operatorname{dom}\mathcal{E})'.
	\end{align*}
	Note that $\operatorname{dom} L=\operatorname{dom}\mathcal{E}$ and  $\langle w,  v\rangle=\langle w,  v\rangle_\mu$ for all $w\in (\operatorname{dom}\mathcal{E})'$ and $v\in \operatorname{dom}\mathcal{E}$.
	It follows that $L$ is an extension of $-\Delta_{\mu}$. 
	
	By identifying $L^2(\Omega,\mu)$ with $(L^2(\Omega,\mu))'$,  we have that $\operatorname{dom}\mathcal{E}\hookrightarrow L^2(\Omega,\mu)\hookrightarrow(\operatorname{dom \mathcal{E}})' $, and $\{\varphi_k\}_{k=0}^\infty$ is a complete orthogonal set of $\operatorname{dom}\mathcal{E}$. Hence $w=\sum_{k=0}^\infty a_k\varphi_k\in (\operatorname{dom}\mathcal{E})'$ if and only if there exists a unique $L^{-1}w=\sum_{k=0}^\infty b_k\varphi_k\in \operatorname{dom}\mathcal{E}$ such that $\mathcal{E}(L^{-1}w, u)=\langle w,  u\rangle$ for all $v\in \operatorname{dom}\mathcal{E}$ (see \cite{Ngai-Tang_2023}). Substituting $v=\varphi_k$ for $k\geq 1$, we get $a_k=\langle w, \varphi_k \rangle=\mathcal{E}(L^{-1}w, \varphi_k)=b_k\lambda_k$. For $k=0$, we have $\lambda_k=0$. Hence $w=\sum_{k=0}^\infty a_k\varphi_k\in (\operatorname{dom\mathcal{E}})'$ if and only if $\|w\|_{(\operatorname{dom \mathcal{E}})'}^2=\|L^{-1}w\|_{\operatorname{dom \mathcal{E}}}^2=\sum_{k=0}^\infty a_k^2/\lambda_k<\infty$. Therefore, for every $u=\sum_{k=0}^\infty a_k\varphi_k\in \operatorname{dom}\mathcal{E}$, we have $Lu=\sum_{k=0}^\infty a_k\lambda_k\varphi_k\in (\operatorname{dom}\mathcal{E})'$, and 
		\begin{align*}
			(\operatorname{dom}\mathcal{E})'=\Big\{\sum_{k=0}^\infty a_k\varphi_k:\sum_{k=1}^\infty|a_k|^2/\lambda_k<\infty\Big\}.
	\end{align*} 
	
	\begin{defi}
		For $\alpha\geq 0$, define
		\begin{align*}
			E_\alpha(\Omega,\mu):=\Big\{\sum_{k=0}^\infty b_k\varphi_k:\sum_{k=0}^\infty |b_k|^2\lambda_k^\alpha<\infty\Big\}
		\end{align*}
		with the norm $\|\cdot\|_{E_\alpha(\Omega,\mu)}$ given by
		\begin{align*}
			\|u\|_{E_\alpha(\Omega,\mu)}:=\Big(\sum_{k=0}^\infty |b_k|^2\lambda_k^\alpha\Big)^{1/2}\qquad\text{for}\,\,u=\sum_{k=0}^\infty b_k\varphi_k.
		\end{align*}
	\end{defi}
	\begin{prop}\label{prop:2.8}
		Let $M$ be a complete oriented smooth Riemannian $n$-manifold. Let  $\Omega\subseteq M$ be a bounded open set and let $\mu$ be a finite positive Borel measure on $M$ with ${\rm supp}(\mu)\subseteq\overline{\Omega}$. For all $\alpha\geq0$, $(E_\alpha(\Omega,\mu),\|\cdot\|_{E_\alpha(\Omega,\mu)})$ is a Hilbert space. In particular, $E_0(\Omega,\mu)=L^2(\Omega, \mu)$, $E_1(\Omega,\mu)=\operatorname{dom}\mathcal{E}$, and $E_2(\Omega,\mu)=\operatorname{dom}(\Delta_\mu)$. Moreover, $E_\alpha(\Omega,\mu)$ is dense in $\operatorname{dom}\mathcal{E}$ for $\alpha>1$.
	\end{prop}
	The proof of Proposition \ref{prop:2.8} is similar to that of \cite[Proposition 2.2]{Hu_2002}; we omit the details.

	\section{Weak solutions of wave equations} \label{S:wave}
	\setcounter{equation}{0}
	In this section, we let $M$ be a complete oriented smooth Riemannian $n$-manifold. Let  $\Omega\subseteq M$ be a bounded open set and let $\mu$ be a finite positive Borel measure on $M$ with ${\rm supp}(\mu)\subseteq\overline{\Omega}$. We prove the existence and uniqueness of weak solutions of the semi-linear wave equation in \eqref{eq:wnonlinear}.

	\begin{defi}\label{def:ww}
		Let $0<T<\infty$. Let $g\in \operatorname{dom}\mathcal{E}$, $h\in L^2(\Omega,\mu)$ and $f\in L^\infty([0,T],\operatorname{dom}\mathcal{E})$. A function $u\in L^\infty([0,T],\operatorname{dom}\mathcal{E})$, $\partial_tu \in L^\infty([0,T],L^2(\Omega,\mu))$, and $\partial_{tt}u \in L^\infty([0,T],(\operatorname{dom}\mathcal{E})')$ is a {\em weak solution} of the following wave equation:
		\begin{align}\label{eq:wlinear}
			\left\{
			\begin{array}{lll}
				\partial_{tt}u-\Delta_\mu u=f(t)\quad &\text{on}\,\,\,\Omega\times[0,T],\\
				u=0\quad  &\text{on}\,\,\,\partial\Omega\times[0,T],\\
				u=g,\,\,\partial_{t}u=h\quad  &\text{on}\,\,\,\Omega\times\{t=0\},\\
			\end{array}
			\right.
		\end{align}
		if the following conditions are satisfied:
		\begin{enumerate}
			\item[(1)] $\langle \partial_{tt}u,v \rangle+\mathcal{E}(u,v)=\langle f(t),v\rangle_\mu$ for each $v\in \operatorname{dom}\mathcal{E}$ and Lebesgue a.e. $t\in [0,T]$;
			\item[(2)] $u(0)=g$ and $\partial_{t}u(0)=h$.
		\end{enumerate}
	\end{defi}	
	\begin{rema}\label{rema:3.1}
		Here we comment on Definition \ref{def:ww}.
		\begin{enumerate}
			\item[(a)] The boundary condition $u|_{\partial \Omega}=0$ in \eqref{eq:wlinear} is included in the assumption $u(t)\in \operatorname{dom}\mathcal{E}$. If $u\in L^\infty([0,T],\operatorname{dom}\mathcal{E})$, $\partial_tu \in L^\infty([0,T],L^2(\Omega,\mu))$, and $\partial_{tt}u \in L^\infty([0,T],\\(\operatorname{dom}\mathcal{E})')$, then $u\in C([0,T], L^2(\Omega,\mu))$, and thus the initial conditions $u(0)=g$ and $\partial_{t}u(0)=h$ make sense.
			\item[(b)] Condition (1) is equivalent to
			\begin{align*}
				\partial_{tt}u+Lu=f(t)\quad \text{in}\,\,(\operatorname{dom}\mathcal{E})'\,\,\text{for Lebesgue a.e.}\,\,t\in [0,T],
			\end{align*}
			where $L$ is defined as in \eqref{eq:L}.
		\end{enumerate}
	\end{rema}
Let $\left\{\varphi_{k}\right\}_{k=0}^{\infty}$ be an orthonormal basis  of $L^{2}(\Omega, \mu)$ such that  $-\Delta_{\mu}\varphi_{k}=\lambda_k\varphi_{k}$.  Let 
		\begin{align*}
			g=&\sum_{k=0}^\infty \alpha_k\varphi_k\in L^2(\Omega,\mu),\qquad h=\sum_{k=0}^\infty \beta_k\varphi_k\in L^2(\Omega,\mu),\quad\text{and}\\
			f(t)&=\sum_{k=0}^\infty \gamma_k(t)\varphi_k\in L^2([0,T], L^2(\Omega,\mu)),
		\end{align*}
		where $\gamma_k(t)=\langle f(t), \varphi_k\rangle_\mu$ for $k\geq 1$. Let
		\begin{align*}
			u_0(t):=\sum_{k=1}^\infty\frac{\sin(\sqrt{\lambda_k}t) }{\sqrt{\lambda_k}}\beta_k\varphi_k+\sum_{k=0}^\infty\alpha_k\cos(\sqrt{\lambda_k}t)\varphi_k.
		\end{align*}
		Define
		\begin{align}\label{eq:wut1}
			u(t):=u_0(t)+\sum_{k=1}^\infty\frac{\varphi_k}{\sqrt{\lambda_k}}\int_0^t\sin\big(\sqrt{\lambda_k}(t-\tau)\big)\gamma_k(\tau)d\tau=:\sum_{k=0}^\infty c_k\varphi_k,
		\end{align}
		\begin{align}\label{eq:wHt}
			H(t):=&\sum_{k=1}^\infty\beta_k\cos(\sqrt{\lambda_k}t)\varphi_k-\sum_{k=0}^\infty\alpha_k\sqrt{\lambda_k}\sin(\sqrt{\lambda_k}t)\varphi_k\nonumber\\
			&+ \sum_{k=1}^\infty\varphi_k\int_0^t\cos(\sqrt{\lambda_k}(t-\tau))\gamma_k(\tau)d\tau=:\sum_{k=0}^\infty d_k\varphi_k
		\end{align}
		and
		\begin{align}\label{eq:wK}
			K(t):=&-\sum_{k=1}^\infty\sqrt{\lambda_k}\beta_k\sin(\sqrt{\lambda_k}t)\varphi_k-\sum_{k=0}^\infty\lambda_k\alpha_k\cos(\sqrt{\lambda_k}t)\varphi_k+f(t)\nonumber\\
			&-\sum_{k=1}^\infty\sqrt{\lambda_k}\varphi_k\int_0^t\sin(\sqrt{\lambda_k}(t-\tau))\gamma_k(\tau)d\tau.
	\end{align}
	
	We have the following theorem.
	\begin{thm}\label{thm:wmain1}
		Let $M$ be a complete oriented smooth Riemannian $n$-manifold. Let  $\Omega\subseteq M$ be a bounded open set and let $\mu$ be a finite positive Borel measure on $M$ with ${\rm supp}(\mu)\subseteq\overline{\Omega}$.  Assume that $\underline{\operatorname{dim}}_{\infty}(\mu)>n-2$.
		Let $g$, $h$, $f(t)$, $u(t)$, $H(t)$ and $K(t)$ be defines as above. For $g\in \operatorname{dom}\mathcal{E}$, $h\in L^2(\Omega,\mu)$, and $f(t)\in L^\infty([0,T], \operatorname{dom}\mathcal{E})$, the following hold:
		\begin{enumerate}
			\item[(a)]  $\partial_{t}u=H(t)$ in $L^2(\Omega,\mu)$ for Lebesgue a.e. $t\in [0,T]$.
			\item[(b)] $\partial_{tt}u=K(t)$ in $(\operatorname{dom}\mathcal{E})'$ for Lebesgue a.e. $t\in [0,T]$.
			\item[(c)] $u(t)$ is the unique weak solution of the wave equation \eqref{eq:wlinear}.
			\item[(d)] If $g\in E_{\alpha}(\Omega,\mu)$ and $f\in L^\infty([0,T],E_{\alpha}(\Omega,\mu))$, where $\alpha\geq 2$, then $u(t)\in L^\infty([0,T]$,\\$E_{\alpha}(\Omega,\mu))$ and $\partial_t u(t)\in L^\infty([0,T],E_{\alpha-2}(\Omega,\mu))$.
			\item[(e)] If $f\equiv0$, then
			\begin{align*}
				\|u(t)\|_\mu\leq c_1\|h\|_{\operatorname{dom}\mathcal{E}}^2+\|g\|_\mu^2\quad\text{and}\quad \mathcal{E}(u,u)\leq c_2 \mathcal{E}(h,h)+\mathcal{E}(g,g)\quad \text{for all}\,\,t\in [0,T],
			\end{align*}
			where $c_1$ and $c_2$ are positive constants.
		\end{enumerate}
	\end{thm}
	\begin{proof}
The proof of this theorem is similar to that of \cite[Theorem 3.1]{Ngai-Tang_2023}; we only indicate the modifications.
	
		(a) Let $\delta$ satisfy $0<2\delta<T$ and $\delta<1$. 	Note that the classical derivative $c_k'(t)=d_k(t)$ for Lebesgue a.e. $t\in [0,T]$. It follows
			that
			\begin{align}\label{eq:wsk1}
				s_k(t,h):=\frac{c_k(t+h)-c_k(t)}{h}-d_k(t)\to 0\quad \text{as}\quad h\to 0
			\end{align}
			for Lebesgue a.e. $t\in [\delta,T-\delta]$ and $h\in (-\delta,\delta)$. For all $t\in [\delta,T-\delta]$ and each $h\in (-\delta,\delta)$,  we remark that 
			\begin{align*}
				& \Big|\int_t^{t+h}\sin(\sqrt{\lambda_k}(t+h-\tau))\gamma_k(\tau)d\tau\Big|^2\\
				\leq& \int_t^{t+h}\big|\sin(\sqrt{\lambda_k}(t+h))\cos(\sqrt{\lambda_k}\tau)-\cos(\sqrt{\lambda_k}(t+h))\sin(\sqrt{\lambda_k}(\tau))\big|^2d\tau\int_t^{t+h}\big|\gamma_k(\tau)\big|^2d\tau\\
				=&\Big(\frac{h}{2}-\frac{\sin(2\sqrt{\lambda_k}h)}{4\sqrt{\lambda_k}}\Big)\int_t^{t+h}\big|\gamma_k(\tau)\big|^2d\tau\\
				\leq&\frac{h^4\lambda_k}{3} \mathop{\esssup}\limits_{t\in [0,T]}\big|\gamma_k(t)\big|^2,
			\end{align*}
			where the last inequality follows by using the inequality $\sin x\geq x-x^3/6$.
			Hence for all $t\in [\delta,T-\delta]$ and each $h\in (-\delta,\delta)$,
			\begin{align}\label{eq:wck1}
				&|c_k(t+h)-c_k(t)|^2\nonumber\\
				=&\Bigg|\frac{\sin(\sqrt{\lambda_k}(t+h))-\sin(\sqrt{\lambda_k}t)}{\sqrt{\lambda_k}}\beta_k+\big(\cos(\sqrt{\lambda_k}(t+h))-\cos(\sqrt{\lambda_k}t)\big)\alpha_k\nonumber\\&+\frac{1}{\sqrt{\lambda_k}} \int_t^{t+h}\sin(\sqrt{\lambda_k}(t+h-\tau))\gamma_k(\tau)d\tau\nonumber\\&+\frac{1}{\sqrt{\lambda_k}}\int_0^t\big(\sin(\sqrt{\lambda_k}(t+h-\tau))-\sin(\sqrt{\lambda_k}(t-\tau))\big)\gamma_k(\tau)d\tau\Bigg|^2\nonumber\\
				\leq&\,3h^2\Big(|\beta_k|^2+\lambda_k|\alpha_k|^2+T\int_0^T|\gamma_k(\tau)|^2d\tau\Big) +\frac{h^4}{3} \mathop{\esssup}\limits_{t\in [0,T]}\big|\gamma_k(t)\big|^2
				\nonumber\\
				=:&3h^2M_k+\frac{h^4}{3} \mathop{\esssup}\limits_{t\in [0,T]}\big|\gamma_k(t)\big|^2,
		\end{align}
		where the facts $|\sin(xt+xh)-\sin(xt)|\leq |x|h$  and $|\cos(xt+xh)-\cos(xt)|\leq |x|h$ are used in the inequality.
		We remark that 
	\begin{align}\label{eq:Mk}
				&\sum_{k=1}^\infty M_k=\|h\|_{\mu}^2+\|g\|_{\operatorname{dom}\mathcal{E}}^2+T\|f(t)\|_{2,L^2(\Omega,\mu)}^2<\infty \quad\text{and}\nonumber\\
				&\sum_{k=1}^\infty  \mathop{\esssup}\limits_{t\in [0,T]}\big|\gamma_k(t)\big|^2=\|f(t)\|_{\infty,L^2(\Omega,\mu)}<\infty.
		\end{align}
		Using \eqref{eq:wHt} and H$\ddot{\rm o}$lder's  inequality, we have
		\begin{align}\label{eq:wdk}
			&|d_k(t)|^2
			\leq3\Big(|\beta_k|^2+\lambda_k|\alpha_k|^2+T\int_0^T|\gamma_k(\tau)|^2d\tau\Big)
			=3M_k.
		\end{align}
		Combining \eqref{eq:wck1} and \eqref{eq:wdk}, we have for Lebesgue a.e. $t\in [\delta,T-\delta]$ and $h\in (-\delta,\delta)$,
		\begin{align}\label{eq:wsk2}
			|s_k(t,h)|^2\leq 2\Big(\frac{|c_k(t+h)-c_k(t)|^2}{h^2}+|d_k(t)|^2\Big)\leq 12M_k+ \mathop{\esssup}\limits_{t\in [0,T]}\big|\gamma_k(t)\big|^2.
		\end{align}
		Using \eqref{eq:Mk}, \eqref{eq:wsk2}, and Weierstrass' M-test, we see the series $\sum_{k=0}^\infty|s_k(t,h)|^2$ converges uniformly for all $h\in (-\delta,\delta)$ and Lebesgue a.e. $t\in [\delta,T-\delta]$. Thus, for Lebesgue a.e. $t\in [\delta,T-\delta]$,
		\begin{align*}
			\lim_{h\to 0}\Big\|\frac{u(t+h)-u(t)}{h}-H(t)\Big\|_{\mu}^2=	\lim_{h\to 0}\sum_{k=0}^\infty|s_k(t,h)|^2
			=\sum_{k=0}^\infty\lim_{h\to 0}|s_k(t,h)|^2=0,
		\end{align*}
		where \eqref{eq:wsk1} is used in the last equality. It follows that $\partial_tu(t)=H(t)$ in $L^2(\Omega,\mu)$ for Lebesgue a.e. $t\in [\delta,T-\delta]$. The desired result follows by letting $\delta\to 0^{+}$.
		
		(b) We can prove that  $\partial_{tt}u=K(t)$ in $(\operatorname{dom}\mathcal{E})'$ for Lebesgue a.e. $t\in [0,T]$ by using a method similar to that in (a).

		(c) We first note that $u(0)=g$, $\partial_t(0)=h$ and
		\begin{align}\label{eq:wlu}
			Lu=&\sum_{k=1}^\infty\sin(\sqrt{\lambda_k}t) \beta_k\varphi_k\sqrt{\lambda_k}+\sum_{k=0}^\infty\alpha_k\cos(\sqrt{\lambda_k}t)\lambda_k\varphi_k\nonumber\\&+\sum_{k=1}^\infty\sqrt{\lambda_k}\varphi_k\int_0^t\sin(\sqrt{\lambda_k}(t-\tau))\gamma_k(\tau)d\tau,
		\end{align}
		where $L$ is defined as in \eqref{eq:L}. Combining \eqref{eq:wlu} and part (b), we have  $\partial_{tt}u(t)+Lu(t)=f(t)$ on $(\operatorname{dom}\mathcal{E})'$ for  Lebesgue a.e. $t\in [0,T]$. It follows from Remark \ref{rema:3.1} that $u(t)$ is a weak solution of \eqref{eq:wlinear}. 
		
		To prove uniqueness, it suffices to show that the only solution of \eqref{eq:wlinear} with $f(t)\equiv g\equiv h \equiv 0$ is $u(t)\equiv 0$. Let $u$ be a weak solution of \eqref{eq:wlinear} with $f(t)\equiv g \equiv h\equiv 0$. Let $\partial_tu\in L^\infty([0,T],L^2(\Omega,\mu))$. Then  
		\begin{align*}
			\langle \partial_{tt}u,\partial_{t}u \rangle_\mu+ \mathcal{E}(u,\partial_tu)=0 \qquad \text{for Lebesgue a.e. $t\in [0,T]$}.
		\end{align*}
		Note that
		\begin{align*}
			\frac{d}{dt}\|\partial_{t}u\|_\mu^2= 2 \langle \partial_{tt}u,\partial_{t}u \rangle_\mu\qquad\text{and}\qquad
			\frac{d}{dt}\|u\|_{\operatorname{dom}\mathcal{E}}^2
			=2\mathcal{E}(u,\partial_tu).
		\end{align*}
		Hence 
		\begin{align*}
			\frac{d}{dt}\Big(\frac{1}{2}\|\partial_{t}u\|_\mu^2+\frac{1}{2}\|u\|_{\operatorname{dom}\mathcal{E}}^2\Big)=0,
		\end{align*}
		i.e., for  Lebesgue a.e. $s\in [0,T]$,
		\begin{align*}
			\frac{1}{2}\|\partial_{t}u(s)\|_\mu^2+\frac{1}{2}\|u(s)\|_{\operatorname{dom}\mathcal{E}}^2=0.
		\end{align*}
		We know that $u(0)=0$ and $\partial_tu(0)=0$. Hence we have $\partial_tu=0$ in $L^\infty([0,T],L^2(\Omega,\mu))$ and $u=0$ in $L^\infty([0,T],\operatorname{dom}\mathcal{E})$. 

		(d) By using \eqref{eq:wut1}--\eqref{eq:wK} and H$\ddot{\rm o}$lder's  inequality, we have  $u(t)\in L^\infty([0,T], E_{\alpha}(\Omega,\mu))$ and $\partial_t u(t)\in L^\infty([0,T],E_{\alpha-2}(\Omega,\mu))$.
		
		(e) Since $f\equiv 0$,  by parts (a), (b), and (c), for all $t\in [0,T]$, we have
		\begin{align*}
			&\|u(t)\|_\mu^2
			\leq\sum_{k=1}^\infty|\beta_k|^2/\lambda_k+\sum_{k=0}^\infty|\alpha_k|^2\leq
			c_1\|h\|_{\operatorname{dom}\mathcal{E}}^2+\|g\|_\mu^2 \quad\text{and}\\
			&\mathcal{E}(u,u)
			\leq\sum_{k=1}^\infty|\beta_k|^2+\sum_{k=0}^\infty|\alpha_k|^2\lambda_k
			\leq c_2\,\mathcal{E}(h,h)+\mathcal{E}(g,g).
		\end{align*}
	\end{proof}
	
	\begin{thm}\label{thm:wmain2}
		Let $M$ be a complete oriented smooth Riemannian $n$-manifold. Let  $\Omega\subseteq M$ be a bounded open set and let $\mu$ be a finite positive Borel measure on $M$ with ${\rm supp}(\mu)\subseteq\overline{\Omega}$.  Assume that $\underline{\operatorname{dim}}_{\infty}(\mu)>n-2$ and $F(\cdot)\in {\rm Lip}(\operatorname{dom}\mathcal{E})$. 
		Let $g=\sum_{k=0}^\infty \alpha_k\varphi_k\in \operatorname{dom}\mathcal{E}$ and $h=\sum_{k=0}^\infty \beta_k\varphi_k\in L^2(\Omega,\mu)$. Then the semi-linear wave equation \eqref{eq:wnonlinear}
		has a unique weak solution $u(t)\in L^\infty([0,T], \operatorname{dom}\mathcal{E})$, which can be expressed as
		\begin{align*}
			u(t)=&\sum_{k=1}^\infty\frac{\sin(\sqrt{\lambda_k}t) }{\sqrt{\lambda_k}}\beta_k\varphi_k+\sum_{k=0}^\infty\alpha_k\cos(\sqrt{\lambda_k}t)\varphi_k\\&+\sum_{k=1}^\infty\frac{\varphi_k}{\sqrt{\lambda_k}}\int_0^t\sin(\sqrt{\lambda_k}(t-\tau))\langle F(u(\tau)),\varphi_k\rangle_\mu d\tau.
		\end{align*}
		Moreover, under the additional assumptions that $F(\cdot)\in {\rm Lip}(E_{\alpha}(\Omega,\mu))$, $g\in E_{\alpha}(\Omega,\mu)$, and $h\in E_{\alpha-1}(\Omega,\mu)$ where $\alpha\geq 2$, we have  $u(t)\in L^\infty([0,T],E_{\alpha}(\Omega,\mu))$, $\partial_t u(t)\in L^\infty([0,T],\\E_{\alpha-1}(\Omega,\mu))$, and $\partial_{tt} u(t)\in L^\infty([0,T],E_{\alpha-2}(\Omega,\mu))$.
	\end{thm}
	\begin{proof}
	The proof of this theorem is similar to that of \cite[Theorem 1.1]{Ngai-Tang_2023} and is omitted.
	\end{proof}

	\section{Weak solutions of heat equations} \label{S:heat}
	\setcounter{equation}{0}
	In this section, we let $M$ be a complete oriented smooth Riemannian $n$-manifold. Let  $\Omega\subseteq M$ be a bounded open set and let $\mu$ be a finite positive Borel measure on $M$ with ${\rm supp}(\mu)\subseteq\overline{\Omega}$. 
	
	\begin{defi}\label{def:hw}
		Let $0<T<\infty$. Let $g\in \operatorname{dom}\mathcal{E}$ and $f\in L^\infty([0,T],\operatorname{dom}\mathcal{E})$. A function $u\in L^\infty([0,T],\operatorname{dom}\mathcal{E})$, with $\partial_tu \in L^\infty([0,T],(\operatorname{dom}\mathcal{E})')$ is a {\em weak solution} of the heat equation 
		\begin{align}\label{eq:hlinear}
			\left\{
			\begin{array}{lll}
				\partial_tu-\Delta_\mu u=f(t)\quad &\text{on}\,\,\,\Omega\times[0,T],\\
				u=0\quad  &\text{on}\,\,\,\partial\Omega\times[0,T],\\
				u=g\quad  &\text{on}\,\,\,\Omega\times\{t=0\},\\
			\end{array}
			\right.
		\end{align}
		if the following conditions are satisfied:
		\begin{enumerate}
			\item[(1)] $\langle \partial_tu,v \rangle+\mathcal{E}(u,v)=\langle f(t),v\rangle_\mu$ for each $v\in \operatorname{dom}\mathcal{E}$ and Lebesgue a.e. $t\in [0,T]$;
			\item[(2)] $u(0)=g$.
		\end{enumerate}
	\end{defi}	
	\begin{rema}\label{rema:5.1}
		Here we comment on Definition \ref{def:hw}.
		\begin{enumerate}
			\item[(a)] The boundary condition $u|_{\partial \Omega}=0$ in \eqref{eq:hlinear} is included in the assumption $u(t)\in \operatorname{dom}\mathcal{E}$. If $u\in L^\infty([0,T],\operatorname{dom}\mathcal{E})$ and $\partial_tu\in L^\infty([0,T],(\operatorname{dom}\mathcal{E})')$, then $u\in C([0,T],\\ L^2(\Omega,\mu))$, and thus the initial condition $u(0)=g$ makes sense.
			\item[(b)] Condition (1) is equivalent to
			\begin{align*}
				\partial_tu+Lu=f(t)\quad \text{in}\,\,(\operatorname{dom}\mathcal{E})'\,\,\text{for Lebesgue a.e.}\,\,t\in [0,T],
			\end{align*}
			where $L$ is defined as in \eqref{eq:L}.
		\end{enumerate}
	\end{rema}
	Let $\left\{\varphi_{k}\right\}_{k=0}^{\infty}$ be an orthonormal basis  of $L^{2}(\Omega, \mu)$ such that  $-\Delta_{\mu}\varphi_{k}=\lambda_k\varphi_{k}$ and let 
	\begin{align*}
		g:=\sum_{k=0}^\infty b_k\varphi_k\in L^2(\Omega,\mu)\quad\text{and}\quad f(t):=\sum_{k=0}^\infty \beta_k(t)\varphi_k\in L^2([0,T], L^2(\Omega,\mu)),
	\end{align*}
	where $\beta_k(t):=\langle f(t), \varphi_k\rangle_\mu$ for $k\geq 1$. Define
	\begin{align*}
		u(t):=\sum_{k=0}^\infty b_ke^{-\lambda_kt}\varphi_k+\sum_{k=0}^\infty \Big(\int_0^te^{-\lambda_k(t-\tau)}\beta_k(\tau)d\tau\Big)\varphi_k
	\end{align*}
	and
	\begin{align}\label{eq:hK}
		K(t):=-\sum_{k=0}^\infty b_k\lambda_ke^{-\lambda_kt}\varphi_k+f(t)-\sum_{k=0}^\infty\lambda_k \Big(\int_0^te^{-\lambda_k(t-\tau)}\beta_k(\tau)d\tau\Big)\varphi_k.
	\end{align}
	\begin{thm}\label{thm:hmain1}
		Let $M$ be a complete oriented smooth Riemannian $n$-manifold. Let  $\Omega\subseteq M$ be a bounded open set and let $\mu$ be a finite positive Borel measure on $M$ with ${\rm supp}(\mu)\subseteq\overline{\Omega}$.  Assume that $\underline{\operatorname{dim}}_{\infty}(\mu)>n-2$.
		Let $g$, $f(t)$, $u(t)$, and $K(t)$ be defined as above. If $g\in \operatorname{dom}\mathcal{E}$ and $f(t)\in L^\infty([0,T], \operatorname{dom}\mathcal{E})$, then the following hold:
		\begin{enumerate}
			\item[(a)] $\partial_tu=K(t)$ in $(\operatorname{dom}\mathcal{E})'$ for Lebesgue a.e. $t\in [0,T]$.
			\item[(b)] $u(t)$ is the unique weak solution of the heat equation \eqref{eq:hlinear}.
			\item[(c)] If $g\in E_{\alpha}(\Omega,\mu)$ and $f\in L^\infty([0,T],E_{\alpha}(\Omega,\mu))$, where $\alpha\geq 2$, then $u(t)\in L^\infty([0,T],\\E_{\alpha}(\Omega,\mu))$ and $\partial_t u(t)\in L^\infty([0,T],E_{\alpha-2}(\Omega,\mu))$.
			\item[(d)] If $f\equiv0$, then
			\begin{align*}
				\|u(t)\|_\mu\leq\|g\|_\mu\quad\text{and}\quad \mathcal{E}(u,u)\leq\mathcal{E}(g,g)\quad \text{for all}\,\,t\in [0,T].
			\end{align*}
		\end{enumerate}
	\end{thm}
	\begin{proof}
With some modifications, one can prove this theorem by following the arguments in  \cite[Theorem 3.1]{Ngai-Tang_2023}.
	\end{proof}
	
		\begin{thm}\label{thm:hmain2}
	Let $n$, $M$, and $\mu$ be as in Theorem \ref{thm:wmain2}.   Assume that $\underline{\operatorname{dim}}_{\infty}(\mu)>n-2$ and $F(\cdot)\in {\rm Lip}(\operatorname{dom}\mathcal{E})$. 
		Let $g=\sum_{k=0}^\infty b_k\varphi_k\in \operatorname{dom}\mathcal{E}$. Then the  semi-linear heat equation
		\eqref{eq:hnonlinear}
		has a unique weak solution $u(t)\in L^\infty([0,T], \operatorname{dom}\mathcal{E})$, which can be expressed as
	\begin{align*}
				u(t)= \sum_{k=0}^\infty b_ke^{-\lambda_kt}\varphi_k+\sum_{k=0}^\infty \Big(\int_0^te^{-\lambda_k(t-\tau)}\langle F(u(\tau)),\varphi_k\rangle_\mu d\tau\Big)\varphi_k.
		\end{align*}
		Moreover, under the additional assumptions that $F(\cdot)\in {\rm Lip}(E_{\alpha}(\Omega,\mu))$ and  $g\in E_{\alpha}(\Omega,\mu)$, where $\alpha\geq 2$, we have  $u(t)\in L^\infty([0,T],E_{\alpha}(\Omega,\mu))$ and $\partial_t u(t)\in L^\infty([0,T],E_{\alpha-2}(\Omega,\mu))$.
	\end{thm}	
	\begin{proof}The proof follows by using the results of Theorem \ref{thm:hmain1}, and a similar argument as that in the proof of  \cite[Theorem 1.1]{Ngai-Tang_2023}.
	\end{proof}
	
	\section{Weak solutions of Schr\"odinger equations} \label{S:sch}
	\setcounter{equation}{0}
	In this section, we also let $M$ be a complete oriented smooth Riemannian $n$-manifold. Let  $\Omega\subseteq M$ be a bounded open set and let $\mu$ be a finite positive Borel measure on $M$ with ${\rm supp}(\mu)\subseteq\overline{\Omega}$.

	\begin{defi}\label{def:sw}
		Let $0<T<\infty$. Let $g\in \operatorname{dom}\mathcal{E}$ and $f\in L^\infty([0,T],\operatorname{dom}\mathcal{E})$. A function $u\in L^\infty([0,T],\operatorname{dom}\mathcal{E})$, with $\partial_tu \in L^\infty([0,T],(\operatorname{dom}\mathcal{E})')$ is a {\em weak solution} of the Schr\"odinger equation
		\begin{align}\label{eq:slinear}
			\left\{
			\begin{array}{lll}
				i\partial_tu+\Delta_\mu u=f(t)\quad &\text{on}\,\,\,\Omega\times[0,T],\\
				u=0\quad  &\text{on}\,\,\,\partial\Omega\times[0,T],\\
				u=g\quad  &\text{on}\,\,\,\Omega\times\{t=0\},\\
			\end{array}
			\right.
		\end{align}
		if the following conditions are satisfied:
		\begin{enumerate}
			\item[(1)] $\langle i\partial_tu,v \rangle-\mathcal{E}(u,v)=\langle f(t),v\rangle_\mu$ for each $v\in \operatorname{dom}\mathcal{E}$ and Lebesgue a.e. $t\in [0,T]$;
			\item[(2)] $u(0)=g$.
		\end{enumerate}
	\end{defi}	
	\begin{rema}\label{rema:6.1}
		Here we comment on Definition \ref{def:sw}.
		\begin{enumerate}
			\item[(a)] The boundary condition $u|_{\partial \Omega}=0$ in \eqref{eq:slinear} is included in the assumption $u(t)\in \operatorname{dom}\mathcal{E}$. If $u\in L^\infty([0,T],\operatorname{dom}\mathcal{E})$ and $\partial_tu\in L^\infty([0,T],(\operatorname{dom}\mathcal{E})')$, then $u\in C([0,T],\\ L^2(\Omega,\mu))$, and thus the initial condition $u(0)=g$ makes sense.
			\item[(b)] Condition (1) is equivalent to
			\begin{align*}
				i\partial_tu-Lu=f(t)\quad \text{in}\,\,(\operatorname{dom}\mathcal{E})'\,\,\text{for Lebesgue a.e.}\,\,t\in [0,T],
			\end{align*}
			where $L$ is defined as in \eqref{eq:L}.
		\end{enumerate}
	\end{rema}
	Let 
	\begin{align*}
		g=\sum_{k=0}^\infty b_k\varphi_k\in L^2(\Omega,\mu)\quad\text{and}\quad f(t)=\sum_{k=0}^\infty \beta_k(t)\varphi_k\in L^2([0,T], L^2(\Omega,\mu)),
	\end{align*}
	where $\beta_k(t)=\langle f(t), \varphi_k\rangle_\mu$ for $k\geq 1$. Define
	\begin{align*}
		u(t):=\sum_{k=0}^\infty b_ke^{-i\lambda_kt}\varphi_k-i\sum_{k=0}^\infty\Big (\int_0^te^{-i\lambda_k(t-\tau)}\beta_k(\tau)d\tau\Big)\varphi_k
	\end{align*}
	and
	\begin{align}\label{eq:sK}
		K(t):=-i\sum_{k=0}^\infty b_k\lambda_ke^{-i\lambda_kt}\varphi_k-if(t)-\sum_{k=0}^\infty\lambda_k \Big(\int_0^te^{-i\lambda_k(t-\tau)}\beta_k(\tau)d\tau\Big)\varphi_k.
	\end{align}
	\begin{thm}\label{thm:smain1}
		Let $M$ be a complete oriented smooth Riemannian $n$-manifold. Let  $\Omega\subseteq M$ be a bounded open set and let $\mu$ be a finite positive Borel measure on $M$ with ${\rm supp}(\mu)\subseteq\overline{\Omega}$.  Assume that $\underline{\operatorname{dim}}_{\infty}(\mu)>n-2$.
		Let $g$, $f(t)$, $u(t)$, and $K(t)$ be defined as above. If $g\in \operatorname{dom}\mathcal{E}$ and $f(t)\in L^\infty([0,T], \operatorname{dom}\mathcal{E})$, then the following hold:
		\begin{enumerate}
			\item[(a)] $\partial_tu=K(t)$ in $(\operatorname{dom}\mathcal{E})'$ for Lebesgue a.e. $t\in [0,T]$.
			\item[(b)] $u(t)$ is the unique weak solution of the Schr$\ddot{o}$dinger equation \eqref{eq:slinear}.
			\item[(c)] If $g\in E_{\alpha}(\Omega,\mu)$ and $f\in L^\infty([0,T],E_{\alpha}(\Omega,\mu))$, where $\alpha\geq 2$, then $u(t)\in L^\infty([0,T],\\E_{\alpha}(\Omega,\mu))$ and $\partial_t u(t)\in L^\infty([0,T],E_{\alpha-2}(\Omega,\mu))$.
			\item[(d)] If $f\equiv0$, then
			\begin{align*}
				\|u(t)\|_\mu=\|g\|_\mu\quad\text{and}\quad \mathcal{E}(u,u)=\mathcal{E}(g,g)\quad \text{for all}\,\,t\in [0,T].
			\end{align*}
		\end{enumerate}
	\end{thm}
	The proof of Theorem \ref{thm:smain1} is similar to that of \cite[Theorem 3.1]{Ngai-Tang_2023}; we omit the proof.
		We have the following theorem.	
	
\begin{thm}\label{thm:smain2}
		Let $n$, $M$, and $\mu$ be as in Theorem \ref{thm:wmain2}.  Assume that $\underline{\operatorname{dim}}_{\infty}(\mu)>n-2$ and $F(\cdot)\in {\rm Lip}(\operatorname{dom}\mathcal{E})$. 
		Let $g=\sum_{k=0}^\infty b_k\varphi_k\in \operatorname{dom}\mathcal{E}$. Then the semi-linear Schr$\ddot{o}$dinger equation \eqref{eq:snonlinear} has a unique weak solution $u(t)\in L^\infty([0,T], \operatorname{dom}\mathcal{E})$, which is given by
			\begin{align*}
				u(t)=\sum_{k=0}^\infty b_ke^{-i\lambda_kt}\varphi_k-i\sum_{k=0}^\infty \Big(\int_0^te^{-i\lambda_k(t-\tau)}\langle F(u(\tau)),\varphi_k\rangle_\mu d\tau\Big)\varphi_k.
		\end{align*}
		Moreover, under the additional assumptions that $F(\cdot)\in {\rm Lip}(E_{\alpha}(\Omega,\mu))$ and  $g\in E_{\alpha}(\Omega,\mu)$, where $\alpha\geq 2$, we have  $u(t)\in L^\infty([0,T],E_{\alpha}(\Omega,\mu))$ and $\partial_t u(t)\in L^\infty([0,T],E_{\alpha-2}(\Omega,\mu))$.
	\end{thm}
		\begin{proof}[Proof of Theorem \ref{thm:smain2}] 
			This follows by using Theorem \ref{thm:smain1} and a similar proof as that of \cite[Theorem 1.1]{Ngai-Tang_2023}.
	\end{proof}

	\section{Examples of IFS} \label{S:exam1}
	\setcounter{equation}{0}
	In this section, we let $M$ be a complete $n$-dimensional smooth Riemannian manifold. Let $\Omega\subseteq M$ be a bounded open set. 
	Let $\{f_i\}_{i=1}^{m}$ be a finite set of {\em contractions} on $M$, i.e., for each $i$, there exists $r_i$ with $0<r_i<1$ such that
	\begin{align}\label{eq:j1.1}
		d_M(f_i(p),f_i(q))\leq r_i d_M(p,q)\quad \text{for all}\,\, p, q\in M.
	\end{align}
	Then there exists a unique nonempty compact set $K$ such that $K=\bigcup^m_{i=1}f_i(K)$ (see Hutchinson\cite{Hutchinson_1981}). We call a family of contractions $\{f_i\}_{i=1}^{m}$ on $M$ an {\em iterated function system} (IFS), and call $K$ the {\em invariant set} or {\em attractor} of the IFS. If equality in (\ref{eq:j1.1}) holds,  then $f_i$ is called a {\em contractive similitude}.  Let $(p_1,\ldots,p_m)$ be a {\em probability vector}, i.e., $p_i>0$ and $\sum_{i=1}^{m}p_i=1$. Then there is a unique Borel probability measure $\mu$ with $\supp (\mu)=K$, called the {\em invariant measure}, that satisfies
	\begin{align*}
		\mu=\sum_{i=1}^{m} p_i\mu\circ f_i^{-1}.
	\end{align*}
	
	For $\tau:=\left(i_{1}, \ldots, i_{k}\right)$, we denote by $|\tau|=k$ the length of $\tau$ and let $f_{\tau}:=f_{i_{1}} \circ \cdots \circ f_{i_{k}}$, $p_{\tau}:=p_{i_{1}}  \cdots  p_{i_{k}}$. For the invariant set $K$, we let $K_{\tau}:=f_{\tau}(K)$.
	We call $\left\{f_{i}\right\}_{i=1}^{m}:\overline{\Omega}\to \overline{\Omega}$ an IFS of {\em $bi$-Lipschitz contractions} if for each $i=$ $1, \ldots, m$, there exist $c_{i}, r_{i}$ with $0<c_{i} \leq r_{i}<1$ such that
	\begin{align}\label{eq:bi}
		c_{i}d_M(p,q) \leq d_M(f_{i}(p),f_{i}(q)) \leq r_{i}d_M(p,q) \quad \text { for all } p, q\in \overline{\Omega}.
	\end{align}
	
	We have the following lemma.
	
	\begin{lem}\label{lem:4.1}
		Let $M$ be a complete $n$-dimensional smooth Riemannian manifold. Let $\Omega\subseteq M$ be a bounded open set. 	Let $\mu$ be an invariant measure of an IFS $\left\{f_{i}\right\}_{i=1}^{m}$ of bi-Lipschitz contractions on $\overline{\Omega}$. Suppose the attractor $K$ is not a singleton. Then $\mu$ is upper $s$-regular for some $s>0$, and hence $\underline{\operatorname{dim}}_{\infty}(\mu)>0$.
	\end{lem}
	The proof of this lemma is similar to that of \cite[Lemma 5.1]{Hu-Lau-Ngai_2006} and is omitted.

	\begin{prop}\label{prop:6.1}
		Let $\mathbb{S}^2$ be a unit 2-sphere with center $O$. Let $$\mathbb{S}^2_+:=\{(\sin\varphi\cos\theta,\sin\varphi\sin\theta,\cos\varphi):0\leq \theta<2\pi, 0\leq \varphi\leq \pi/2\}$$ be a subset of $\mathbb{S}^2$, parameterized by the polar angle $\varphi\in [0,\pi/2]$ and the azimuthal angle $\theta\in[0,2\pi)$. For any point $p:=(\sin\varphi\cos\theta,\sin\varphi\sin\theta,\cos\varphi)\in \mathbb{S}^2_+$, let $f:\mathbb{S}^2_+\to \mathbb{S}^2_+$ be a mapping defined by
			\begin{align*}
				f(p)=\Big(\sin\frac{\varphi}{2}\cos\theta,\sin\frac{\varphi}{2}\sin\theta,\cos\frac{\varphi}{2}\Big).
		\end{align*}
		Then $f$ is a bi-Lipschitz contraction on $\mathbb{S}^2_+$, i.e., there exist $c, r$ with $0<c \leq r<1$ such that
		\begin{align}\label{eq:bi}
			cd_M(p,q) \leq d_M(f(p),f(q)) \leq rd_M(p,q) \quad \text { for all } p, q \in \mathbb{S}^2_+.
		\end{align}
		
	\end{prop}
	
	\begin{proof}
		For any two points $p,q\in\mathbb{S}^2_+$ with respect parameters $(\varphi_1,\theta_1)$ and $(\varphi_2,\theta_2)$, where  $\varphi_1\geq\varphi_2$, we let  $a:=\pi/2-\varphi_1$, $b:=\pi/2-\varphi_2$, $\alpha:=|\theta_1-\theta_2|$. Then $b\geq a$, $a,b\in [0,\pi/2]$, and $\alpha\in [0,2\pi)$.
			It follows from distance formula on $\mathbb{S}^2$ that
				\begin{align}\label{eq:dis1}
					d_M(p,q)=\arccos\big(\sin a\sin b+\cos a \cos b \cos\alpha\big)
				\end{align}
				(see, e.g., \cite[p35]{Green_1985}).
				Let $w$ be the point at which the parallel of $p$ intersects the meridian of $q$ and let $(\varphi_1,\theta_2)$ be its parameter.
			By the law of cosines on $\mathbb{S}^2$ (see, e.g. \cite[p6]{Green_1985}) and \eqref{eq:dis1}, we have
		\begin{align*}
			d_M(p,q)=&\arccos\big(\cos(d_M(p,w))\cos(d_M(q,w))\big)
			=\arccos\big((\sin^2a+\cos^2 a \cos\alpha )\cos(b-a)\big).
		\end{align*}
		Hence 
		\begin{align}\label{eq:dis}
			&\frac{d_M(f(p),f(q))}{d_M(p,q)}\nonumber\\
			=&\frac{\arccos\big((\sin^2(a/2+\pi/4)+\cos^2 (a/2+\pi/4) \cos\alpha )\cos((b-a)/2)\big)}{\arccos\big((\sin^2a+\cos^2 a \cos\alpha )\cos(b-a)\big)}.
		\end{align}

	\noindent{\em Step 1.} We first show that \eqref{eq:dis} has a positive lower bound $c$. Since $\arccos x$ is non-negative and decreasing on $[-1,1]$, the right-hand side of \eqref{eq:dis} is non-negative and equals zero  only  when $a=b=\alpha=0$ and $a=b=\pi/2$. Using \eqref{eq:dis} and L'H\^{o}pital's rule, we have 
			\begin{align}
				&\lim\limits_{(\alpha,a,b)\to (0,0,0)}\frac{d_M(f(p),f(q))}{d_M(p,q)}=\frac{1}{2}\label{eq:low1}\qquad\text{and}\\
				&\lim\limits_{(a,b)\to (\pi/2, \pi/2)}\frac{d_M(f(p),f(q))}{d_M(p,q)}=\frac{1}{2}\label{eq:low2}.
		\end{align}
		For any $\epsilon_1\in (0,1/4)$, there exists $\delta_1>0$ such that for $\alpha,a,b\in[0,\delta_1)$,
		\begin{align*}
			\Big|\frac{d_M(f(p),f(q))}{d_M(p,q)}-\frac{1}{2}\Big|<\epsilon_1,
		\end{align*}
		For any $\epsilon_2\in (0,1/4)$, there exists $\delta_2>0$ such that for $\alpha\in[0,2\pi)$ and $a,b\in(\pi/2-\delta_2,\pi/2]$ with $b\geq a$,
		\begin{align*}
			\Big|\frac{d_M(f(p),f(q))}{d_M(p,q)}-\frac{1}{2}\Big|<\epsilon_2,
		\end{align*}
		Let $c_1:=\min\{1/2+\epsilon_1,1/2+\epsilon_2\}$. Then for any $\alpha\in[\delta_1,2\pi)$ and  $a,b\in[\delta_1,\pi/2-\delta_2]$ with $b\geq a$,
		\begin{align*}
			\frac{d_M(f(p),f(q))}{d_M(p,q)}\geq c_1.
		\end{align*}
		For any $\widetilde{\delta_1}\geq\delta_1$, by \eqref{eq:low1} and the continuity of $d_M(f(p),f(q))/d_M(p,q)$, there exists $c_2\in(0,1)$ such that for $a,b,c\in[0,\widetilde{\delta_1}]$,
		\begin{align*}
			\frac{d_M(f(p),f(q))}{d_M(p,q)}\geq c_2.
		\end{align*}
		For any $\widetilde{\delta_2}\geq\delta_2$, by \eqref{eq:low2} and the continuity of $d_M(f(p),f(q))/d_M(p,q)$, there exists $c_3\in(0,1)$ such that for $a,b\in[2\pi-\widetilde{\delta_2},2\pi)$ and $\alpha\in[0,2\pi)$,
		\begin{align*}
			\frac{d_M(f(p),f(q))}{d_M(p,q)}\geq c_3.
		\end{align*}
		Hence for any  $a,b\in[0,\pi/2]$ and $\alpha\in[0,2\pi)$ with $b\geq a$,
		\begin{align}\label{eq:low3}
			\frac{d_M(f(p),f(q))}{d_M(p,q)}\geq c,
		\end{align}
		where $c:=\min\{1/2,c_1,c_2,c_3\}$.
		
		We will often assume that $a,b$ lie in the following intervals:
			\begin{align}
				a\in [0,\pi/4],\quad b\in[0,\pi/2],\quad b\geq a,\quad a\neq 0\,\,\text{or}\,\,b\neq \pi/2.\label{eq:int1}
			\end{align}
			
		\noindent	Step 2. We show that \eqref{eq:dis} has a positive lower bound $r$. Let
		\begin{align}
			g(a,b,\alpha):=&\big(\sin^2(a/2+\pi/4)+\cos^2 (a/2+\pi/4) \cos\alpha \big)\cos((b-a)/2),\nonumber\\
			h(a,b,\alpha):=&(\sin^2a+\cos^2 a \cos\alpha )\cos(b-a),\label{eq:F} \\
			F(a,b,\alpha):=&\frac{g(a,b,\alpha)}{h(a,b,\alpha)}.\nonumber
		\end{align}
	Then \begin{align}\label{eq:fab}
				\partial_\alpha( F(a,b,\alpha))=\frac{\cos((b-a)/2)\sec(b-a)(\sin a+\cos(2a)) \sin\alpha}{2(\cos^2 a \cos\alpha+\sin^2a )^2}.
		\end{align}
		For $\alpha\in [0,\pi]$ and $a,b\in[0,\pi/2]$ with $b\geq a$, if $h(a,b,\alpha)=0$,  then  either $a=0$ and $b=\pi/2$, or $\alpha=\arccos(-\tan^2a)$. Similarly, for $\alpha\in [\pi,2\pi)$ and $a,b\in[0,\pi/2]$ with $b\geq a$, if $h(a,b,\alpha)=0$, then  either $a=0$ and $b=\pi/2$, or $\alpha=2\pi-\arccos(-\tan^2a)$. Hence we consider the following three cases. 
		
		\noindent{\em Case I. $a,b$ as in \eqref{eq:int1}, and $\alpha\in [0,2\pi)$.} We further subdivide this into four subcases, according to the value of $\alpha$.
			
			\noindent {\em Subcase 1. $\alpha\in[0,\arccos(-\tan^2a)]\subseteq [0,\pi/2]$.} In this subcase, $\partial_\alpha( F(a,b,\alpha))\geq 0$.
			Hence $F(a,b,\alpha)$ is an increasing function of $\alpha$ and thus
			\begin{align*}
				1=F(b,b,0)\leq F(a,b,0) \leq F(a,b,\alpha)\leq	F(a,b,\arccos(-\tan^2a)),
			\end{align*}
			where the first inequality follows by a direct verification.
			Therefore, for any $\epsilon_3\in (0,1/4)$, there exists $\delta_3>0$ such that for $\alpha\in (0,\delta_3)$ and $|a-b|<\delta_3$,
			\begin{align*}
				|F(a,b,\alpha)-1|<\epsilon_3.
			\end{align*}
			Thus for $\alpha\in [\delta_3,\arccos(-\tan^2a)]$, and $a,b$ as in \eqref{eq:int1} with $|a-b|\geq\delta_3$,
			\begin{align}\label{eq:mv1}
				\arccos( g(a,b,\alpha)) \leq \arccos( (1+\epsilon_3)	h(a,b,\alpha)).
			\end{align}
			By the mean-value theorem, there exists 
			\begin{align}\label{eq:xi1}
				\xi(a,b,\alpha)\in\big(h(a,b,\alpha) , (1+\epsilon_3)	h(a,b,\alpha)\big)
			\end{align}
			such that 
			\begin{align}\label{eq:mv2}
				\frac{\arccos( (1+\epsilon_3)	h(a,b,\alpha))-\arccos( h(a,b,\alpha))}{(1+\epsilon_3)	h(a,b,\alpha)-	h(a,b,\alpha)}=\arccos'(\xi(a,b,\alpha)).
			\end{align}
			Combining \eqref{eq:mv1} and \eqref{eq:mv2}, we have
			\begin{align}
				\arccos( g(a,b,\alpha)) 
				\leq&\arccos( h(a,b,\alpha))\Big(1-\Big|\frac{\epsilon_3h(a,b,\alpha)\arccos'(\xi(a,b,\alpha))}{\arccos( h(a,b,\alpha))}\Big|\Big)\label{eq:mv3}\\
				=:&(1-|\epsilon_3\eta_1|)\arccos( h(a,b,\alpha)),\nonumber
			\end{align}
			where 
			\begin{align}\label{eq:eta11}
				\eta_1:=\frac{h(a,b,\alpha)\arccos'(\xi(a,b,\alpha))}{\arccos( h(a,b,\alpha))}.
			\end{align}
			
			We claim that for $\widetilde{\epsilon}_3\in (0,1/4)$ sufficiently small, there exists some $\widetilde{\delta}_3>0$ such that for
			$a,b$ as in \eqref{eq:int1} with $|a-b|\geq\delta_3$, we have
			\begin{align}\label{eq:eta12}
				\eta_1\geq \widetilde{\epsilon}_3.
			\end{align}
			In fact, by the definition of $h$ in \eqref{eq:F},
			$\arccos(-\tan^2a)$ is the unique $\alpha$ such that  $h(a,b,\alpha)=0$. Hence for $\alpha\in [\delta_3,  \arccos(-\tan^2a)]$,
			$h(a,b,\alpha)\geq0$, and thus
			$$\frac{h(a,b,\alpha)}{\arccos( h(a,b,\alpha))}\geq\frac{h(a,b, \arccos(-\tan^2a))}{\arccos( h(a,b,\arccos(-\tan^2a)))}=0.$$
			Thus by the continuity of $h$, for $\widetilde{\epsilon}_3\in (0,1/4)$ sufficiently small, there exists $\widetilde{\delta}_3>0$ such that for $\alpha\in [\delta_3,\arccos(-\tan^2a)-\widetilde{\delta}_3]$, and $a,b$ as in \eqref{eq:int1} with $|a-b|\geq\delta_3$,
			\begin{align}\label{eq:bou}
				\frac{h(a,b,\alpha)}{\arccos( h(a,b,\alpha))}\geq\widetilde{\epsilon}_3.
			\end{align}
			As $h(a,b,\alpha)$ is a decreasing function of $\alpha$ on 
			$[\delta_3,\arccos(-\tan^2a)-\widetilde{\delta}_3]$, for all such that $\alpha$ and all $a,b$ as in \eqref{eq:int1} with $|a-b|\geq\delta_3$, we have
			\begin{align}\label{eq:eta13}
				h(a,b,\alpha)\leq h(a,b,\delta_3)<1.
			\end{align}
			It follows from the continuity of $h$ and \eqref{eq:eta13} that
			$$\delta_3':=\min\big\{1-h(a,b,\delta_3):\,a,b\,\,\text{as\,\,in}\,\, \eqref{eq:int1},\,\,|a-b|\geq\delta_3\big\}<0.$$
			Now we properly adjust $\epsilon_3$ so that for all $\epsilon_3>0$ sufficiently small, we have
			\begin{align}\label{eq:eta14} 
				1-(1+\epsilon_3)h(a,b,\alpha))\leq \delta_3'/2.
			\end{align}
			It follows from \eqref{eq:xi1} and \eqref{eq:eta14} that $\xi(a,b,\alpha)<1$, and thus  $|\arccos'(\xi(a,b,\alpha))|> 1$. Combining \eqref{eq:eta11} and \eqref{eq:bou} yields 
			$\eta_1\geq\widetilde{\epsilon}_3$. Proving the claim.
			
			It follows by combining \eqref{eq:dis} and \eqref{eq:mv3} that
			for $\alpha\in [\delta_3,\arccos(-\tan^2a)-\widetilde{\delta}_3]$, and $a,b$ as in \eqref{eq:int1} with $|a-b|\geq\delta_3$,
			\begin{align}\label{eq:upp1}
				\frac{d_M(f(p),f(q))}{d_M(p,q)}=\frac{\arccos(g(a,b,\delta_3))}{\arccos(h(a,b,\delta_3))} \leq1-\epsilon_3\widetilde{\epsilon}_3 .
			\end{align}
			By direct calculation,
			\begin{align}\label{eq:upp2}
				\lim\limits_{(\alpha,a)\to(0,b)}\frac{d_M(f(p),f(q))}{d_M(p,q)}=\frac{1}{2}.
			\end{align}
			For any $\bar{\delta}_3\geq\delta_3$, by \eqref{eq:upp2} and the continuity of $d_M(f(p),f(q))/d_M(p,q)$, there exists $r_1\in(0,1)$ such that for $\alpha\in [0,\bar{\delta}_3]$, $a\in[0,\pi/4]$, and $b\in[a,a+\bar{\delta}_3]$,
			\begin{align}\label{eq:upp3}
				\frac{d_M(f(p),f(q))}{d_M(p,q)}\leq r_1.
			\end{align}
			By direct evaluation,
			\begin{align}\label{eq:upp4}
				\lim\limits_{\substack{\alpha\to \arccos(-\tan^2a) }}&\frac{d_M(f(p),f(q))}{d_M(p,q)}\nonumber\\
				=&\frac{2\arccos[\cos((b-a)/2)\sec^2a(\sin a+\cos(2a))/2]}{\pi}\nonumber\\
				\leq&\frac{2\arccos(\sqrt{2}/4)}{\pi}<1.
			\end{align}
			For any $\bar{\delta}_4\geq\widetilde{\delta}_3$, by \eqref{eq:upp4} and the continuity of $d_M(f(p),f(q))/d_M(p,q)$, there exists $r_2\in(0,1)$ such that for $\alpha\in [\arccos(-\tan^2a)-\bar{\delta}_4,\arccos(-\tan^2a)+\bar{\delta}_4]$, and $a,b$ as in \eqref{eq:int1},
			\begin{align}\label{eq:upp5}
				\frac{d_M(f(p),f(q))}{d_M(p,q)}\leq r_2.
			\end{align}
			Combining \eqref{eq:upp1}, \eqref{eq:upp3}, and \eqref{eq:upp5}, for $\alpha\in [0,\arccos(-\tan^2a)]$, $a,b$ as in \eqref{eq:int1},
			\begin{align}\label{eq:case1}
				\frac{d_M(f(p),f(q))}{d_M(p,q)}\leq r_3,
			\end{align}
			where $r_3:=\max\{1-\epsilon_3\widetilde{\epsilon}_3,r_1,r_2\}<1$.

			\noindent{\em Subcase 2. $\alpha\in (\arccos(-\tan^2a),\pi]$.} Similarly, we can show that for $\epsilon_3, \widetilde{\epsilon}_3\in (0,1/4)$, there exist  $\delta_3>0$  and $\widetilde{\delta}_3>0$ such that for $\alpha\in[\arccos(-\tan^2a)+\widetilde{\delta}_3,\pi-\delta_4])$, $a\in [\delta_4,\pi/4]$, $b\in [\delta_4,\pi/2]$, and $b\geq a$,
			\begin{align}\label{eq:upp6}
				\frac{d_M(f(p),f(q))}{d_M(p,q)} \leq1-\epsilon_4\widetilde{\epsilon}_3.
			\end{align}
			By direct evaluation,
			\begin{align}\label{eq:upp7}
				\lim\limits_{(\alpha,a,b)\to(\pi,0,0)}\frac{d_M(f(p),f(q))}{d_M(p,q)}=\frac{1}{2}.
			\end{align}
			For any $\delta_4'\geq\delta_4$, by \eqref{eq:upp7} and the continuity of $d_M(f(p),f(q))/d_M(p,q)$, there exists $r_4\in(0,1)$ such that for $\alpha\in [\pi-\bar{\delta}_4,\pi+\delta_4']$, $a\in[0, \delta_4']$, $b\in[0,\delta_4']$, and $b\geq a$,
			\begin{align}\label{eq:upp7*}
				\frac{d_M(f(p),f(q))}{d_M(p,q)}\leq r_4.
			\end{align}
			Combining \eqref{eq:upp5}, \eqref{eq:upp6}, and \eqref{eq:upp7*}, for $\alpha\in [\arccos(-\tan^2a),\pi]$, and $a,b$ as in \eqref{eq:int1},
			\begin{align}\label{eq:case2}
				\frac{d_M(f(p),f(q))}{d_M(p,q)}\leq r_5,
			\end{align}
			where $r_5:=\max\{1-\epsilon_4\widetilde{\epsilon}_3,r_2,r_5\}<1$.

			\noindent{\em Subcase 3. $\alpha\in (\pi,2\pi-\arccos(-\tan^2a))$.} Similarly, we can show that  for $\epsilon_4,\widetilde{\epsilon}_4\in (0,1/4)$, there exist  $\delta_4>0$  and $\widetilde{\delta}_4>0$ such that for $\alpha\in[\pi+\delta_4,2\pi-\arccos(-\tan^2a)-\widetilde{\delta}_4])$, $a\in [\delta_4,\pi/4]$, $b\in [\delta_4,\pi/2]$,  and $b\geq a$,
			\begin{align}\label{eq:eq3}
				\frac{d_M(f(p),f(q))}{d_M(p,q)} \leq 1-\epsilon_4\widetilde{\epsilon}_4.
			\end{align}
			By direct evaluation,
			\begin{align}\label{eq:upp8}
				\lim\limits_{\substack{\alpha\to 2\pi-\arccos(-\tan^2a) }}&\frac{d_M(f(p),f(q))}{d_M(p,q)}\nonumber\\
				=&\frac{2\arccos[\cos((b-a)/2)\sec^2(a)(\sin a+\cos(2a))/2]}{\pi}<1.
			\end{align}
			For any $\bar{\delta}_5\geq\widetilde{\delta}_4$, by \eqref{eq:upp8} and the continuity of $d_M(f(p),f(q))/d_M(p,q)$, there exists $r_6\in(0,1)$ such that for $\alpha\in [2\pi-\arccos(-\tan^2a)-\bar{\delta}_5,2\pi-\arccos(-\tan^2a)+\bar{\delta}_5]$, $a,b$ as in \eqref{eq:int1},
			\begin{align}\label{eq:upp9}
				\frac{d_M(f(p),f(q))}{d_M(p,q)}\leq r_6.
			\end{align}
			Combining \eqref{eq:upp7*}, \eqref{eq:eq3}, and \eqref{eq:upp9}, for $\alpha\in (\pi,2\pi-\arccos(-\tan^2a)]$ and $a,b$ as in \eqref{eq:int1},
			\begin{align}\label{eq:case3}
				\frac{d_M(f(p),f(q))}{d_M(p,q)}\leq r_7,
			\end{align}
			where $r_7:=\max\{1-\epsilon_4\widetilde{\epsilon}_7,r_4,r_6\}<1$.

			\noindent{\em Subcase 4. $\alpha\in [2\pi-\arccos(-\tan^2a),2\pi)$.} Similarly, we can show that  for $\epsilon_5>0$ sufficiently small and $\widetilde{\epsilon}_4\in (0,1/4)$, there exist  $\delta_5>0$  and $\widetilde{\delta}_4>0$ such that for $\alpha\in [2\pi-\arccos(-\tan^2a)+\widetilde{\delta}_4,2\pi-\delta_5]$ and $a,b$ as in \eqref{eq:int1} with $|a-b|\geq\delta_5$,
			\begin{align}\label{eq:eq4}
				\frac{d_M(f(p),f(q))}{d_M(p,q)} \leq 1-\epsilon_5\widetilde{\epsilon}_4.
			\end{align}
			By direct evaluation,
			\begin{align}\label{eq:upp10}
				\lim\limits_{(\alpha,a)\to (2\pi, b)}\frac{d_M(f(p),f(q))}{d_M(p,q)}=\frac{1}{2}.
			\end{align}
			For any $\bar{\delta}_5\geq\delta_5$, by \eqref{eq:upp10} and the continuity of $d_M(f(p),f(q))/d_M(p,q)$, there exists $r_8\in(0,1)$ such that for $\alpha\in [\bar{\delta}_5,2\pi)$, $a\in[0,\pi/4]$, and $b\in[a,a+\bar{\delta}_5]$,
			\begin{align}\label{eq:upp11}
				\frac{d_M(f(p),f(q))}{d_M(p,q)}\leq r_8.
			\end{align}
			Combining \eqref{eq:upp9}, \eqref{eq:eq4}, and \eqref{eq:upp11}, for $\alpha\in (2\pi-\arccos(-\tan^2a),2\pi)$, $a,b$ as in \eqref{eq:int1},
			\begin{align}\label{eq:case4}
				\frac{d_M(f(p),f(q))}{d_M(p,q)}\leq r_9,
			\end{align}
			where $r_9:=\max\{1-\epsilon_5\widetilde{\epsilon}_4,r_6,r_8\}<1$.

			\noindent{\em Case II. $a\in (\pi/4, \pi/2]$, $b\in (\pi/4,\pi/2]$, $b\geq a$ and $\alpha\in [0,2\pi)$.} 
		Similarly, we can show that for $\epsilon_6>0$ sufficiently small, $\widetilde{\epsilon}_6\in (0,1/4)$, there exist  $\delta_6>0$  and $\widetilde{\delta}_6>0$ such that for $\alpha\in [\delta_6,\pi-\widetilde{\delta}_6]$,  $a,b\in [\pi/4+\widetilde{\delta}_6,\pi/2]$, and $b-a\geq\delta_6$,
			\begin{align}\label{eq:upp12}
				\frac{d_M(f(p),f(q))}{d_M(p,q)} \leq(1-\epsilon_6\widetilde{\epsilon}_6).
			\end{align}
			By direct evaluation,
			\begin{align}\label{eq:upp13}
				\lim\limits_{(\alpha,a)\to (0,b)}	\frac{d_M(f(p),f(q))}{d_M(p,q)}=\frac{1}{2}.
			\end{align}
			For any $\bar{\delta}_6\geq\delta_6$, by \eqref{eq:upp13} and the continuity of $d_M(f(p),f(q))/d_M(p,q)$, there exists $r_{10}\in(0,1)$ such that for $\alpha\in [0,\bar{\delta}_6]$, $a\in(\pi/4,\pi/2]$, and $b\in[a,a+\bar{\delta}_6]$,
			\begin{align}\label{eq:upp14}
				\frac{d_M(f(p),f(q))}{d_M(p,q)}\leq r_{10}.
			\end{align}
			By direct evaluation,
			\begin{align}\label{eq:upp14*}
				\lim\limits_{(\alpha,a,b)\to (\pi,\pi/4,\pi/4)}	\frac{d_M(f(p),f(q))}{d_M(p,q)}=\frac{1}{2}.
			\end{align}
			For any $\bar{\delta}_7\geq\widetilde{\delta}_6$, by \eqref{eq:upp14*} and the continuity of $d_M(f(p),f(q))/d_M(p,q)$, there exists $r_{11}\in(0,1)$ such that for $\alpha\in [\pi-\bar{\delta}_7,\pi+\bar{\delta}_7]$, $a,b\in(\pi/4,\pi/4+\bar{\delta}_7]$, and $b\geq a$,
			\begin{align}\label{eq:upp15}
				\frac{d_M(f(p),f(q))}{d_M(p,q)}\leq r_{11}.
			\end{align}
			Combining \eqref{eq:upp12}, \eqref{eq:upp14}, and \eqref{eq:upp15}, for $\alpha\in [0,\pi]$, $a\in(\pi/2]$, $b\in(\pi/4,\pi/2]$, and $b\geq a$,
			\begin{align}\label{eq:case5}
				\frac{d_M(f(p),f(q))}{d_M(p,q)}\leq r_{12},
			\end{align}
			where $r_{12}:=\max\{1-\epsilon_6\widetilde{\epsilon}_6,r_{10},r_{11}\}<1$.

			Similarly, for $\epsilon_7>0$ sufficiently small, $\widetilde{\epsilon}_6\in (0,1/4)$, there exist  $\delta_7>0$  and $\widetilde{\delta}_6>0$ such that for $\alpha\in [\pi+\widetilde{\delta}_6,2\pi-\delta_7]$,  $a,b\in [\pi/4+\widetilde{\delta}_6,\pi/2]$, and $|a-b|\geq\delta_7$,
			\begin{align}\label{eq:upp16}
				\frac{d_M(f(p),f(q))}{d_M(p,q)} \leq(1-\epsilon_7\widetilde{\epsilon}_6).
			\end{align}
			By direct evaluation,
			\begin{align}\label{eq:upp17}
				\lim\limits_{(\alpha,a)\to(2\pi, b)}	\frac{d_M(f(p),f(q))}{d_M(p,q)}=\frac{1}{2}.
			\end{align}
			For any $\bar{\delta}_7\geq\delta_7$, by \eqref{eq:upp17} and the continuity of $d_M(f(p),f(q))/d_M(p,q)$, there exists $r_{10}\in(0,1)$ such that for $\alpha\in [2\pi-\bar{\delta}_7,2\pi)$, $a\in(\pi/4,\pi/2]$, and $b\in[a,a+\bar{\delta}_7]$,
			\begin{align}\label{eq:upp18}
				\frac{d_M(f(p),f(q))}{d_M(p,q)}\leq r_{13}.
			\end{align}
			Combining \eqref{eq:upp15}, \eqref{eq:upp16}, and \eqref{eq:upp18}, for $\alpha\in (\pi,2\pi)$, $a\in(\pi/4,\pi/2]$, $b\in(\pi/4,\pi/2]$, and $b\geq a$,
			\begin{align}\label{eq:case6}
				\frac{d_M(f(p),f(q))}{d_M(p,q)}\leq r_{14},
			\end{align}
			where $r_{14}:=\max\{1-\epsilon_7\widetilde{\epsilon}_6,r_{11},r_{13}\}<1$.

			\noindent{\em Case III. $a=0$ and $b=\pi/2$}. In this case, by direct evaluation, we have
			\begin{align}\label{eq:case6}
				\frac{d_M(f(p),f(q))}{d_M(p,q)}=\frac{1}{2}.
			\end{align}
			
			Combining \eqref{eq:case1}, \eqref{eq:case2}, \eqref{eq:case3}, \eqref{eq:case4}, \eqref{eq:case5}, and \eqref{eq:case6}, for $\alpha\in [0,2\pi)$, $a\in [0,\pi/2]$, $b\in [0,\pi/2]$, and $b\geq a$, we have
			\begin{align}\label{eq:up7}
				\frac{d_M(f(p),f(q))}{d_M(p,q)}\leq r,
			\end{align}
			where $r:=\max\{1/2,r_{3},r_{5},r_{7},r_{9},r_{12},r_{14}\}$. 
			
			Combining \eqref{eq:low3} and \eqref{eq:up7} proves that  $f$ is a bi-Lipschitz contraction on $\mathbb{S}^2_+$.
	\end{proof}

	As a consequence of Lemma \ref{lem:4.1}, the closure of a bounded open set of a complete smooth Riemannian 2-manifold $M$, the above measure $\mu$ satisfies $\underline{\operatorname{dim}}_{\infty}(\mu)>0$. Next,  we construct an IFS on $\mathbb{S}^2_+$ that satisfies the conditions of Lemma \ref{lem:4.1}.
	
	Rotations about the $x$-axis and the $y$-axis through an angle $\alpha$ are given respectively by the matrices
	\[R_x(\alpha):=
	\begin{pmatrix}
		1 & 0 & 0 \\
		0 & \cos\alpha &-\sin\alpha \\
		0 & \sin\alpha &\cos \alpha
	\end{pmatrix}
	\qquad \text{and}\qquad R_y(\alpha):=
	\begin{pmatrix}
		\cos\alpha& 0 & \sin\alpha \\
		0 & 1 &0 \\
		-\sin\alpha & 0 &\cos \alpha
	\end{pmatrix}.
	\]

	\begin{figure}[H]
		\centering
		\mbox{\subfigure[]
			{\includegraphics[scale=0.28]{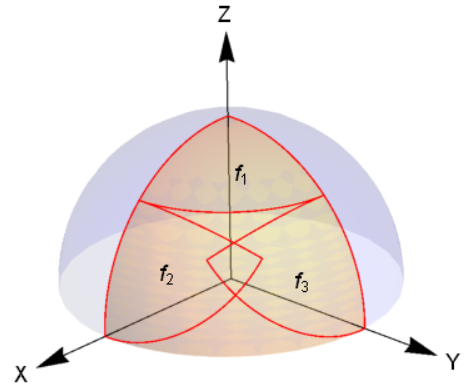}}
		}\qquad
		\mbox{\subfigure[]
			{\includegraphics[scale=0.27]{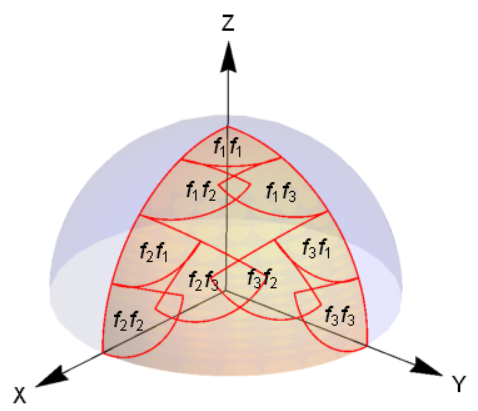}}
		}\,\,
		\mbox{\subfigure[]
			{	\includegraphics[scale=0.36]{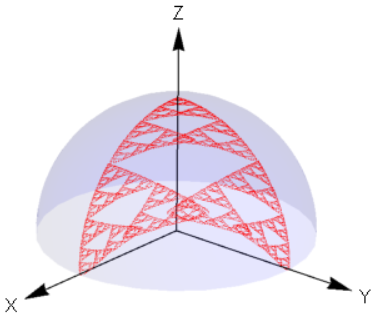}}
		}
		\vspace{0.5cm}
		\caption{Iterations and attractor of the IFS $\{f_i\}_{i=1}^3$ on $\mathbb{S}^2_+$ described in  Example \ref{exam1}. (a) First iteration.  (b) Second iteration.  (c) Attractor.}
		\label{fig.1}
	\end{figure}
	\begin{exam}\label{exam1}
		Let $f$, $\mathbb{S}^2$, and $\mathbb{S}^2_+$ be defined as in Proposition \ref{prop:6.1}. Let $$D:=\{(\sin\varphi\cos\theta,\sin\varphi\sin\theta,\cos\varphi):0\leq \theta\leq\pi/2,\,\, 0\leq \varphi\leq \pi/2\}$$ be the subset parametrizing $\mathbb{S}_+^2$. Let
	\begin{align*}
				f_1:=f|_{D},\qquad f_2:=R_y(\pi/4)\circ f_1, \qquad f_3:=R_x(-\pi/4)\circ f_1
		\end{align*}
		be a family bi-Lipschitz contractions on $D$ (see Figure \ref{fig.1}(a)). Let $\mu$ be an invariant measure of the IFS $\left\{f_{i}\right\}_{i=1}^{3}$. Then  $\underline{\operatorname{dim}}_{\infty}(\mu)>0$. It follows that the conclusions of Theorems \ref{thm:wmain1}, \ref{thm:hmain1}, and \ref{thm:smain1} hold for $\mu$, and the conclusions of Theorems \ref{thm:wmain2}, \ref{thm:hmain2}, and \ref{thm:smain2} hold for all such $\mu$ and all $F\in {\rm Lip}(\operatorname{dom}\mathcal{E})$.
	\end{exam}

	\begin{proof}
		We need to show that for any $p\in D$ with parameter $(\varphi,\theta)$, $f_i(p)\in D$, $i=1,2,3$. It is obvious that $f_1(p)\in D$. In Euclidean coordinates, let
	\begin{align*}
			h(\theta):=\frac{\sqrt{2}}{2}(\cos\theta,\sin\theta,1).
		\end{align*}
	Then
		\[R_y(\pi/4)\circ h(\theta)=
		\begin{pmatrix}
			\vspace{0.15cm}
			\frac{1}{2}\cos\theta +\frac{1}{2}\\
			\vspace{0.15cm}
			\frac{\sqrt{2}}{2}\sin\theta \\
			-\frac{1}{2}\cos\theta +\frac{1}{2}
		\end{pmatrix}\quad\text{and}\quad
		R_x(-\pi/4)\circ h(\theta)=
		\begin{pmatrix}	
			\vspace{0.15cm}
			\frac{\sqrt{2}}{2}\cos\theta\\
			\vspace{0.15cm}
			\frac{1}{2}\sin\theta +\frac{1}{2}\\
			-\frac{1}{2}\sin\theta +\frac{1}{2}
		\end{pmatrix}.
		\]
		For $\theta\in[0,\pi/2]$, $-\sin\theta/2 +1/2\geq 0$ and $-\cos\theta/2 +1/2\geq 0$. Hence $f_2(p), f_3(p)\in D$.  By Lemma \ref{lem:4.1} and Proposition \ref{prop:6.1}, we have $\underline{\operatorname{dim}}_{\infty}(\mu)>0$. Furthermore, we see that the measures $\mu$ in this example satisfies the assumptions on $\mu$ in Theorems \ref{thm:wmain1}, \ref{thm:wmain1}, \ref{thm:hmain1}, \ref{thm:hmain2}, \ref{thm:smain1}, and \ref{thm:smain2}. The asserted results follow.
	\end{proof}

	\section{Examples of GIFS} \label{S:exam2}
	\setcounter{equation}{0}
	
	Let $G = (V, E)$ be a {\em graph}, where $V = \{1,\ldots, t\}$ is the set of vertices and $E$ is
	the set of all directed edges. It is possible that the initial and terminal vertices are the same. We allow more than one edge between two vertices. A {\em directed
		path} in $G$ is a finite string $\mathbf{e} = e_1\cdots e_k$ of edges in $E$ such that the terminal vertex of
	each $e_i$ is the initial vertex of the edge $e_{i+1}$. For such a path, denote the length of $\mathbf{e}$ by
	$|v|=k$. For any two vertices $i,j\in V$ and any positive integer $k$, let $E^{i,j}$ be the set
	of all directed edges from $i$ to $j$, $E^{i,j}_k$ be the set of all directed paths of length $p$ from
	$i$ to $j$, $E_k$ be the set of all directed paths of length $k$, and $E^{*}$ be the set of all directed
	paths, i.e.,
	\begin{align*}
		E_k:=\bigcup_{i,j=1}^kE^{i,j}_k\qquad \text{and}\qquad E^{*}:=\bigcup_{k=1}^\infty E_k.
	\end{align*}
	Recall that a directed graph $G = (V, E)$ is said to be {\em strongly connected} provided that whenever $i,j$, there exists a directed path from $i$ to $j$.

	Let $M$ be a smooth oriented complete Riemannian $n$-manifold.
	A {\em graph-directed iterated
		function system} (GIFS) on $M$ consists of
	\begin{enumerate}
		\item [(1)] a finite collection of compact subsets of $M$: $W_1,\ldots,W_t$ such 
		that each $W_i$ has a nonempty interior;
		\item [(2)] a directed graph $G$ with vertex set consisting of the integers $1,\ldots, t$ and 
		contractions $S_{e}:W_j\to W_i$ with constant of contraction $\rho_e$, where $e\in E^{i,j}$ and $i,j\in V$.
	\end{enumerate}
	Then there exists a unique collection of nonempty compact sets $\{K_i\}_{i\in V}$ satisfying
	\begin{align*}
		K_i=\bigcup_{j=1}^t\bigcup_{e\in E^{i,j}} S_e(K_j), \qquad i\in V
	\end{align*}
	and a unique collection of Borel probability measures $\{\mu_i\}_{i\in V}$ satisfying
	\begin{align}\label{eq:mea}
		\mu_i=\sum_{j=1}^t\sum_{e\in E^{i,j}}p_e\mu_j\circ S_e^{-1}, \qquad i\in V,
	\end{align}
	where $\{p_e\}_{e\in E}$ are probability weights such that $p_e \in (0, 1)$ and
	\begin{align}\label{eq:pro}
		\sum_{j=1}^t\sum_{e\in E^{i,j}}p_e=1, \qquad i\in V
	\end{align} 
	(see, e.g., \cite{Mauldin-Williams_1988,Edgar_1990,Olsen_1994}). 
	Define 
	\begin{align*}
		K:=\bigcup_{i=1}^tK_i \qquad\text{and} \qquad	\mu:=\bigcup_{i=1}^t\mu_i.
	\end{align*}
	$K$ is called the {\em graph self-similar set} and $\mu$ the {\em  graph self-similar measure}.
	
	Let $S_\mathbf{e}:=S_{e_{1}} \circ \cdots \circ S_{e_{k}}$. For the invariant set $K$, we let $K_{\mathbf{e}}:=S_{\mathbf{e}}(K)$.
	We have the following lemma.
	
	\begin{lem}\label{lem:7.1}
		Let $M$ be a smooth oriented complete Riemannian n-manifold and $W \subseteq M$ be a compact set. Let $\mu$ be an invariant measure of a GIFS $\left\{S_{e}\right\}_{e\in E}$ of bi-Lipschitz contractions on $W$. Suppose the attractor $K$ is not a singleton. Then $\mu$ is upper $s$-regular for some $s>0$, and hence $\underline{\operatorname{dim}}_{\infty}(\mu)>0$.
	\end{lem}
	\begin{proof}\, The proof of this theorem is similar to that of \cite[Lemma 5.1]{Hu-Lau-Ngai_2006}; we only describe the modifications.
		Let $c_{e}, r_{e}$, where $e\in E$, be given as in \eqref{eq:bi} and let $\{p_e\}_{e\in E}$ be given as in \eqref{eq:pro}. Since $K$ is not a singleton, there are indices $\mathbf{e}_{1}, \mathbf{e}_{2}$ of the same length such that $$\Big(\bigcup_{\mathbf{e}_1\in E^{i,j_1}}K_{\mathbf{e}_1}\Big) \bigcap \Big(\bigcup_{\mathbf{e}_2\in E^{i,j_{2}}}K_{\mathbf{e}_2}\Big) =\emptyset .$$
		Hence, without loss of generality, we assume that $(\cup_{e\in E^{i,1}}K_{e} ) \cap (\cup_{e\in E^{i,2}}K_{e}) =\emptyset$. There exists $r_{0}>0$ such that for any $x \in W$, the ball $B^M(x,r_0)$ intersects at most one of the two sets $\cup_{e\in E^{i,1}}K_{e}$ and  $\cup_{e\in E^{i,2}}K_{e}$.
		Let 
		$$ p:=\min_{1\leq i\leq t}\min \Big\{\sum_{e\in E^{i,1}}p_e,\sum_{e\in E^{i,2}}p_e \Big\}<1\quad\text{and}\quad c:=\min _{e\in E}\left\{c_e\right\}. $$  
			Let
			$$ \phi(r):=\max _{1\leq i\leq t} \sup _{x \in M} \mu_i\left(B^M(x,r)\right) \quad(r \geq 0).
			$$
		For $x \in W$ and $0<r \leq r_{0}$, either $B^M(x,r) \cap (\cup_{e\in E^{i,1}}K_{e}) =\emptyset$ or $B^M(x,r) \cap (\cup_{e\in E^{i,2}}K_{e}) =\emptyset.$ We only consider the former case; the latter case can be treated similarly. In this case,
		Therefore 
		\begin{align}\label{eq:si}
			S_{e}^{-1}(B^M(x,r)) \subseteq B^M(S_{e}^{-1}(x), r/c).
		\end{align}
		Hence
		\begin{align*}
			\mu(B^M(x,r))&=\sum_{i=1}^t\sum_{j\neq 2}\sum_{e\notin E^{i,2}}p_e\mu_j\Big(S_{e}^{-1}\big(B^M(x,r)\big)\Big)\qquad(\text{by}\,\,\eqref{eq:mea})\\
			&\leq\sum_{i=1}^t\sum_{j\neq 2}\sum_{e\notin E^{i,2}}p_e\mu_j(B^M(S_{e}^{-1}(x), r/c))\qquad(\text{by}\,\,\eqref{eq:si})\\
				&\leq\sum_{i=1}^t\Big(1-\sum_{e\in E^{i,2}}p_e\Big)\phi(r/c)\\
			&\leq t(1-p)\phi(r/c).
		\end{align*}
		It follows that for $r\in(0, r_{0}]$, $\phi(r) \leq t(1-p) \phi(r/c)$. 
		Therefore, for any $n \geqslant 0$ and any $0<r \leq r_{0}$,
		$
		\phi\left(c^{n} r\right) \leq t(1-p) \phi\left(c^{n-1} r\right) \leq \cdots \leq t^n(1-p)^{n} \phi(r) .
		$
		This implies that
		$$
		\mu\left(B^M(x,r_{0} c^{n})\right) \leq C\left(r_{0} c^{n}\right)^{s},
		$$
		where $s=\ln (t-tp) / \ln c$ and $C=\exp \left(-\ln (t-tp) \ln r_{0} / \ln c\right)$. 
		Hence $\mu$ is upper $s$ -regular, and thus $\underline{\operatorname{dim}}_{\infty}(\mu)>0$.
	\end{proof}
	

	Let $\mathbb{T}^2 := \mathbb{S}^1\times \mathbb{S}^1$ be a 2-torus, viewed as
	$[0, 1]\times[0, 1]$ with opposite sides identified, and $\mathbb{T}^2$ be endowed with the Riemannian metric induced from $\R^2$. We consider the following IFS with overlaps on $\R^2$: 
	\begin{align*}
		\tilde{h}_1(\boldsymbol{x})=\frac{1}{2}\boldsymbol{x}+\Big(0,\frac{1}{4}\Big),\qquad \tilde{h}_2(\boldsymbol{x})=\frac{1}{2}\boldsymbol{x}+\Big(\frac{1}{4},\frac{1}{4}\Big),\\
		\tilde{h}_3(\boldsymbol{x})=\frac{1}{2}\boldsymbol{x}+\Big(\frac{1}{2},\frac{1}{4}\Big),\qquad \tilde{h}_4(\boldsymbol{x})=\frac{1}{2}\boldsymbol{x}+\Big(\frac{1}{4},\frac{3}{4}\Big).
	\end{align*}
	Iterations of $\{\tilde{h}_m\}_{m=1}^4$ induce iterations $\{h_m\}_{m=1}^4$ on $\mathbb{T}^2=\R^2\backslash \Z^2$, defined 
	$$h_m(\boldsymbol{x}):=\tilde{h}_m(\boldsymbol{x})(\text{\rm mod}\,\Z^2),\quad \boldsymbol{x}\in \Omega_0:=[0, 1)\times[0, 1)$$ (see \cite{Ngai-Xu_2023}).  Let $W:=\cup_{m=1}^4\overline{h_m(\Omega_0)}$. Iterations $\{h_m\}_{m=1}^4$ on $W$ generates a compact set $K\subseteq \mathbb{T}^2$ defined as
	$$K:=\bigcap_{l=1}^\infty\bigcup_{m\in\{1,\ldots,4\}^l} h_m(W).$$
	We call $K$ the {\em attractor} of $\{h_m\}_{m=1}^4$ on $\mathbb{T}^2$.

	Next we consider the GIFS $G=(V,E)$ with $V=\{1,\ldots,12\}$, $E=\{e_1,\ldots,e_{48}\}$, and invariant family $\{W_i\}_{i=1}^{12}$, where 
	\begin{align*}
		&W_1=\Big[\frac{1}{4},\frac{1}{2}\Big]\times \Big[\frac{3}{4},1\Big],\,\,\,\,&& W_2=\Big[\frac{1}{2},\frac{3}{4}\Big]\times \Big[\frac{3}{4},1\Big],&&&W_3=\Big[0,\frac{1}{4}\Big]\times \Big[\frac{1}{2},\frac{3}{4}\Big],\\
		&W_4=\Big[\frac{1}{4},\frac{1}{2}\Big]\times \Big[\frac{1}{2},\frac{3}{4}\Big], \,\,\,\,&&W_5=\Big[\frac{1}{2},\frac{3}{4}\Big]\times \Big[\frac{1}{2},\frac{3}{4}\Big],&&& W_6=\Big[\frac{3}{4},1\Big]\times \Big[\frac{1}{2},\frac{3}{4}\Big],\\
		&W_7=\Big[0,\frac{1}{4}\Big]\times \Big[\frac{1}{4},\frac{1}{2}\Big],\,\,\,\,&& W_8=\Big[\frac{1}{4},\frac{1}{2}\Big]\times \Big[\frac{1}{4},\frac{1}{2}\Big],&&&W_9=\Big[\frac{1}{2},\frac{3}{4}\Big]\times \Big[\frac{1}{4},\frac{1}{2}\Big],\\
		&W_{10}=\Big[\frac{3}{4},1\Big]\times \Big[\frac{1}{4},\frac{1}{2}\Big],\,\,\,\,&&W_{11}=\Big[\frac{1}{4},\frac{1}{2}\Big]\times \Big[0,\frac{1}{4}\Big],&&&W_{12}=\Big[\frac{1}{2},\frac{3}{4}\Big]\times \Big[0,\frac{1}{4}\Big]
	\end{align*}
(see Figure \ref{fig:2} (b)) and
	\begin{table}[h]
		\centering
		\small
		\renewcommand\arraystretch{1.3}
		\renewcommand\tabcolsep{5.0pt}
		\begin{tabular}{|l|l|l|l|l|l|}	\hline
			$e_1\in E^{1,7}$& $e_2\in E^{1,8}$& $e_3\in E^{1,11}$& $ e_4\in E^{2,9}$& $ e_5\in E^{2,10}$& $e_6\in E^{2,12}$\\                  
			$e_7\in E^{3,1}$& $e_8\in E^{3,3}$& $e_9\in E^{3,4}$& $e_{10}\in E^{4,1}$& $e_{11}\in E^{4,2}$& $e_{12}\in E^{4,3}$\\
			$e_{13}\in E^{4,4}$& $e_{14}\in E^{4,5}$& $
			e_{15}\in E^{4,6}$& $e_{16}\in E^{5,1}$& $e_{17}\in E^{5,2}\quad $& $e_{18}\in E^{5,3}$\\ 
			$e_{19}\in E^{5,4}$& $e_{20}\in E^{5,5}$& $ e_{21}\in E^{5,6}$& $
			e_{22}\in E^{6,2}$& $ e_{23}\in E^{6,5}$& $ e_{24}\in E^{6,6}$\\
			$e_{25}\in E^{7,7}$& $e_{26}\in E^{7,8}$& $e_{27}\in E^{7,11}$& $e_{28}\in E^{8,7}$& $
			e_{29}\in E^{8,8}$& $e_{30}\in E^{8,9}$\\
			$e_{31}\in E^{8,10}$& $e_{32}\in E^{8,11}$& $e_{33}\in E^{8,12}$& $ e_{34}\in E^{9,7}$& $ e_{35}\in E^{9,8}$& $
			e_{36}\in E^{9,9}$\\
			$e_{37}\in E^{9,10}$& $e_{38}\in E^{9,11}$& $e_{39}\in E^{9,12}$& $ e_{40}\in E^{10,9}$& $e_{41}\in E^{10,10}$& $ e_{42}\in E^{10,12}$\\
			$e_{43}\in E^{11,1}$& $ e_{44}\in E^{11,3}$& $e_{45}\in E^{11,4}$& $ e_{46}\in E^{12,2}$& $ e_{47}\in E^{12,5}$& $e_{48}\in E^{12,6}$     \\ \hline                           
		\end{tabular}
		\vspace{5pt}
		\caption{All the edges in Example \ref{exam2}.}
		\label{T4}
	\end{table}\\
Note that $W=\cup_{i=1}^{12}W_i$. The associated similitudes $S_e$, $e \in E$, are defined as $S_{e_i}(\boldsymbol{x}) := \boldsymbol{x}/2 + d_i$, where $\boldsymbol{x}\in W$ and
	\begin{table}[h]
		\centering
		\small
		\renewcommand\arraystretch{1.3}
		\renewcommand\tabcolsep{5.0pt}
		\begin{tabular}{|c|c|c|c|c|c|c|c|c|c|}\hline
			$i$&$d_i$              &$i$& $d_i$                 &$i$&$d_i$                 &$i$&$d_i$   &$i$&$d_i$                   \\ \hline
			
			1  &$\big(\frac{1}{4},-\frac{3}{8}\big)$     &11 & $\big(-\frac{1}{4},-\frac{1}{8}\big)$               &21&$\big(-\frac{1}{8},0\big)$   &31 &$\big(-\frac{3}{8},\frac{1}{8}\big)$    &41&$\big(\frac{1}{8},\frac{1}{8}\big)$  \\\hline
			
			2  &$\big(\frac{1}{8},-\frac{3}{8}\big)$      &12 & $\big(\frac{1}{4},0\big)$     & 22&$\big(\frac{1}{4},-\frac{1}{8}\big)$  &32 &$\big(\frac{1}{8},\frac{1}{4}\big)$      &42 &$\big(\frac{1}{4},\frac{1}{4}\big)$   \\ \hline
			
			3  &$\big(\frac{1}{8},-\frac{1}{4}\big)$      &13& $\big(\frac{1}{8},0\big)$     & 23&$\big(\frac{1}{4},0\big)$  &33 &$\big(-\frac{1}{4},\frac{1}{4}\big)$   &43 &$\big(\frac{1}{8},-\frac{3}{8}\big)$       \\ \hline
			
			4  &$\big(0,-\frac{3}{8}\big)$                &14& $\big(-\frac{1}{4},0\big)$     & 24&$\big(\frac{1}{8},0\big)$  &34 &$\big(\frac{1}{2},\frac{1}{8}\big)$     &44 &$\big(\frac{1}{4},-\frac{1}{2}\big)$      \\ \hline
			
			5  &$\big(-\frac{1}{8},-\frac{3}{8}\big)$     &15& $\big(-\frac{3}{8},0\big)$    & 25&$\big(0,\frac{1}{8}\big)$   &35 &$\big(\frac{3}{8},\frac{1}{8}\big)$  &45 &$\big(\frac{1}{8},-\frac{1}{2}\big)$        \\ \hline
			
			6  &$\big(0,-\frac{1}{4}\big)$                &16& $\big(\frac{3}{8},-\frac{1}{8}\big)$    & 26&$\big(-\frac{1}{8},\frac{1}{8}\big)$   &   36&$\big(0,\frac{1}{8}\big)$  &  46&$\big(0,-\frac{3}{8}\big)$                           \\ \hline
			
			7  &$\big(-\frac{1}{8},-\frac{1}{8}\big)$     &17& $\big(0,-\frac{1}{8}\big)$   & 27&$\big(-\frac{1}{8},\frac{1}{4}\big)$   &  37&$\big(-\frac{1}{8},\frac{1}{8}\big)$ & 47&$\big(0,-\frac{1}{2}\big)$   \\ \hline   
			
			8  &$\big(0,0\big)$                           &18& $\big(\frac{1}{2},0\big)$   & 28&$\big(\frac{1}{4},\frac{1}{8}\big)$   &   38&$\big(\frac{3}{8},\frac{1}{4}\big)$ &    48&$\big(-\frac{1}{8},-\frac{1}{2}\big)$                           \\ \hline
			
			9  &$\big(-\frac{1}{8},0\big)$                &19& $\big(\frac{3}{8},0\big)$   & 29 &$\big(\frac{1}{8},\frac{1}{8}\big)$  &39 &$\big(0,-\frac{1}{4}\big)$  &   &                              \\ \hline
			
			10  &$\big(\frac{1}{8},-\frac{1}{8}\big)$     &20& $\big(0,0\big)$   & 30&$\big(-\frac{1}{4},\frac{1}{8}\big)$   &40&$\big(\frac{1}{4},\frac{1}{8}\big)$  &   &                              \\ \hline                        
		\end{tabular}
		\vspace{10pt}
		\caption{The translations $d_i$ for the GIFS in Example \ref{exam2}.}
		\label{di}
	\end{table}\\
	We notice that \\
	\begin{align*}
		&h_1|_{W_3}=(S_{e_7}+S_{e_8}+S_{e_9})|_{W_3},\qquad\qquad\qquad &&h_1|_{W_4}=(S_{e_{11}}+S_{e_{12}}+S_{e_{13}}+S_{e_{14}}+S_{e_{15}})|_{W_4},\\
		&h_1|_{W_7}=(S_{e_{25}}+S_{e_{26}}+S_{e_{27}})|_{W_7},\quad\qquad\qquad &&h_1|_{W_8}=(S_{e_{28}}+S_{e_{29}}+S_{e_{30}}+S_{e_{31}}+S_{e_{33}})|_{W_8},\\
		&h_2|_{W_4}=(S_{e_{10}}+S_{e_{12}}+S_{e_{13}}+S_{e_{14}}+S_{e_{15}})|_{W_4}, &&h_2|_{W_5}=(S_{e_{17}}+S_{e_{18}}+S_{e_{19}}+S_{e_{20}}+S_{e_{21}})|_{W_5},\\
		&h_2|_{W_8}=(S_{e_{28}}+S_{e_{29}}+S_{e_{30}}+S_{e_{31}}+S_{e_{32}})|_{W_8}, &&h_2|_{W_9}=(S_{e_{34}}+S_{e_{35}}+S_{e_{36}}+S_{e_{37}}+S_{e_{39}})|_{W_9},\\
		&h_3|_{W_5}=(S_{e_{16}}+S_{e_{18}}+S_{e_{19}}+S_{e_{20}}+S_{e_{21}})|_{W_5}, &&h_3|_{W_6}=(S_{e_{22}}+S_{e_{23}}+S_{e_{24}})|_{W_6},\\
		&h_3|_{W_9}=(S_{e_{34}}+S_{e_{35}}+S_{e_{36}}+S_{e_{37}}+S_{e_{38}})|_{W_9},
		&&h_3|_{W_{10}}=(S_{e_{40}}+S_{e_{41}}+S_{e_{42}})|_{W_{10}},\\
		&h_4|_{W_1}=(S_{e_{1}}+S_{e_{2}}+S_{e_{3}})|_{W_1},\quad\qquad\qquad\qquad &&h_4|_{W_2}=(S_{e_{4}}+S_{e_{5}}+S_{e_{6}})|_{W_2},\\
		&h_4|_{W_{11}}=(S_{e_{43}}+S_{e_{44}}+S_{e_{45}})|_{W_{11}},\qquad\qquad &&h_4|_{W_{12}}=(S_{e_{46}}+S_{e_{47}}+S_{e_{48}})|_{W_{12}}.\\
	\end{align*}
	Hence 
	\begin{align*}
		\bigcup_{m=1}^4h_m(W)=\bigcup_{m=1}^4\bigcup_{i=1}^{12}h_m(W_i)=\bigcup_{i,j=1}^{12}\bigcup_{e\in E^{i,j}}S_e(W_j).	
	\end{align*}
It follows by  induction that for $l\geq 1$,
		\begin{align*}
			\bigcup_{m\in\{1,\ldots,4\}^l}h_m(W)=\bigcup_{i,j=1}^{12}\bigcup_{\mathbf{e}\in E^{i,j}_l}S_\mathbf{e}(W_j).	
	\end{align*}
	Thus  $$K=\bigcap_{l=1}^\infty\bigcup_{i,j=1}^{12}\bigcup_{\mathbf{e}\in E^{i,j}_{l}}S_\mathbf{e}(W_j).$$
	Hence $K$ is the graph self-similar set generated by the GIFS $G = (V, E)$
	associated to $\{S_e\}_{e\in E}$.
	
\begin{rema}\label{rem:tu}
	For any $x,y\in W_i$ for some $i\in V$, there exists some $j\in V$, such that $S_e(x), S_e(y)\in W_j$.
\end{rema}
	
	\begin{prop}\label{prop:1}
		Let $G = (V, E)$ be defined as above.  Then $G$ is strongly connected.
	\end{prop}
	\begin{proof}
		We write $i\to j$ if there is a directed edge from $i$ to $j$. By the definition of $E$, for all $i,j\in V$, we have the following edges:
		\begin{align*}
			&1\to 7,8,11;\qquad \qquad\,\,2\to 9,10,12;\quad\, 3\to 1,3,4;\quad\,\,\, 4\to 1,2,3,4,5,6;\\
			&5\to 1,2,3,4,5,6;\qquad6\to 2,5,6;\qquad\,\, 7\to 7,8,11;\quad 8\to 7,8,9,10,11,12;\\
			&9\to 7,8,9,10,11,12;\,\, 10\to 9,10,12;\quad 11\to 1,3,4;\quad 12\to 2,5,6.
		\end{align*}
		Therefore, the following path passes through all the vertices:
		\begin{align*}
			1\to 8\to 12\to 2\to 9\to 11\to 3\to 4\to 5\to 6\to 2\to 9\to 7\to 8\to 10\to 12 \to 5\to 1.
		\end{align*}
		Hence $G$ is strongly connected.
	\end{proof}
	\begin{lem}\label{lem:dis1}
		Let $G$, $\mathbb{T}^2$, and $\{W_i\}_{i=1}^{12}$ be defined as above. Fix $i\in V$ and let $x,y\in W_i$. Then
		$$d_{\mathbb{T}^2}(x,y)=d_{\mathbb{E}}(x,y).$$
	\end{lem}
	\begin{proof}
		We only show that for any $x,y \in W_1$,
		$$d_{\mathbb{T}^2}(x,y)=d_{\mathbb{E}}(x,y);$$
	the proofs for  $x,y \in W_i$, $i=2,\ldots,8$ are the same. Let the four boundaries of $W_1$ be $l_1,l_2,l_3,l_4$, respectively (see Figure \ref{fig:2}(b)). Let $x,y \in W_1$, where $x$ is above $y$, or at the same level. Then 
		$d_{\mathbb{E}}(x,y)\leq\sqrt{2}/4.$
		We prove that
		$$d_M(x,y)<d_{\mathbb{E}}(x,y)$$
		by considering the following eight cases:
		
		\noindent{\em Case 1.  $x$ goes up to $l_1$, goes through $W_{11},W_{8},W_{4}$, and then back to the interior of $W_{1}$.} 
		
		\noindent{\em Case 2.  $x$ goes up to $l_1$, then moves along $l_1$ by a distance $\ge 0$, and finally back to the interior of $W_{1}$}.
		
		\noindent{\em Case 3. $x$ turns right and goes to $l_4$, goes through  $W_{2}$ all the way to the gluing edge, and goes from  $l_2$ into the interior of  $W_{1}$.} 
		
		\noindent{\em Case 4.  $x$ goes to the right to $l_4$,  then moves along $l_4$ by a distance $\ge 0$, and finally returns to the interior of $W_{1}$}.
		
		\noindent{\em Case 5. $x$ turns left and goes to $l_2$, then goes to the gluing edge, and finally traverses $W_{2}$ from $l_4$ into the interior of $W_{1}$.}	
		
		\noindent{\em Case 6.  $x$ turns left and goes to the $l_2$,  then moves along $l_2$ by a distance $\ge 0$, and finally returns to the interior of $W_{1}$}.
		
		\noindent{\em Case 7.  $x$ goes down to $l_3$, then goes through $W_{4},W_{8},W_{11}$, and finally goes back to the interior of $W_{1}$.}
		
		\noindent{\em Case 8.  $x$ goes down to $l_3$,  then moves along $l_3$ by a distance $\ge 0$, and finally returns to the interior of the $W_{1}$}.	
		
		For Case 1, it is easy to see that
		$$d_M(x,y)\geq \frac{3}{4}>d_{\mathbb{E}}(x,y).$$ 
		For Case 2, the triangular inequality implies that
		$d_M(x,y)>d_{\mathbb{E}}(x,y).$ 
		Similarly, we can prove the other cases and thus complete the proof.
	\end{proof}
	
	\begin{lem}\label{lem:dis2}
		Let $G=(V,E)$, $\mathbb{T}^2$, $\{W_i\}_{i=1}^{12}$, and $\{S_e\}_{e\in E}$  be defined as above. Fix $i\in V$ and let $x,y\in W_i$. Then
		$$d_{\mathbb{T}^2}(S_e(x),S_e(y))=d_{\mathbb{E}}(S_e(x),S_e(x)).$$
	\end{lem}
	\begin{proof}
		For the fixed $i\in V$, by Remark \ref{rem:tu}, there exists some $j\in V$ such that $S_e(x), S_e(y)\subseteq W_j$. By Lemma \ref{lem:dis1}, we have
		$$d_{\mathbb{T}^2}(S_e(x),S_e(y))=d_{\mathbb{E}}(S_e(x),S_e(x)).$$
	\end{proof}

	\begin{exam}\label{exam2}
		Let $G=(V,E)$, $\mathbb{T}^2$, $W$, $\{h_m\}_{m=1}^4$, $\{S_e\}_{e\in E}$, and $K$ be defined as above. Let $\mu$ be an invariant measure of the GIFS  $\{S_e\}_{e\in E}$ of bi-Lipschitz contractions on $W$.  Then $\underline{\operatorname{dim}}_{\infty}(\mu)>0$. Consequently, the conclusions of Theorems \ref{thm:wmain1}, \ref{thm:hmain1}, and \ref{thm:smain1} hold for $\mu$, and the conclusions of Theorems \ref{thm:wmain2}, \ref{thm:hmain2}, and \ref{thm:smain2} hold for all such $\mu$ and all $F\in {\rm Lip}(\operatorname{dom}\mathcal{E})$.
	\end{exam}
	\begin{proof}
	This follows by combining Lemmas \ref{lem:7.1} and  \ref{lem:dis2}.
	\end{proof}
		\begin{figure}[H]
		\centering
		\mbox{\subfigure[]
			{	\includegraphics[scale=0.25]{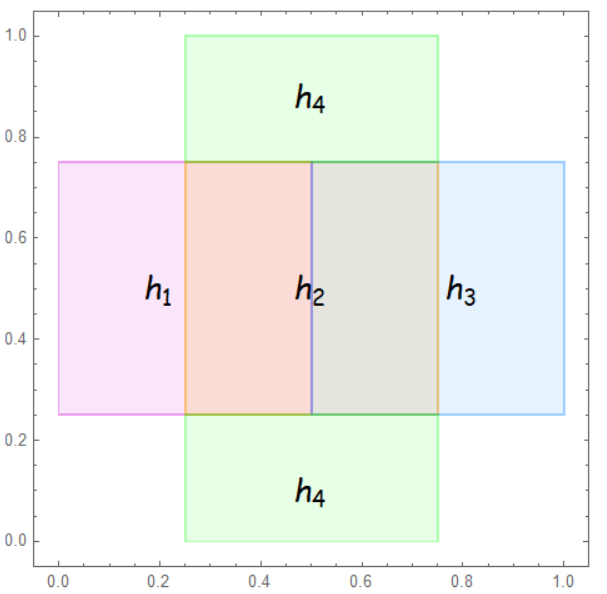}}
		}\qquad
		\mbox{\subfigure[]
			{	\includegraphics[scale=0.36]{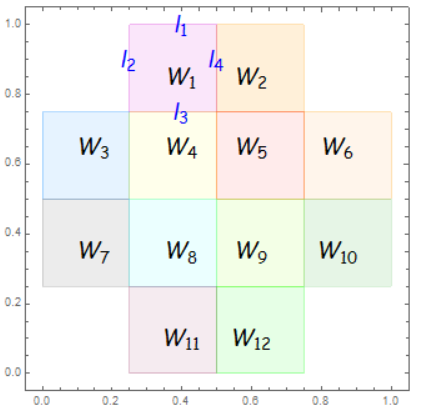}}
		}\\
		\mbox{\subfigure[]
			{	\includegraphics[scale=0.236]{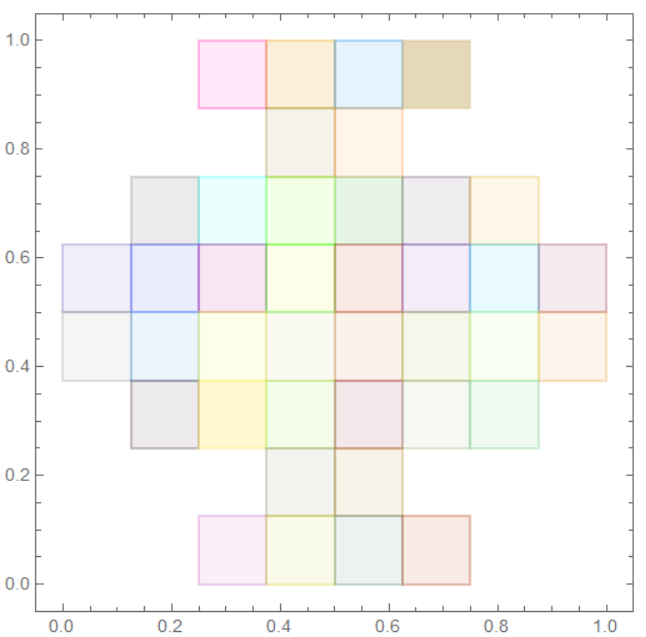}}
		}\qquad\quad
		\mbox{\subfigure[]
			{\includegraphics[scale=0.245]{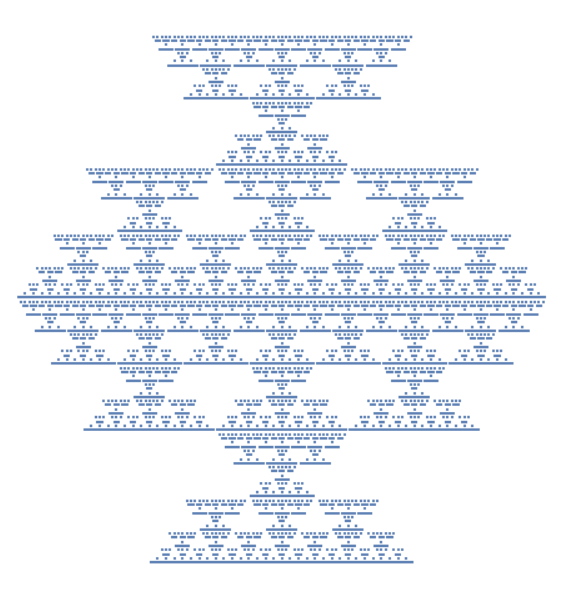}}
		}
		\vspace{0.5cm}
		\caption{ (a) $\{h_m\}_{m=1}^4$ on $\R^2\backslash \Z^2$. (b) First iteration of $G$ on $\R^2$ associated to the relations induced by  $\{h_m\}_{m=1}^4$.  (c) Second iteration. (d) Attractor.}
		\label{fig:2}
	\end{figure}

\section{Example of solutions}\label{S:So}
	\setcounter{equation}{0}
	
	In this section, we study weak solutions of linear wave, heat, and  Schr\"odinger equations by using examples on $\mathbb{S}^1$.
	
	\begin{exam}\label{ex:efd}
		Let $M=\mathbb{S}^1:=\{(\cos(\theta+\pi/2),\sin(\theta+\pi/2)):\theta\in[-\pi,\pi]\}$	and let $\Omega:=\{(\cos(\theta+\pi/2),\sin(\theta+\pi/2)):\theta\in (-\pi/2,\pi/2)\}\subseteq M$ be an open set. Let $\mu$ be the Dirac point mass at the north pole $\theta=0$. Then $\dim(L^2(\Omega,\mu))=1$. Let 
		\begin{align*}
			\varphi=\left\{
			\begin{aligned}
				&\frac{	2\theta}{\pi}+1,\qquad \,\,\,\, &&\theta\in [-\pi/2,0],\\
				&		-\frac{2\theta}{\pi}+1,\qquad &&\theta\in (0,\pi/2].\\
			\end{aligned}	
			\right.
		\end{align*}
		Then $\varphi$ is an eigenfunction of $-\Delta_\mu$ with $\lambda=4/\pi$ being an eigenvalue. Moreover,  $\underline{\operatorname{dim}}_{\infty}(\mu)=0$.
	\end{exam}
	\begin{proof}
	We use the eigenvalues equation (see \cite{Ngai-Ouyang_2023})
			\begin{align*}
				\int_\Omega(-\Delta\varphi)v\,d\nu=\lambda\int_\Omega \varphi v\,d\mu\qquad v\in C_c^\infty(\Omega).
			\end{align*}
		Since 	$$\int_\Omega(-\Delta\varphi)v\,d\nu=\int_\Omega\Big(\frac{4\delta_0}{\pi}\Big)v\,d\nu=\frac{4}{\pi}v\Big(\frac{\pi}{2}\Big)$$
		and $\lambda\int_\Omega \varphi v\,d\mu=\lambda  v(\pi/2)$, we conclude that $4/\pi$ is an eigenvalue and $\varphi$ is an eigenfunction.
	\end{proof}
	
	\begin{exam}\label{ex:slud}
			Let $M$, $\Omega$, $\mu$, $\varphi$, and $\lambda$ be defined as in  Example \ref{ex:efd}. 
			\begin{enumerate}
				\item [(a)] Let $g=\varphi/4$, $h=0$, and $f=0$. Then the weak solution of the linear wave equation \eqref{eq:wlinear} is
				\begin{align*}
					u(t)=\frac{\varphi}{4}\cos\Big(\frac{2t}{\sqrt{\pi}}\Big).
				\end{align*}
				See Figure \ref{fig:DW}. Note that the solution is periodic.
				\item [(b)] Let $g=\varphi/4$ and $f=c$, where $c$ is constant. Then the weak solution of the linear heat equation \eqref{eq:hlinear} is
				\begin{align*}
					u(t)=\frac{\varphi}{4}\big( e^{-4t/\pi}\big(1-c\pi\big)+c\pi\big).	
				\end{align*}
				See Figures \ref{fig:HDF0}--\ref{fig:HDF3}.  Note that if $c=0$, then $u(t)$ tends to zero due to heat dissipation at the boundary points. If $c=1/8$, $u(t)$ decreases as heat dissipation exceeds heat supply, and then it reaches an equilibrium when the rate of heat dissipation equals  that rate of heat supply. If $c=3/4$, $u(t)$ increases since heat supply exceeds heat dissipation, and again, $u(t)$ reaches an equilibrium when the rate of heat dissipation equals that of heat supply.
				\item [(c)] Let $g=\varphi/4$ and $f=0$. Then the weak solution of the linear Schr\"odinger equation \eqref{eq:slinear} is
				\begin{align*}
					u(t)=\frac{\varphi}{4}e^{-i4t/\pi}.
				\end{align*}
				Note that both the real and imaginary parts of $u(t)$ exhibit periodic motion.
				
			\end{enumerate}
			
	\end{exam}
	
	\begin{figure}[H]
		\centering
		\mbox{\subfigure[]
			{	\includegraphics[scale=0.2]{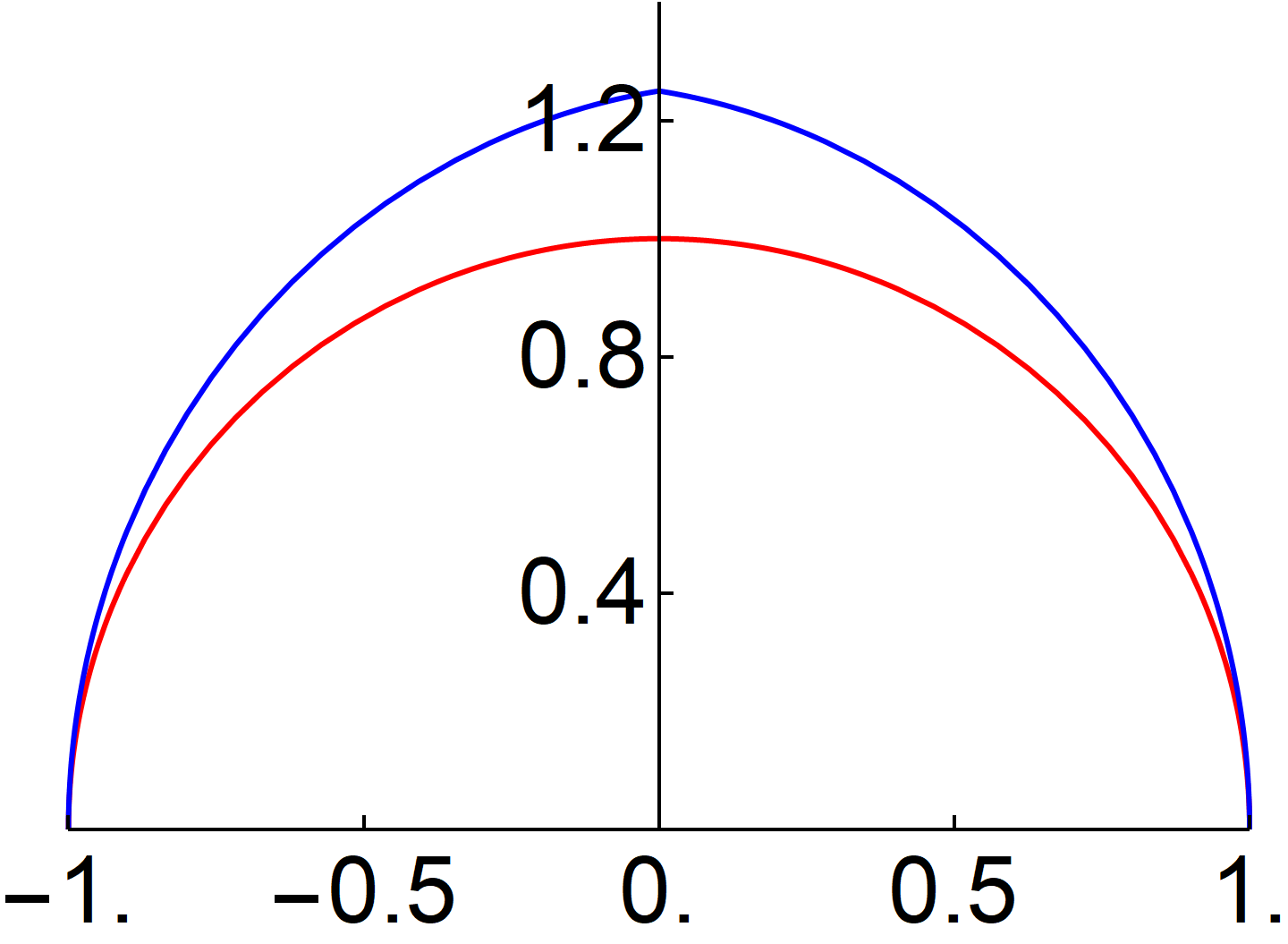}}
		}\quad
		\mbox{\subfigure[]
			{	\includegraphics[scale=0.2]{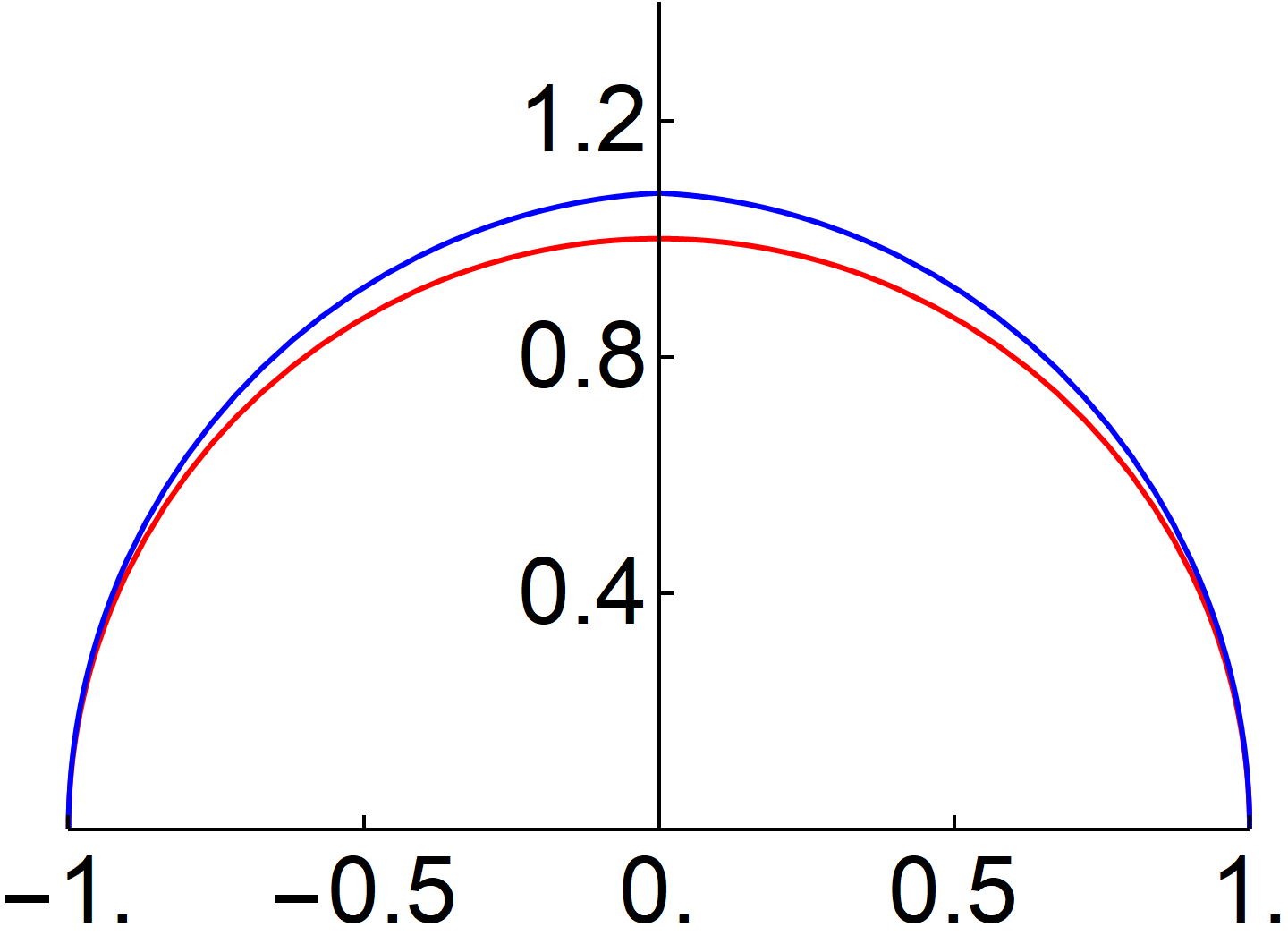}}
		}\quad
		\mbox{\subfigure[]
			{	\includegraphics[scale=0.2]{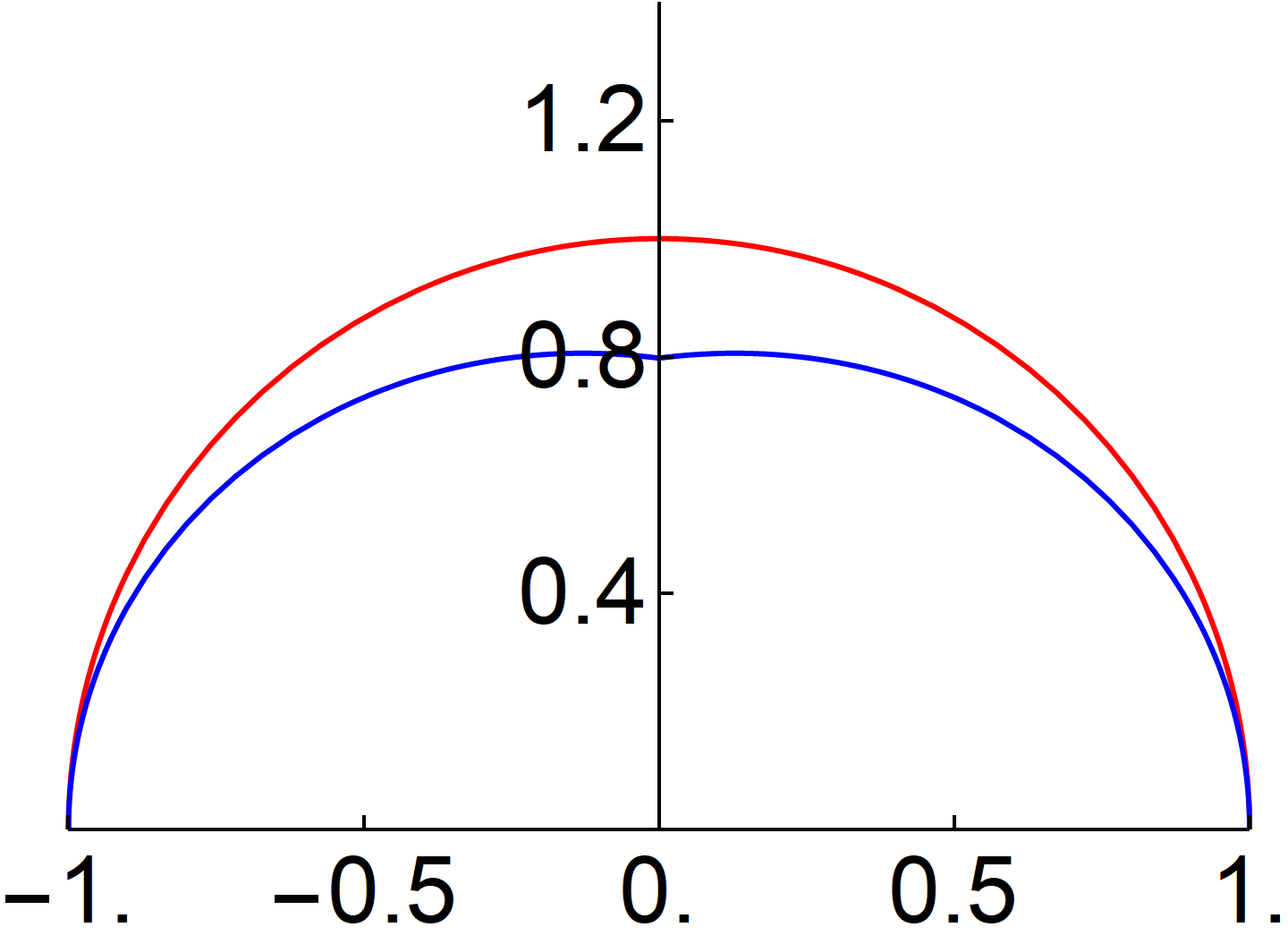}}
		}\quad
		\mbox{\subfigure[]
			{	\includegraphics[scale=0.2]{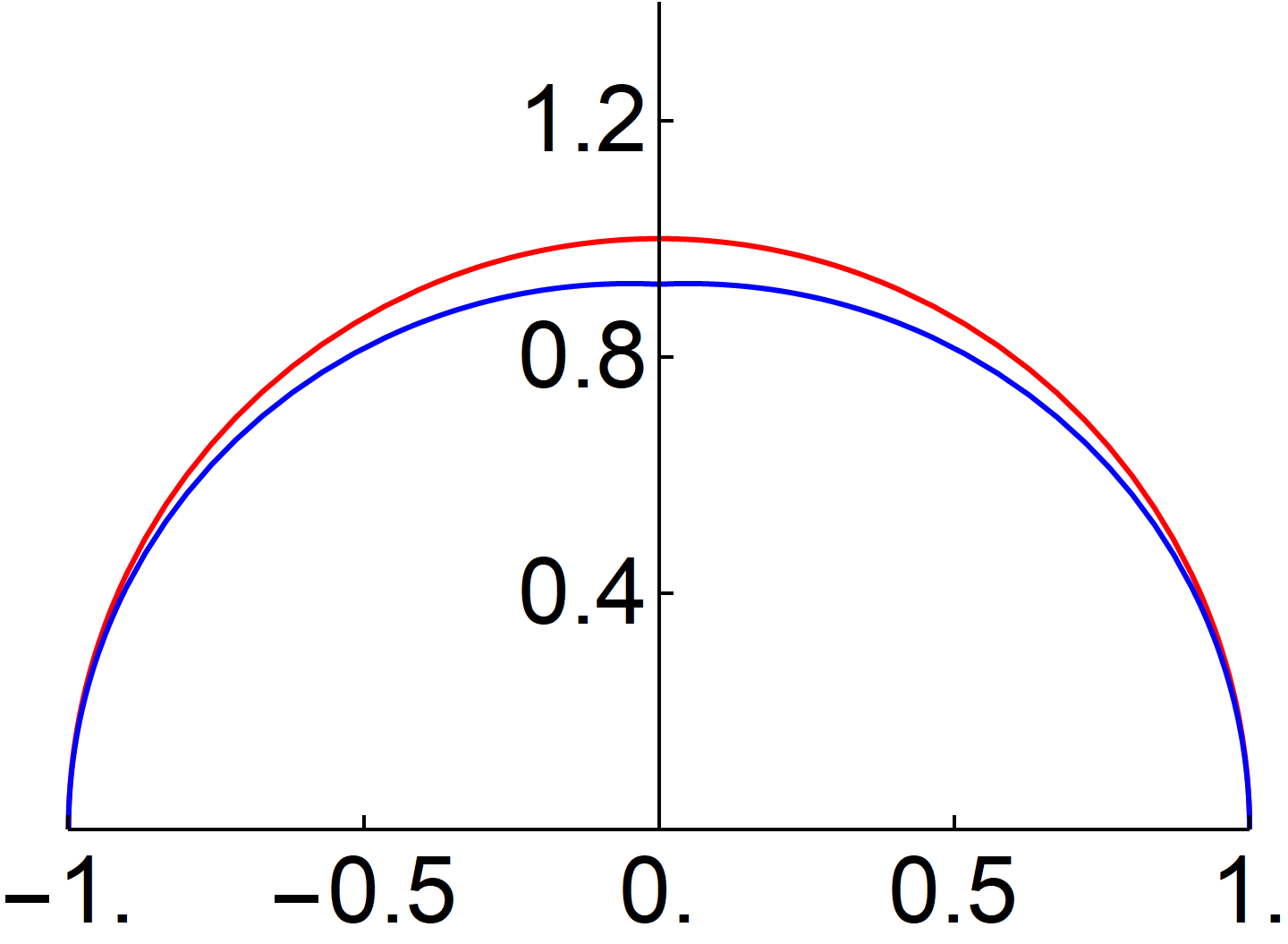}}
		}\quad
		\mbox{\subfigure[]
			{	\includegraphics[scale=0.2]{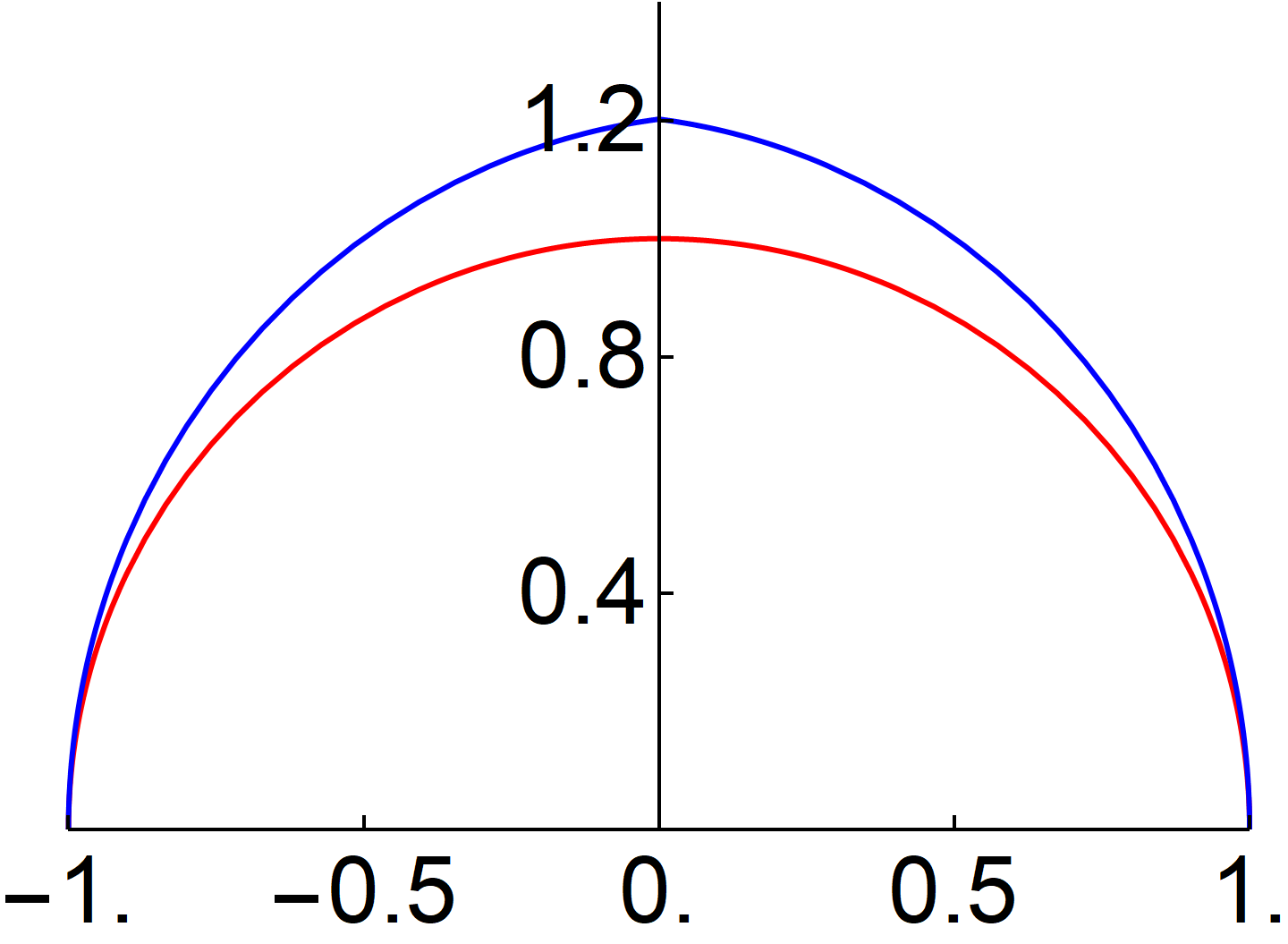}}
		}
		\vspace{0.5cm}
		\caption{Graphs of the solution $u(t)$ of  the linear wave equation \eqref{eq:wlinear} with $g=\varphi/4$, $h=0$, and $f=0$. The graphs are drawn with $t=0$, $2\pi\sqrt{\pi}/9$, $4\pi\sqrt{\pi}/9$,  $7\pi\sqrt{\pi}/9$,  $\pi\sqrt{\pi}$.}
		\label{fig:DW}
	\end{figure}
	\begin{figure}[H]
		\centering
		\mbox{\subfigure[]
			{	\includegraphics[scale=0.2]{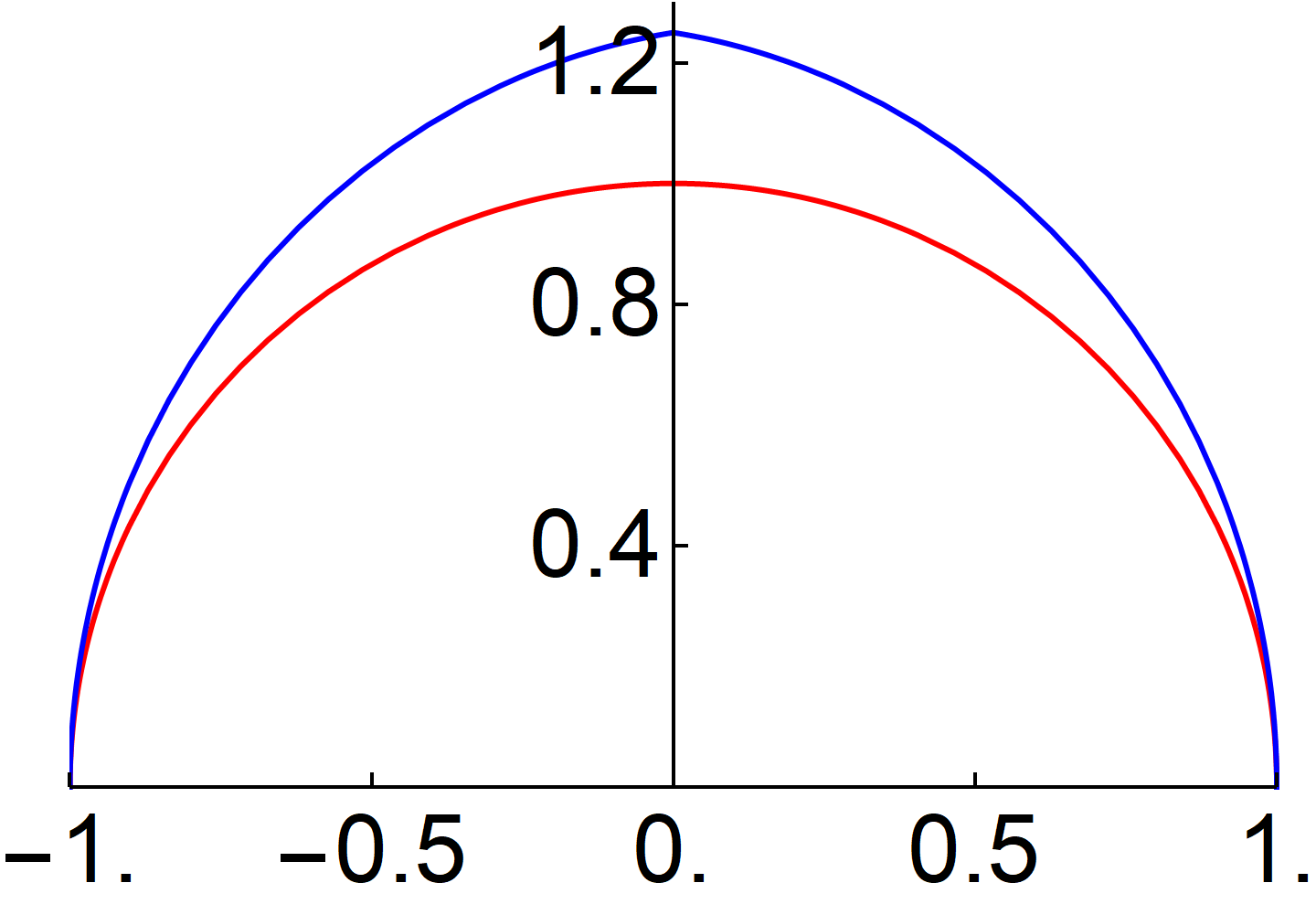}}
		}\quad
		\mbox{\subfigure[]
			{	\includegraphics[scale=0.2]{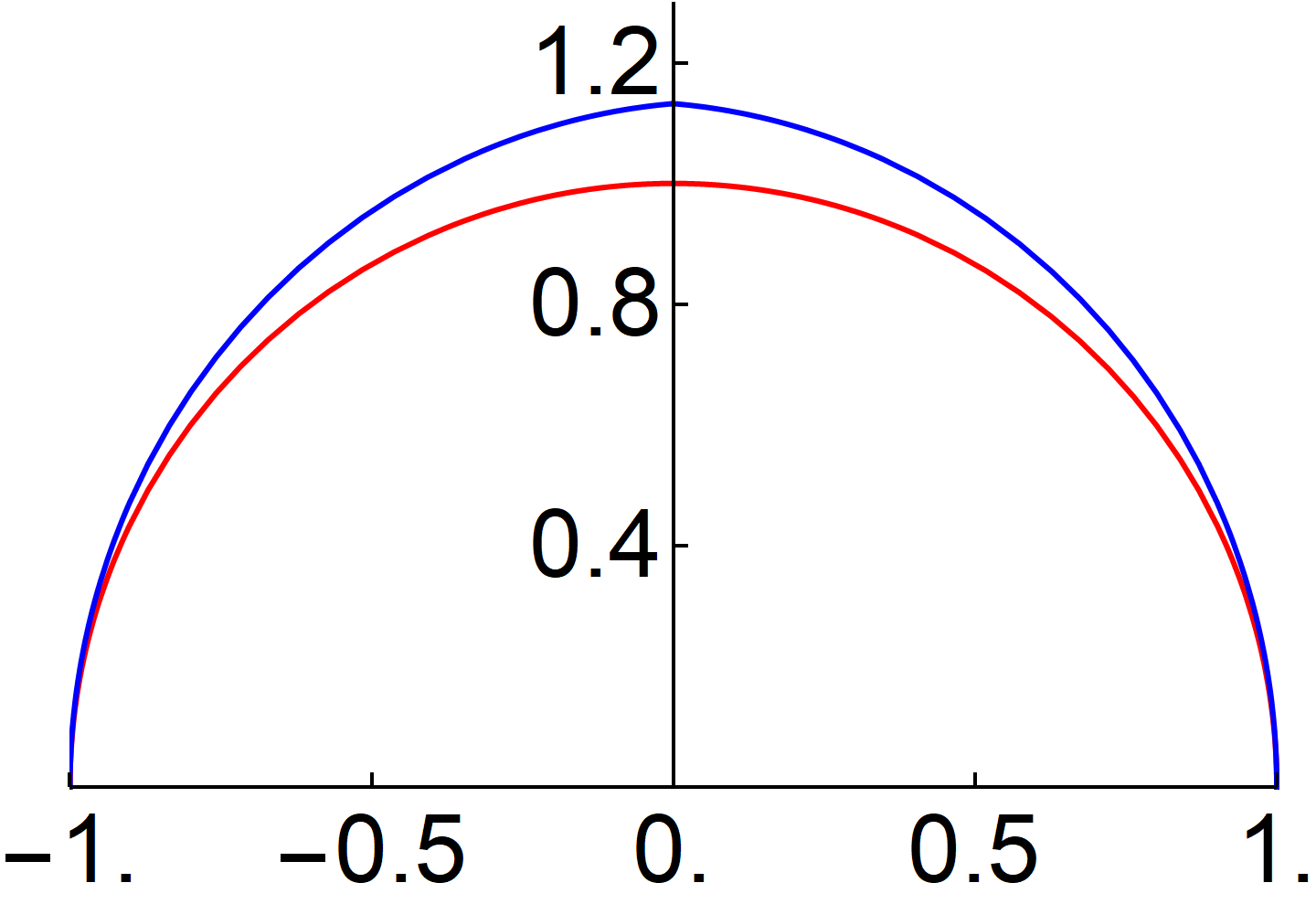}}
		}\quad
		\mbox{\subfigure[]
			{	\includegraphics[scale=0.2]{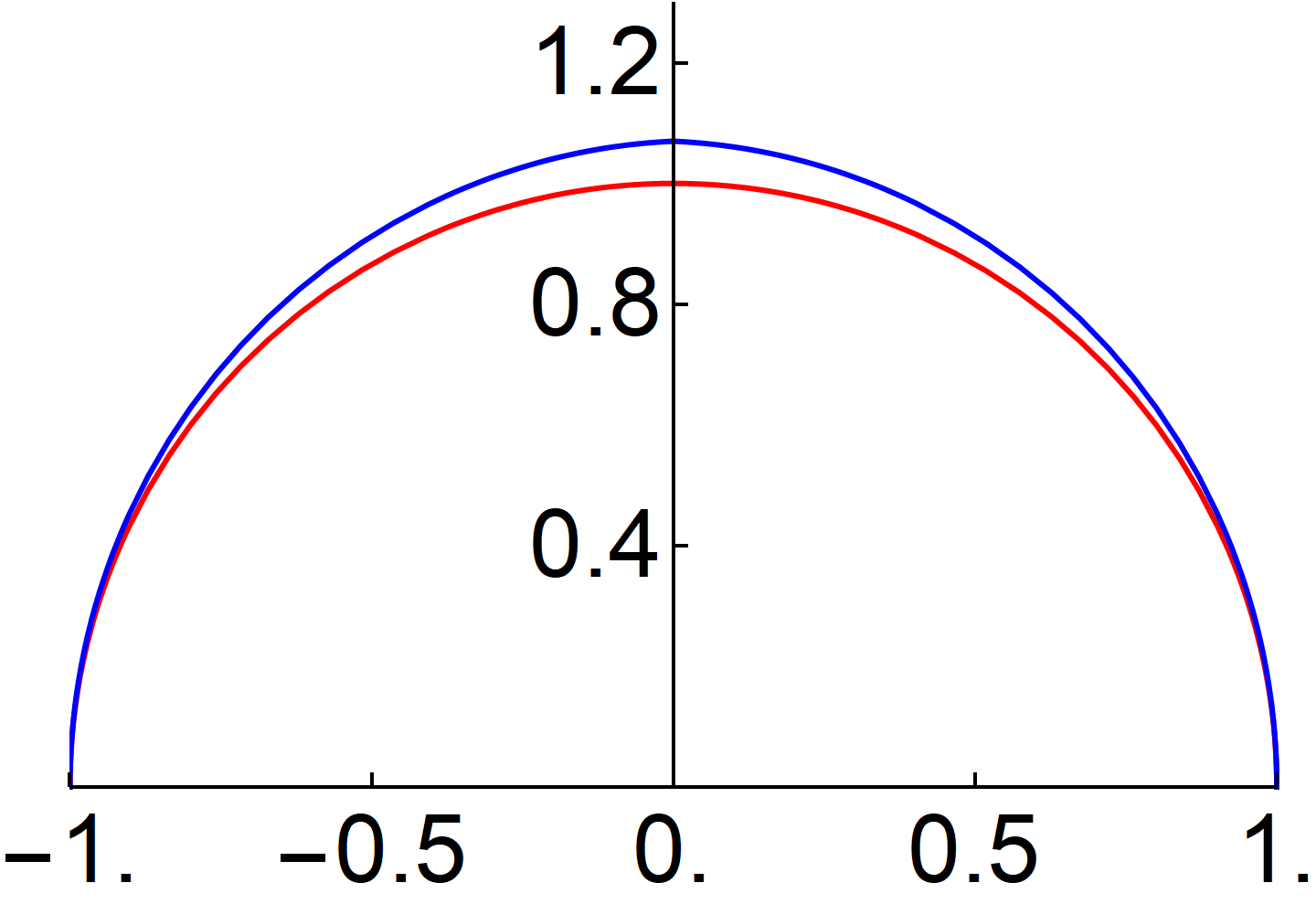}}
		}\quad
		\mbox{\subfigure[]
			{	\includegraphics[scale=0.2]{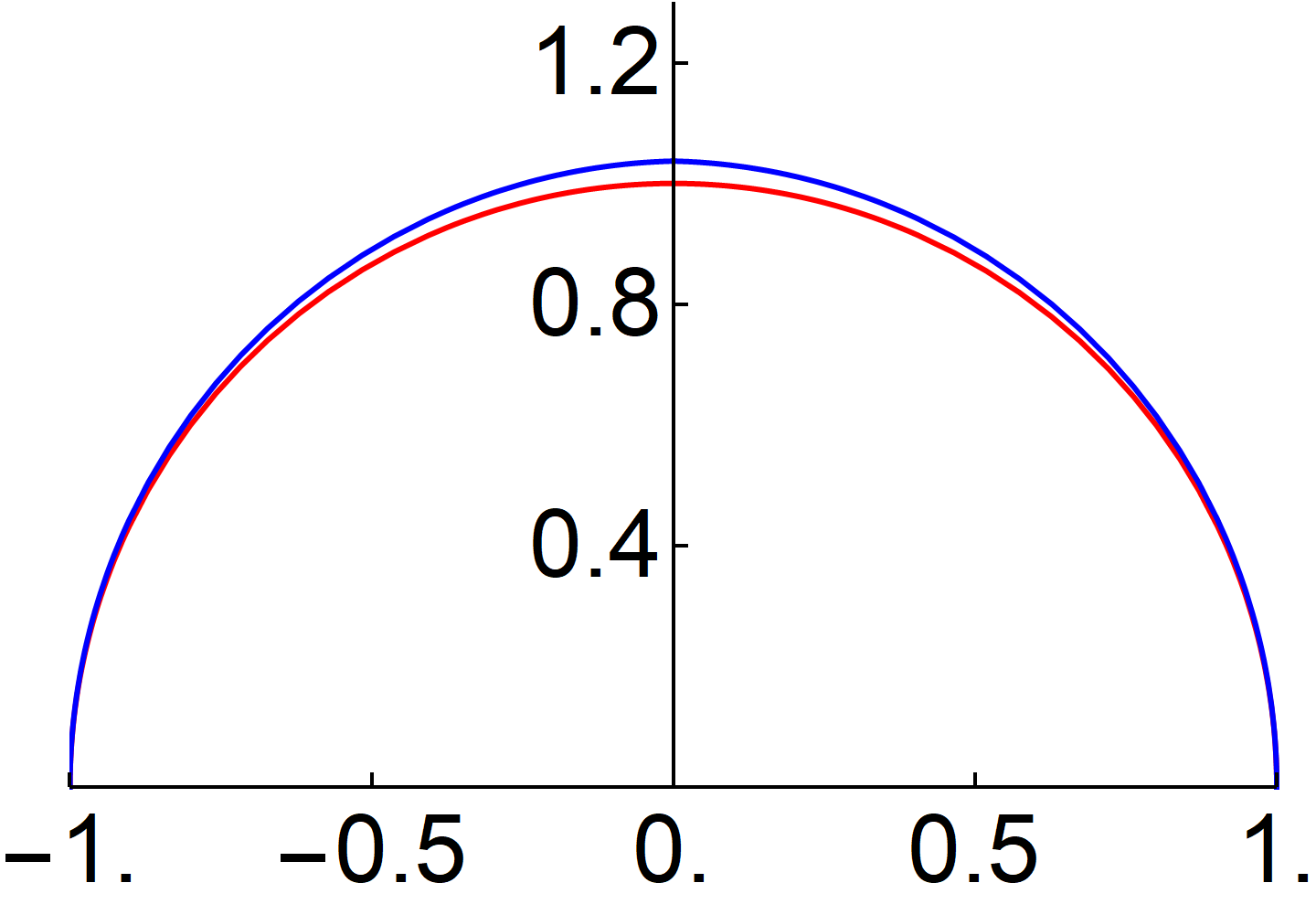}}
		}\quad
		\mbox{\subfigure[]
			{	\includegraphics[scale=0.2]{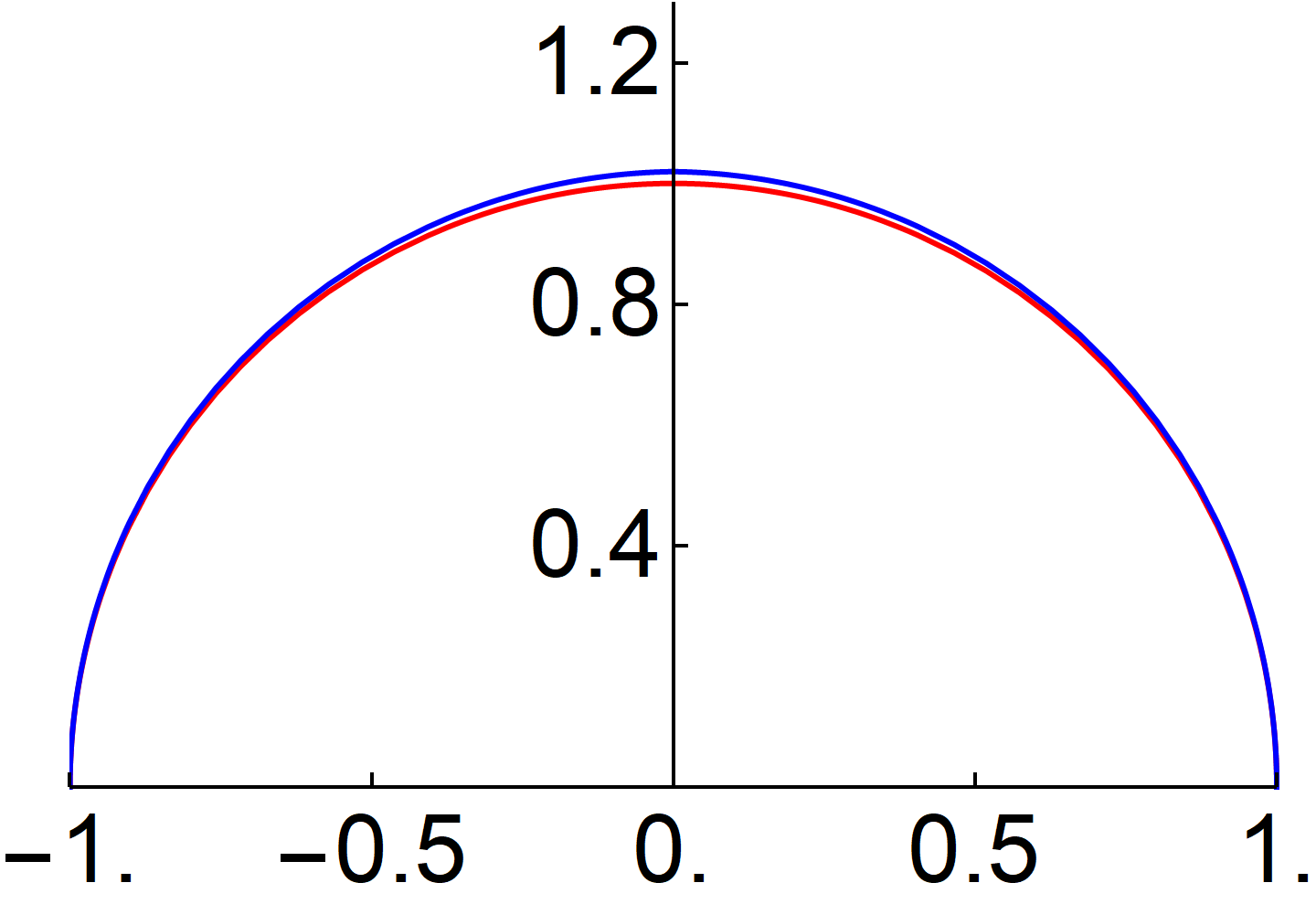}}
		}
		\vspace{0.5cm}
		\caption{Graphs of $u(t)$  of the linear heat equation \eqref{eq:hlinear} for $t=0$, $1/2$, $1$, $3/2$, $2$. The graphs are drawn with $g=\varphi/4$ and $f=0$. $u(t)$ decreases to $0$ as $t\to\infty$.}
		\label{fig:HDF0}
	\end{figure}
	
	\begin{figure}[H]
		\centering
		\mbox{\subfigure[]
			{	\includegraphics[scale=0.2]{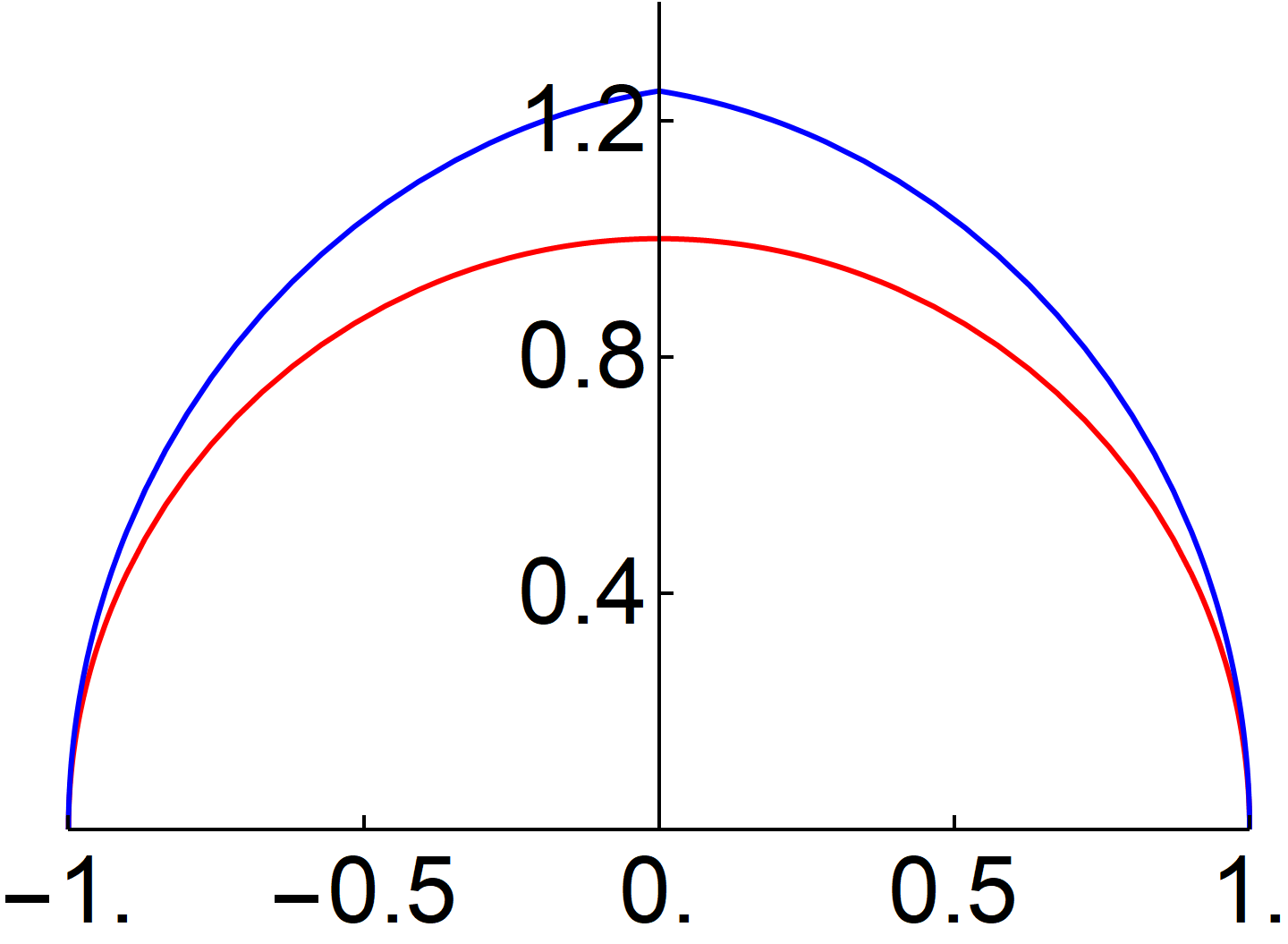}}
		}\quad
		\mbox{\subfigure[]
			{	\includegraphics[scale=0.2]{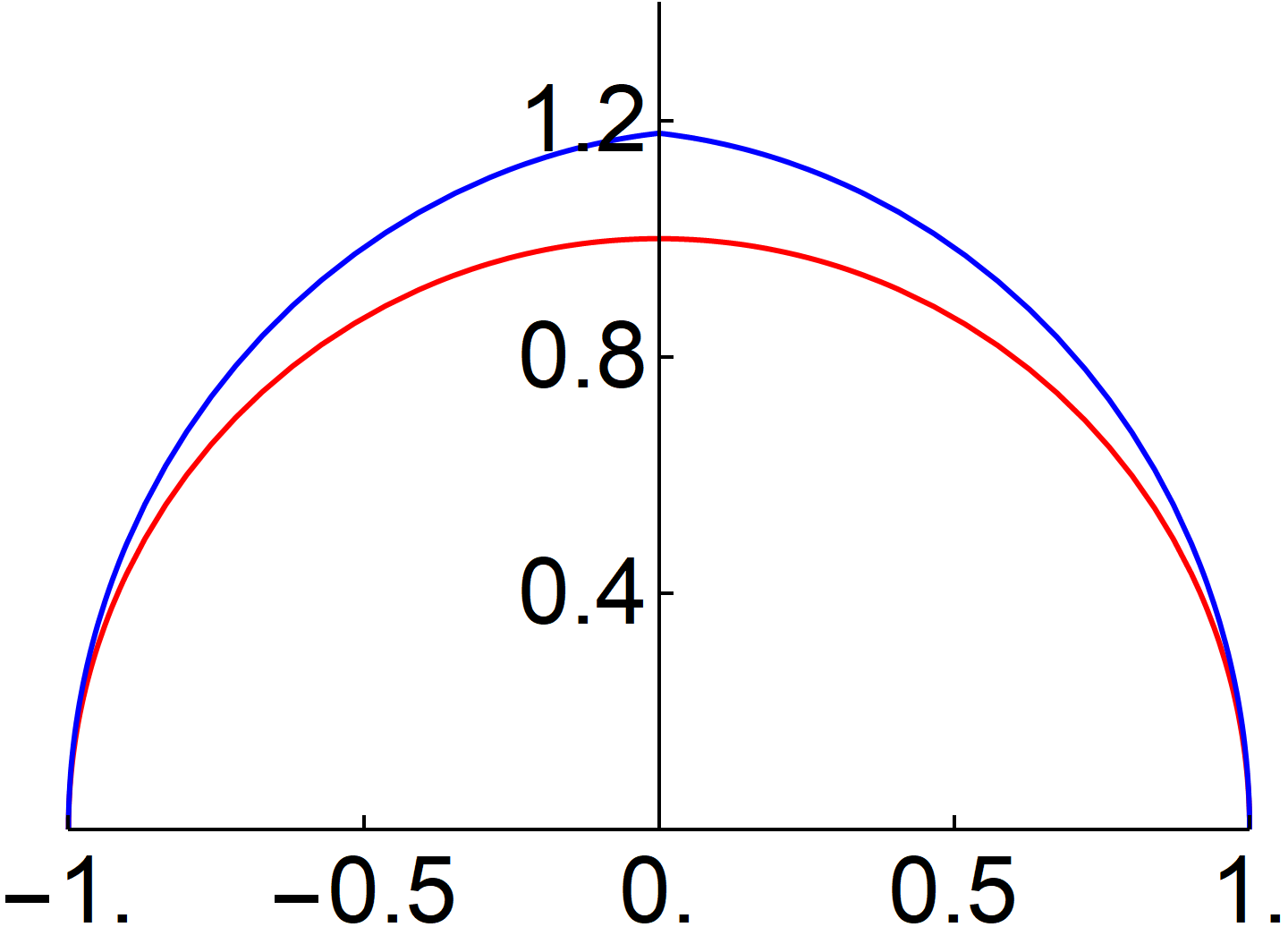}}
		}\quad
		\mbox{\subfigure[]
			{	\includegraphics[scale=0.2]{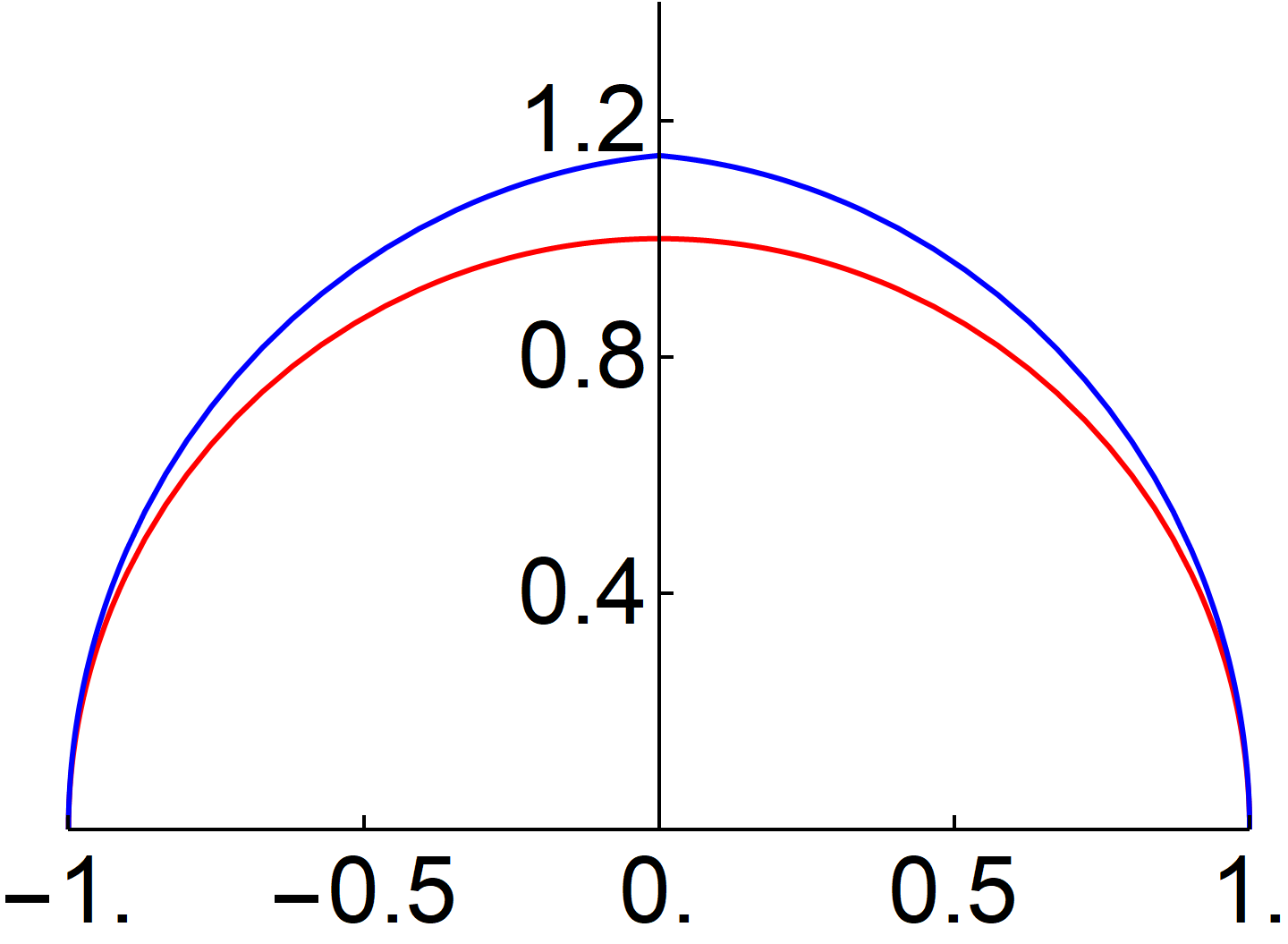}}
		}\quad
		\mbox{\subfigure[]
			{	\includegraphics[scale=0.2]{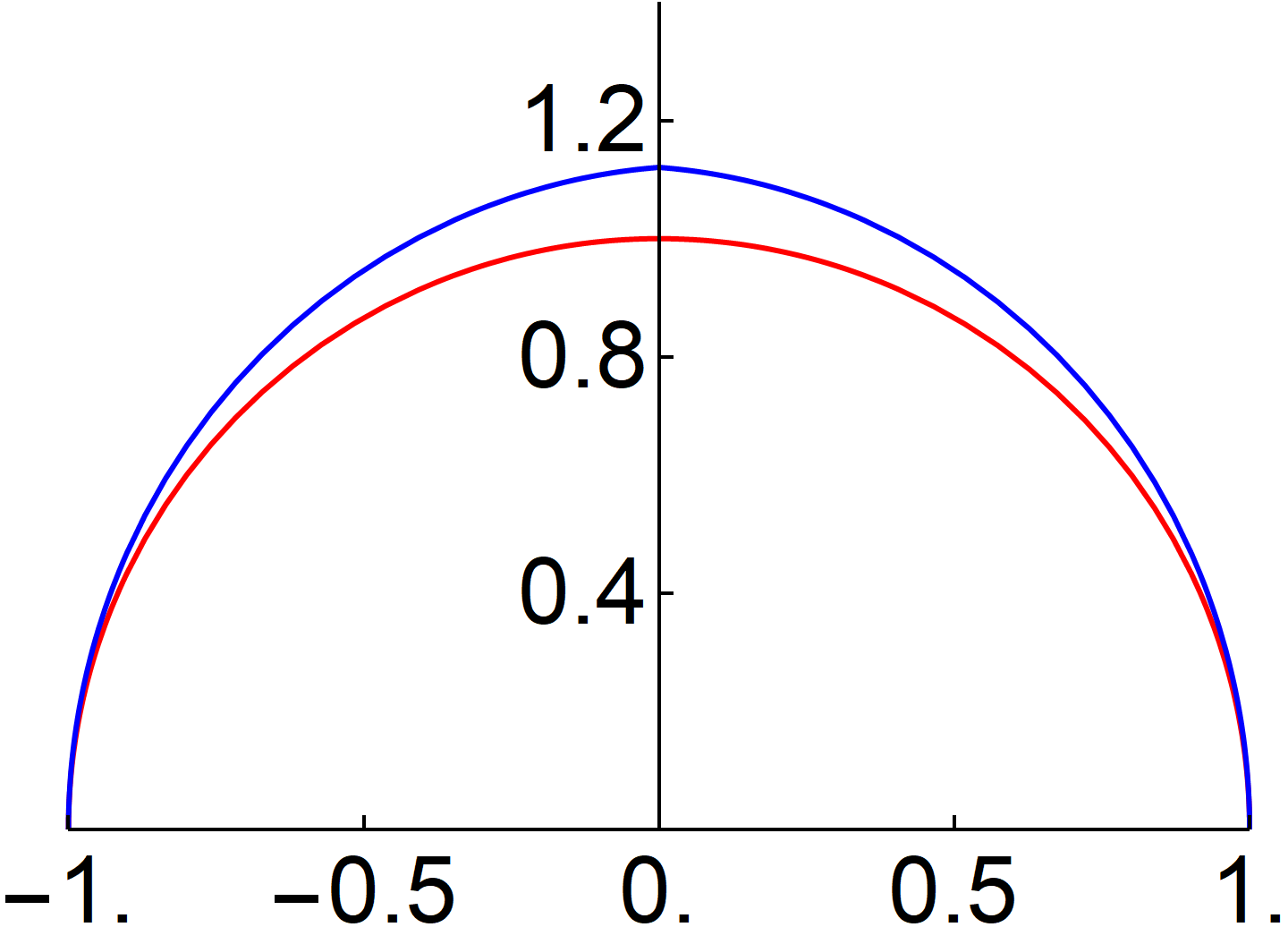}}
		}\quad
		\mbox{\subfigure[]
			{	\includegraphics[scale=0.2]{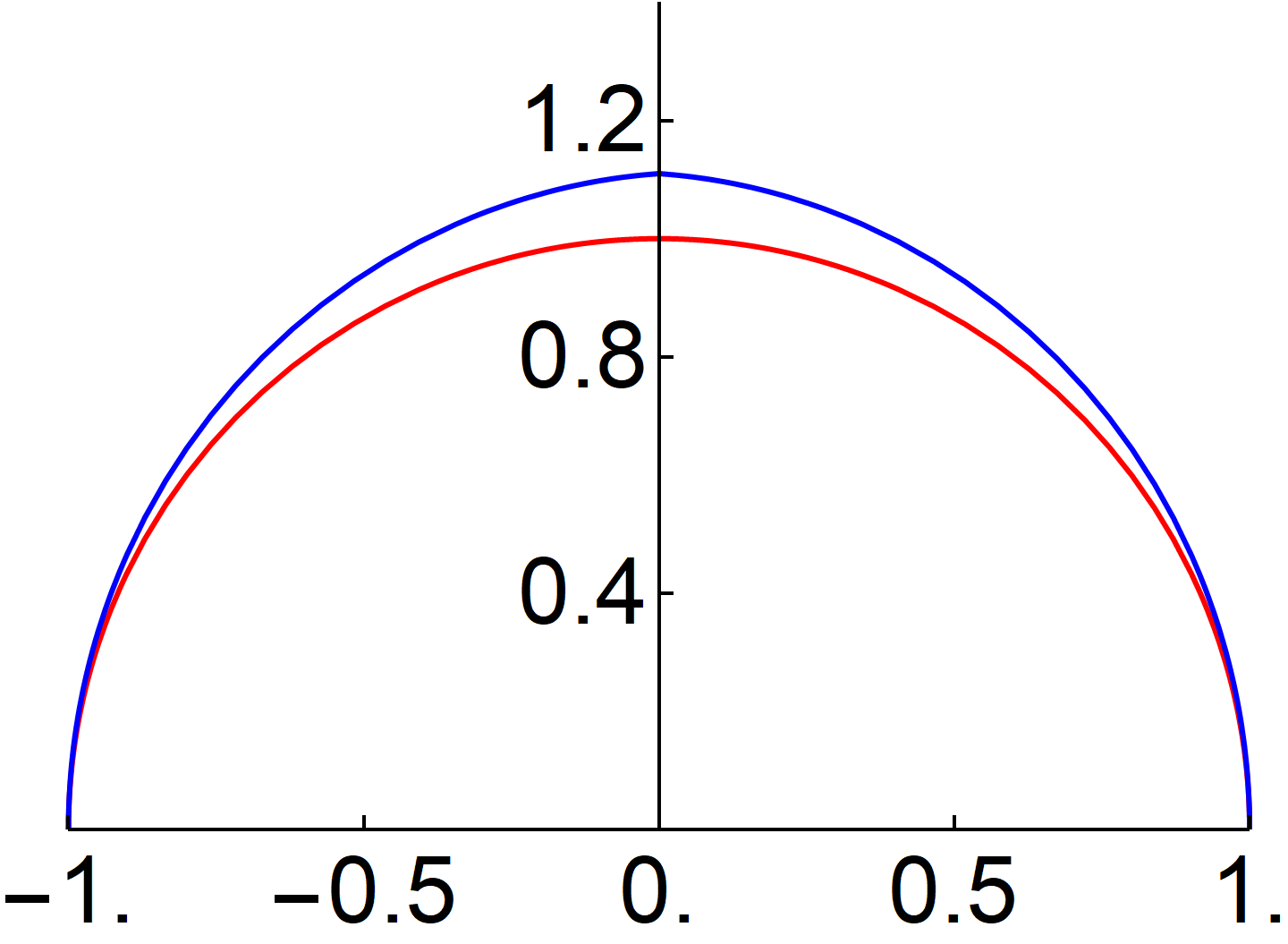}}
		}
		\vspace{0.5cm}
		\caption{Graphs of $u(t)$ of  the linear heat equation \eqref{eq:hlinear} for $t=0$, $1/2$, $1$, $3/2$, $2$. The graphs are drawn with $g=\varphi/4$ and $f=1/8$. $u(t)$ decreases to a nonzero equilibrium less than its initial value as $t\to\infty$.}
		\label{HDF0.5}
	\end{figure}

	\begin{figure}[H]
		\centering
		\mbox{\subfigure[]
			{	\includegraphics[scale=0.2]{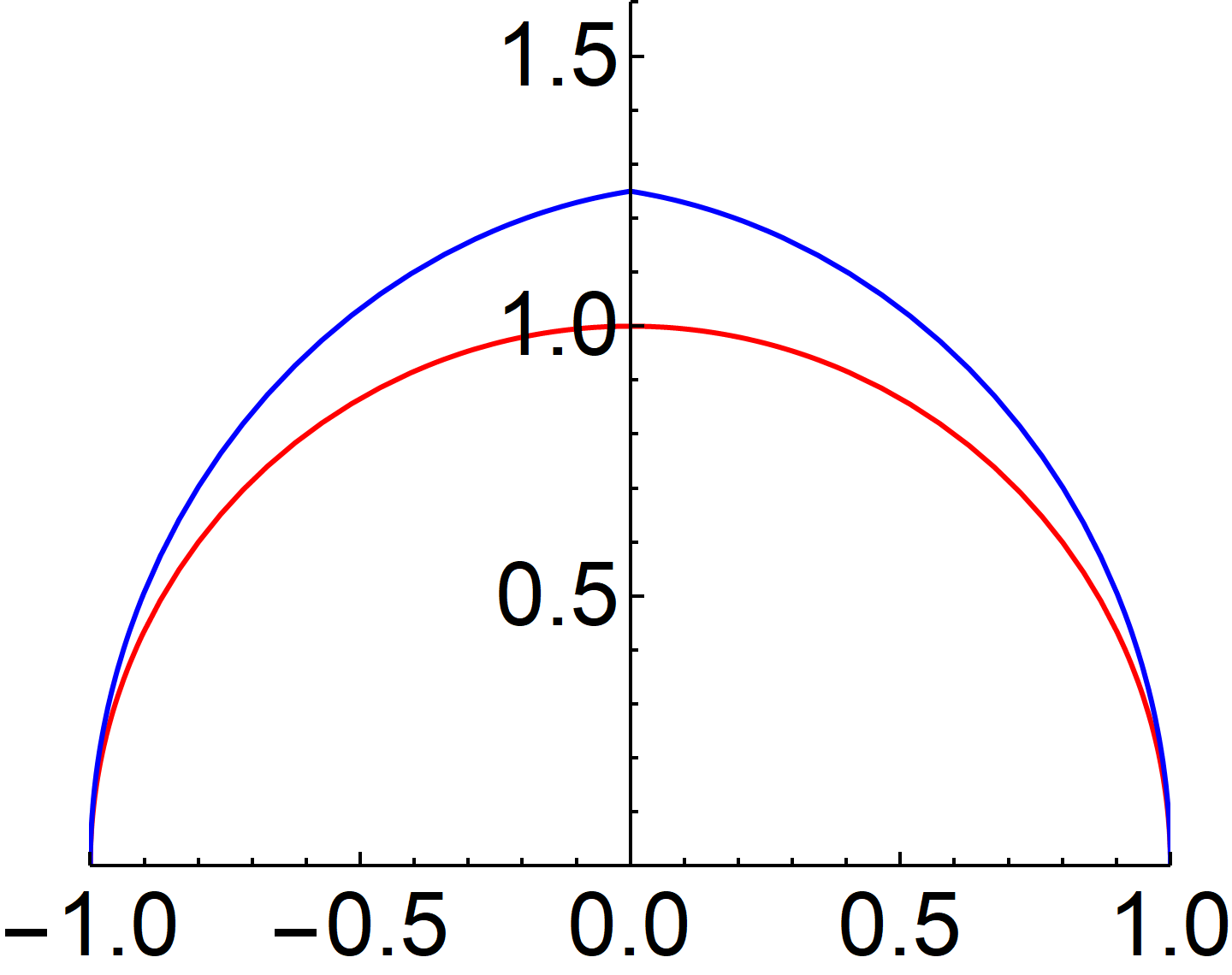}}
		}\quad
		\mbox{\subfigure[]
			{	\includegraphics[scale=0.2]{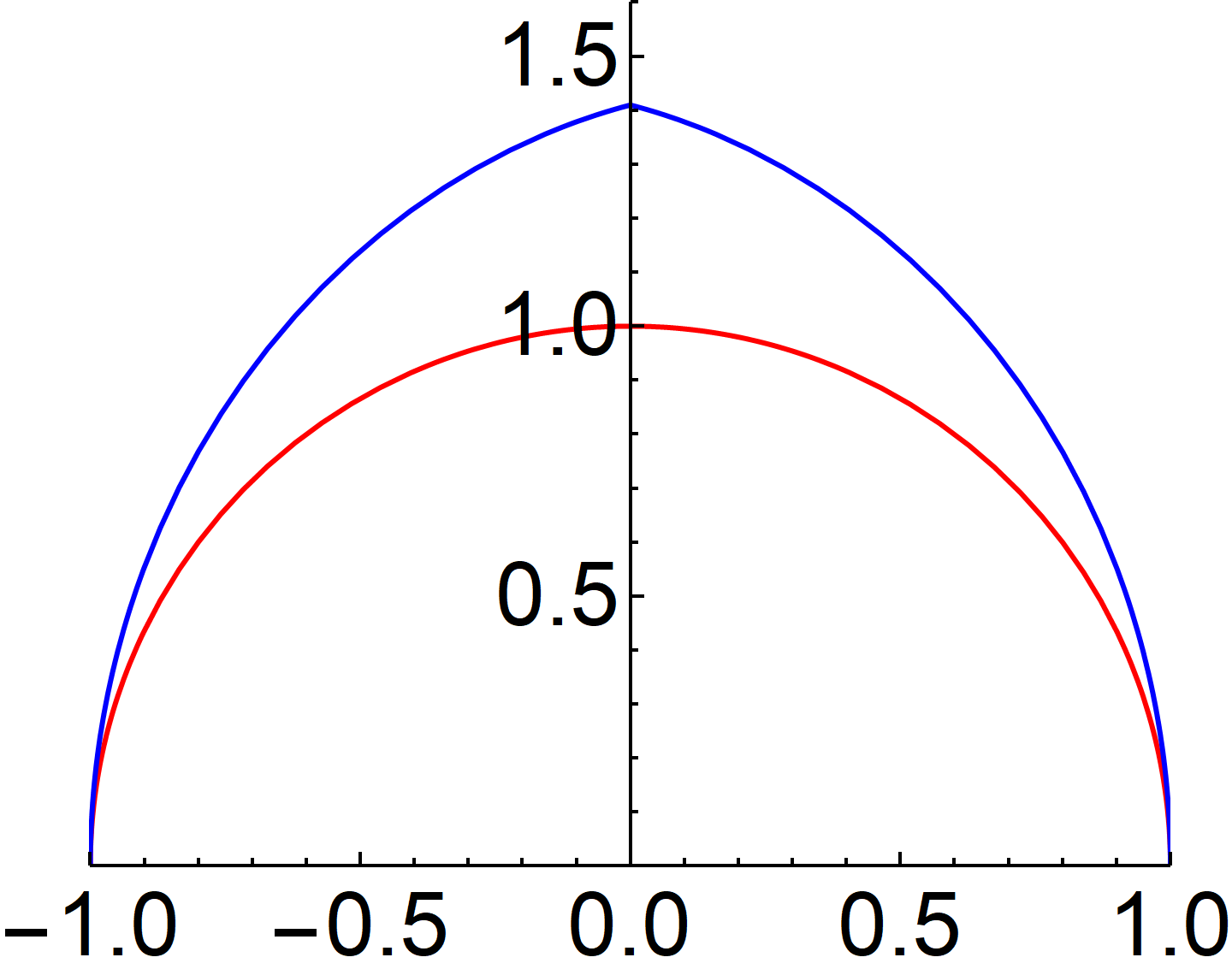}}
		}\quad
		\mbox{\subfigure[]
			{	\includegraphics[scale=0.2]{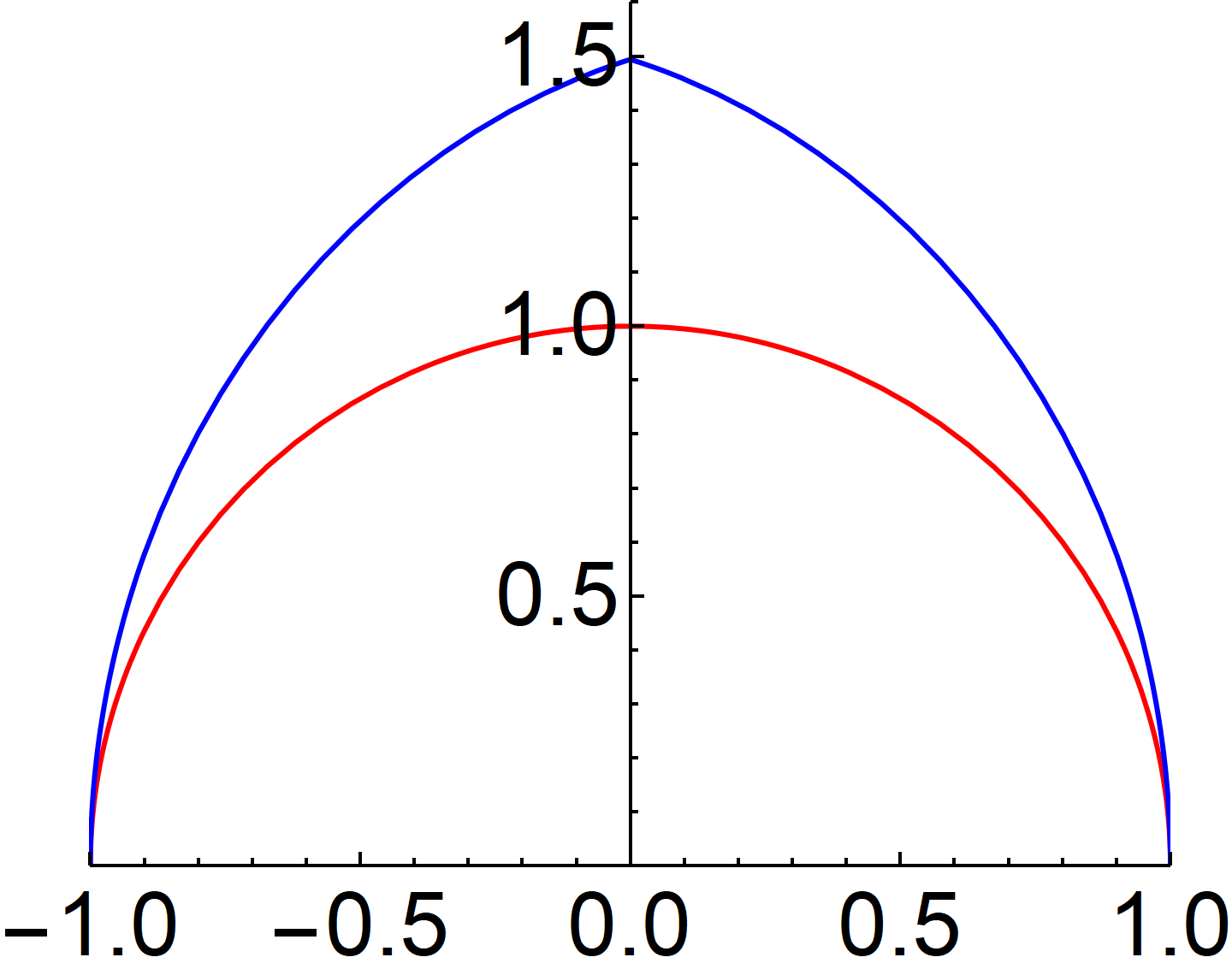}}
		}\quad
		\mbox{\subfigure[]
			{	\includegraphics[scale=0.2]{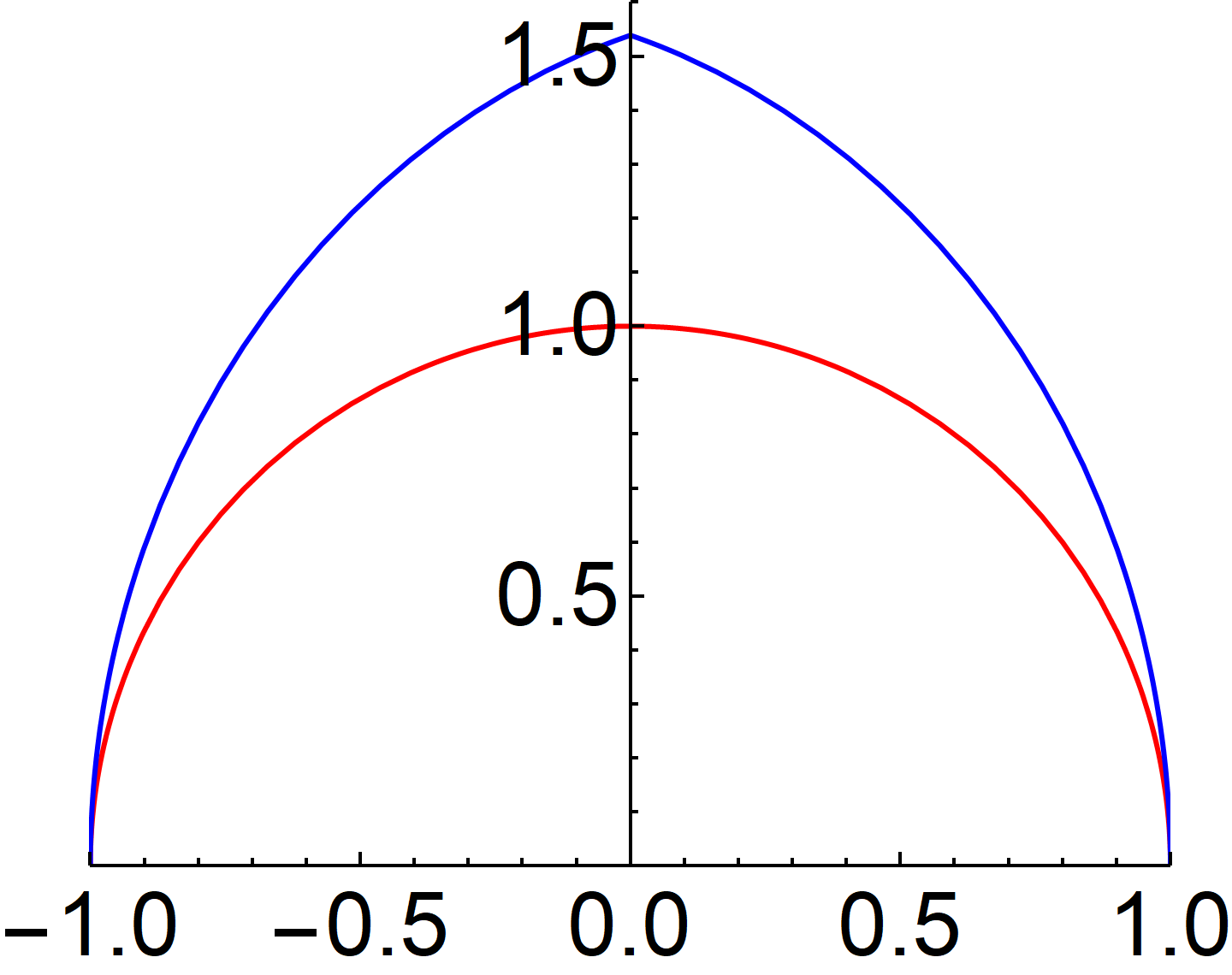}}
		}\quad
		\mbox{\subfigure[]
			{	\includegraphics[scale=0.2]{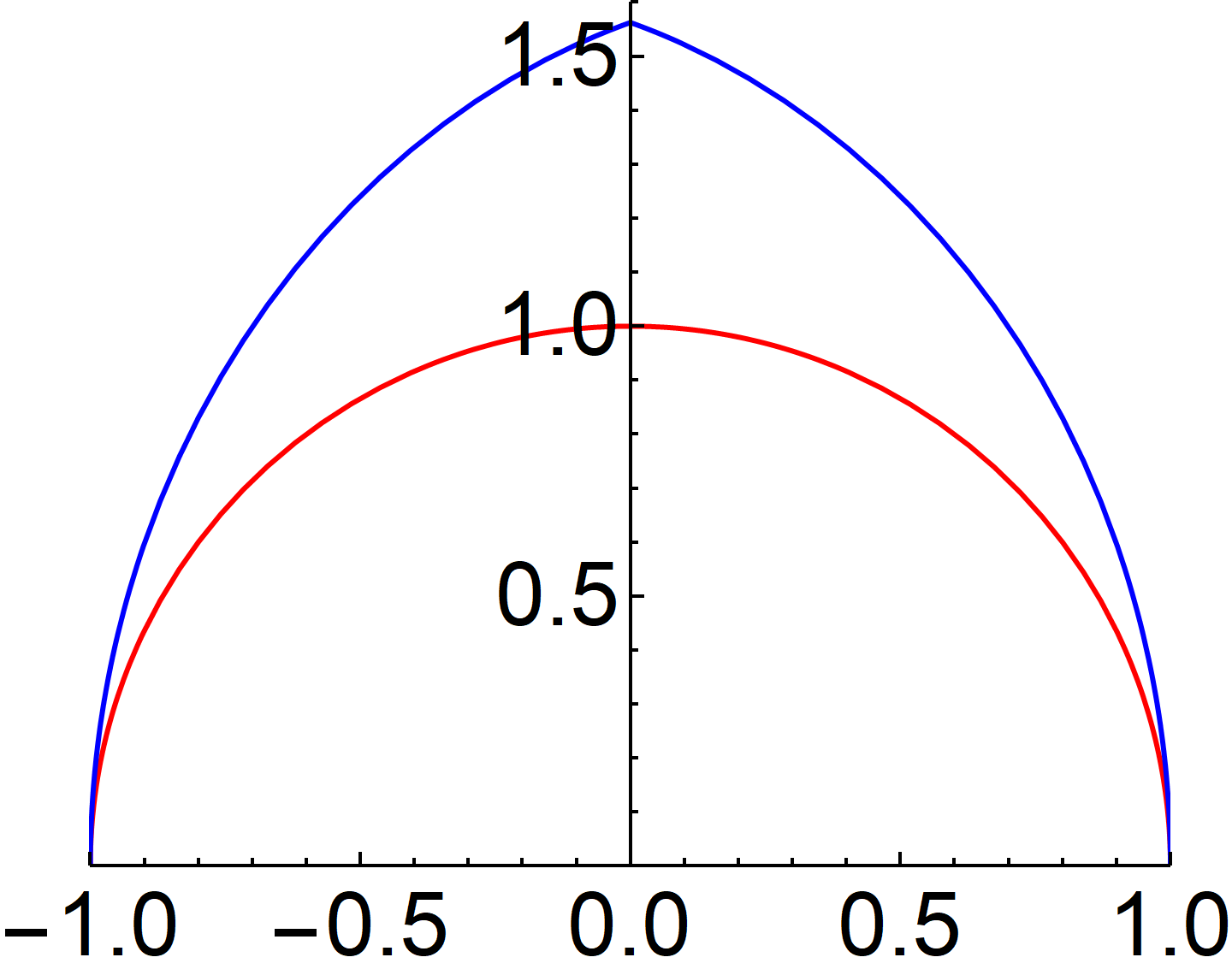}}
		}
		\vspace{0.5cm}
		\caption{Graphs of $u(t)$ of the linear heat equation \eqref{eq:hlinear} for $t=0$, $1/2$, $1$, $3/2$, $2$. The graphs are drawn with $g=\varphi/4$ and $f=3/4$. $u(t)$ increases to an equilibrium state greater than its initial value as $t\to\infty$.}
		\label{fig:HDF3}
	\end{figure}

	\begin{exam}\label{ex:efn}
		Let $M=\mathbb{S}^1:=\{(\cos(\theta+\pi/2),\sin(\theta+\pi/2)):-\pi\leq\theta\leq \pi\}$. Let $\mu$ be two Dirac point mass at the north pole $\theta=0$ and  the south pole $\theta=\pi$. Let $\varphi_1=c_1$ and
		\begin{align*}
			\varphi_2=\left\{
			\begin{aligned}
				&\frac{	2\theta}{\pi}+1&\qquad \theta\in [-\pi,0],\\
				&		\frac{-2\theta}{\pi}+1&\qquad \theta\in (0,\pi].\\
			\end{aligned}	
			\right.
		\end{align*}
		Then $\varphi_1$ and $\varphi_2$ are eigenfunctions of $-\Delta_\mu$ with the corresponding eigenvalues are $\lambda_1=0$ and $\lambda_2=4/\pi$, respectively. Moreover,  $\underline{\operatorname{dim}}_{\infty}(\mu)=0$.
	\end{exam}
	The proof of Example \ref{ex:efn} is similar to that of Example \ref{ex:efd}; we omit the proof.

	\begin{exam}\label{ex:slun}
		Let $M$, $\mu$, $\varphi_1$, $\varphi_2$, and $\lambda$ be defined in  Example \ref{ex:efn}. 
		\begin{enumerate}
			\item [(a)] Let $g=\varphi_2/4$, $h=0$, and $f=0$. Then the weak solution of the linear wave equation \eqref{eq:wlinear} is
			\begin{align*}
				u(t)=\frac{\varphi_2}{4}\cos\Big(\frac{2t}{\sqrt{\pi}}\Big).
			\end{align*}
		Figure \ref{NW} shows the periodic wave motion.
			\item [(b)] Let $g=c_1+c_2\varphi_2$ and $f=c$, where $c$ is a constant. Then the weak solution of the linear heat equation \eqref{eq:hlinear} is
			\begin{align*}
				u(t)=c_1+\varphi_2\Big( e^{-4t/\pi}\Big(c_2-\frac{\pi}{4}c\Big)+\frac{\pi}{4}c\Big).	
			\end{align*}
			For example, if $c_1=1$, $c_2=1$, and $c=0$, then $	u(t)=1+\varphi_2 e^{-4t/\pi}.$ In this case, the temperature of points in the upper semicircle decreases to 1, while that of points on the lower semicircle increases to 1. Total heat energy is conserved.
			\item [(c)] Let $g=\varphi_2/4$ and $f=0$. Then the weak solution of the linear Schr\"odinger equation \eqref{eq:slinear} is
			\begin{align*}
				u(t)=\frac{\varphi_2}{4} e^{-i4t/\pi}.	
			\end{align*}
		Figures \ref{schr} and \ref{ima} show the periodic behavior of $u(t)$.
		\end{enumerate}
		
\end{exam}

\begin{figure}[H]
	\centering
	\mbox{\subfigure[]
		{	\includegraphics[scale=0.22]{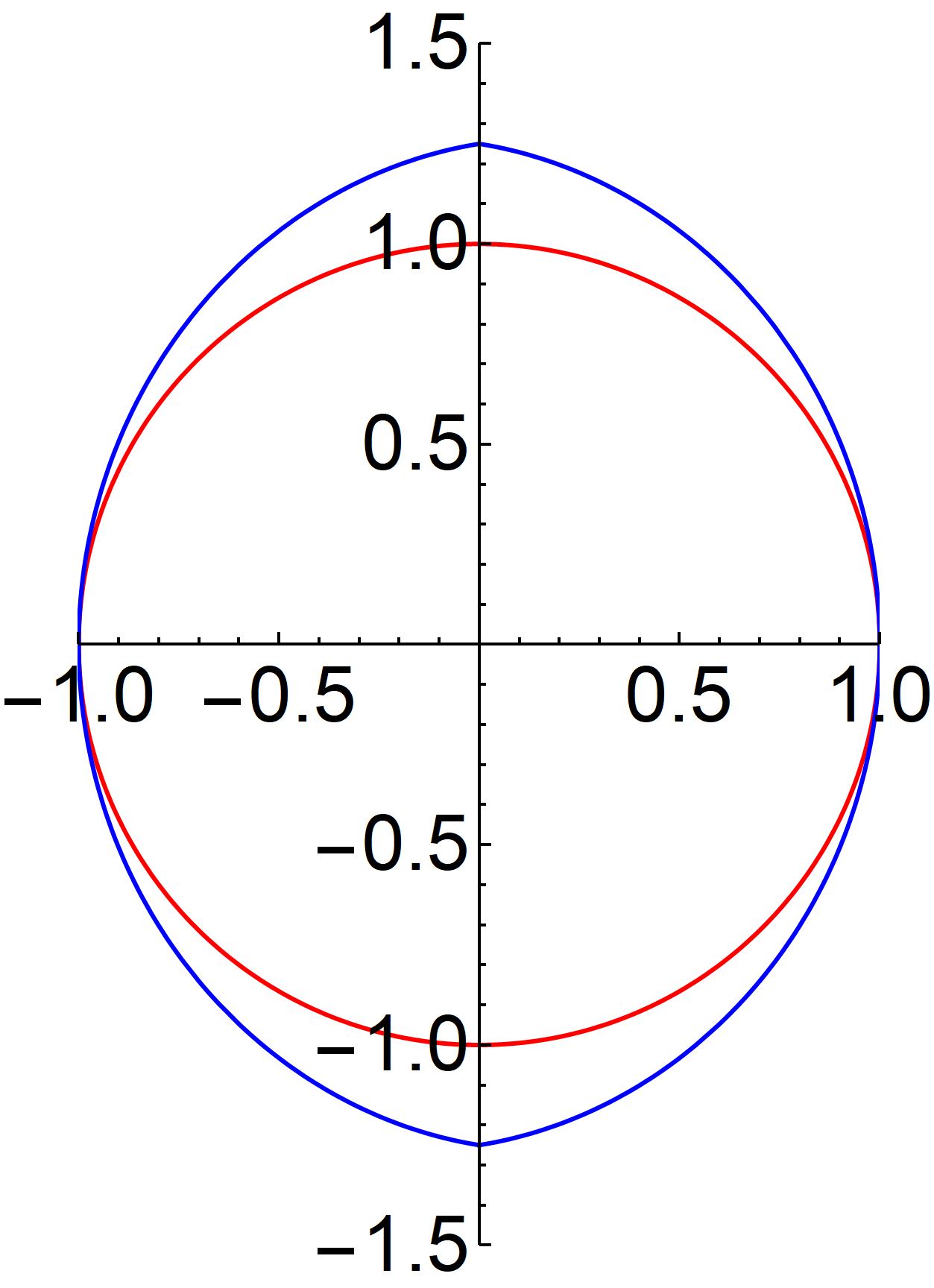}}
	}\quad
	\mbox{\subfigure[]
		{	\includegraphics[scale=0.22]{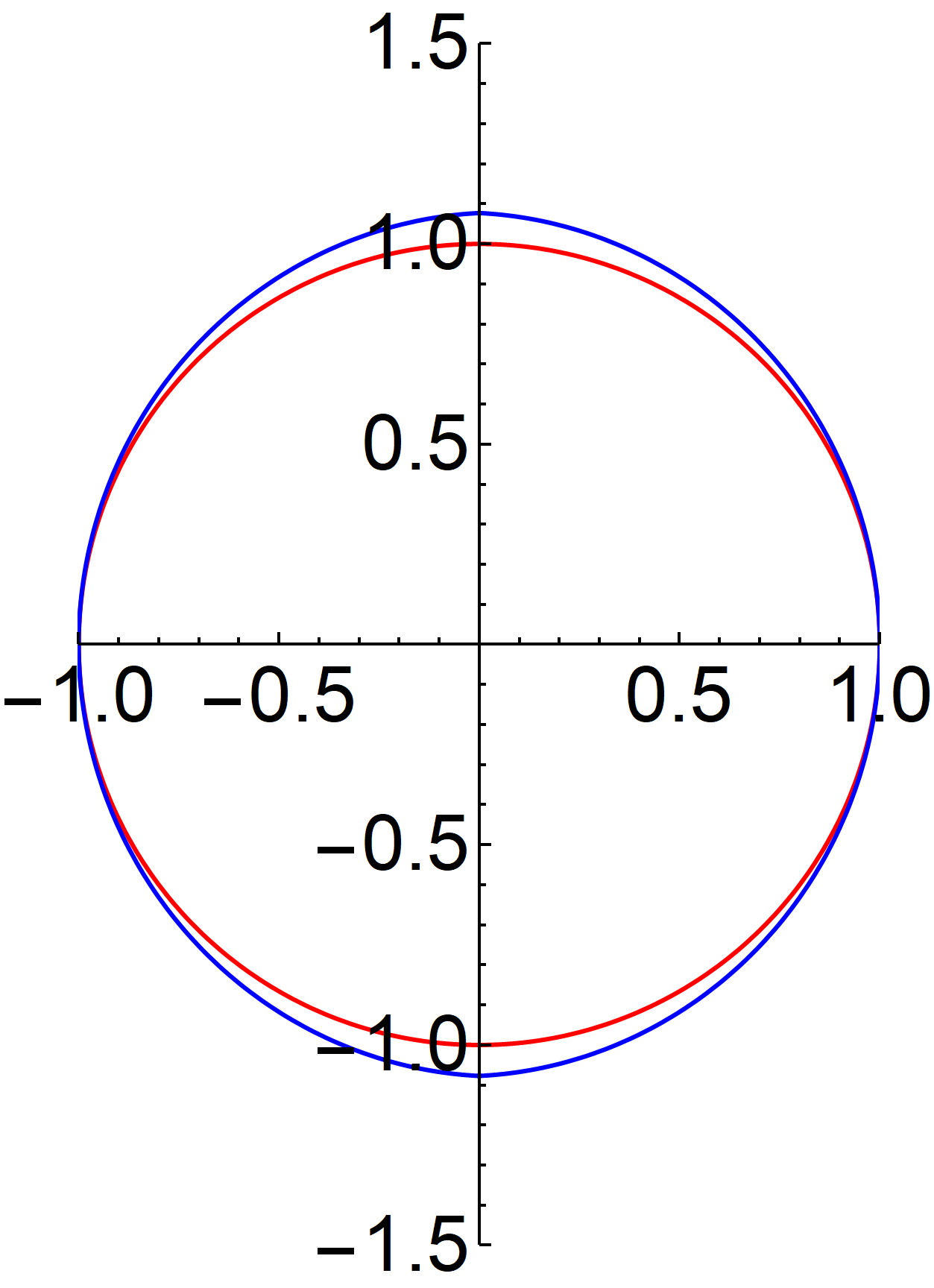}}
	}\quad
	\mbox{\subfigure[]
		{	\includegraphics[scale=0.22]{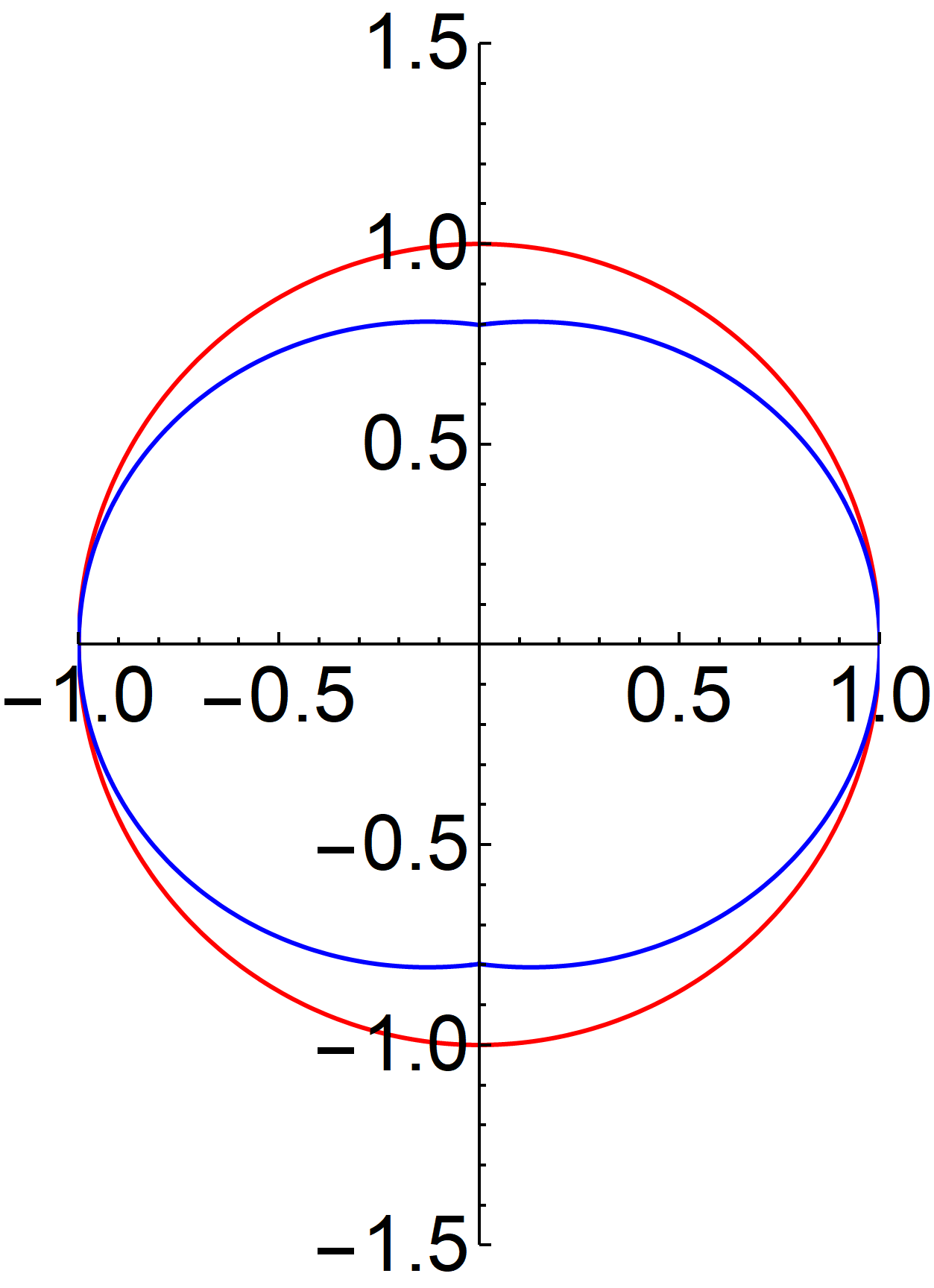}}
	}\quad
	\mbox{\subfigure[]
		{	\includegraphics[scale=0.22]{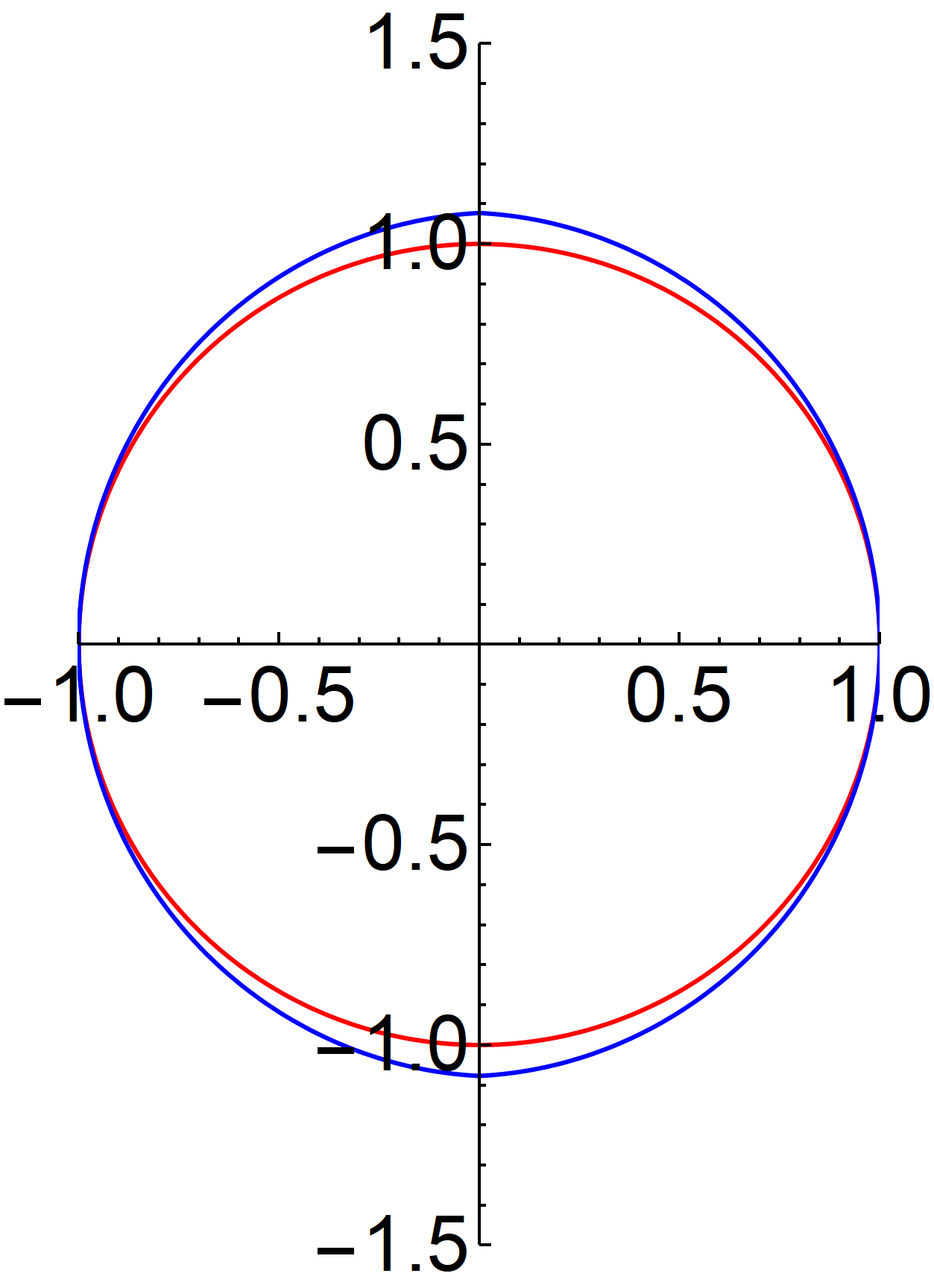}}
	}\quad
	\mbox{\subfigure[]
		{	\includegraphics[scale=0.22]{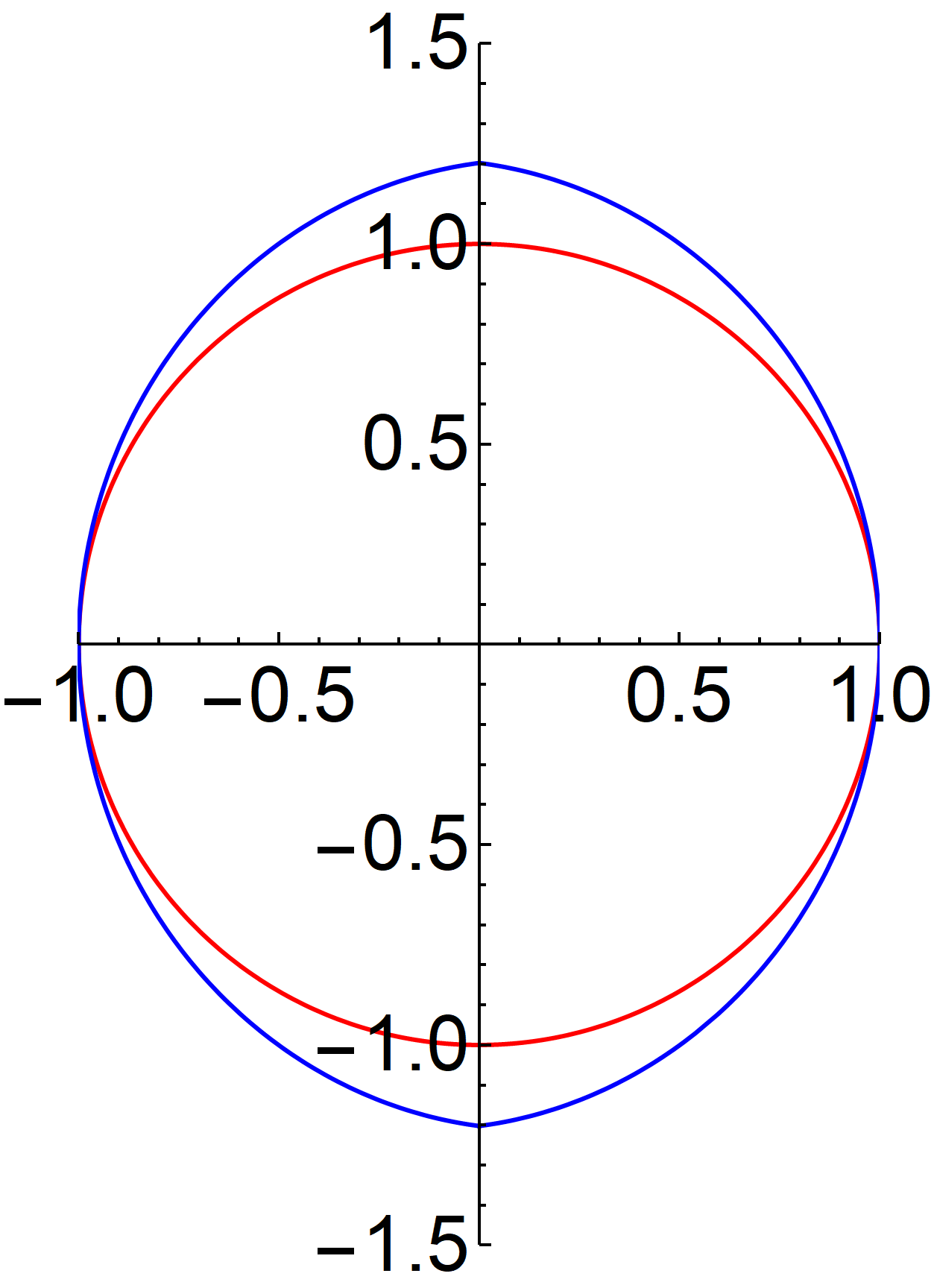}}
	}
	\vspace{0.5cm}
	\caption{Graphs of the solution $u(t)$ of the linear wave equation \eqref{eq:wlinear} with $g=\varphi_2/4$, $h=0$, and $f=0$. The graphs are drawn with  $t=k\pi\sqrt{\pi}/9$, $k=0,2,4,8,9$. Note that $u(t)$ undergoes periodic motion.}
	\label{NW}
\end{figure}

\begin{figure}[H]
	\centering
	\mbox{\subfigure[]
		{	\includegraphics[scale=0.2]{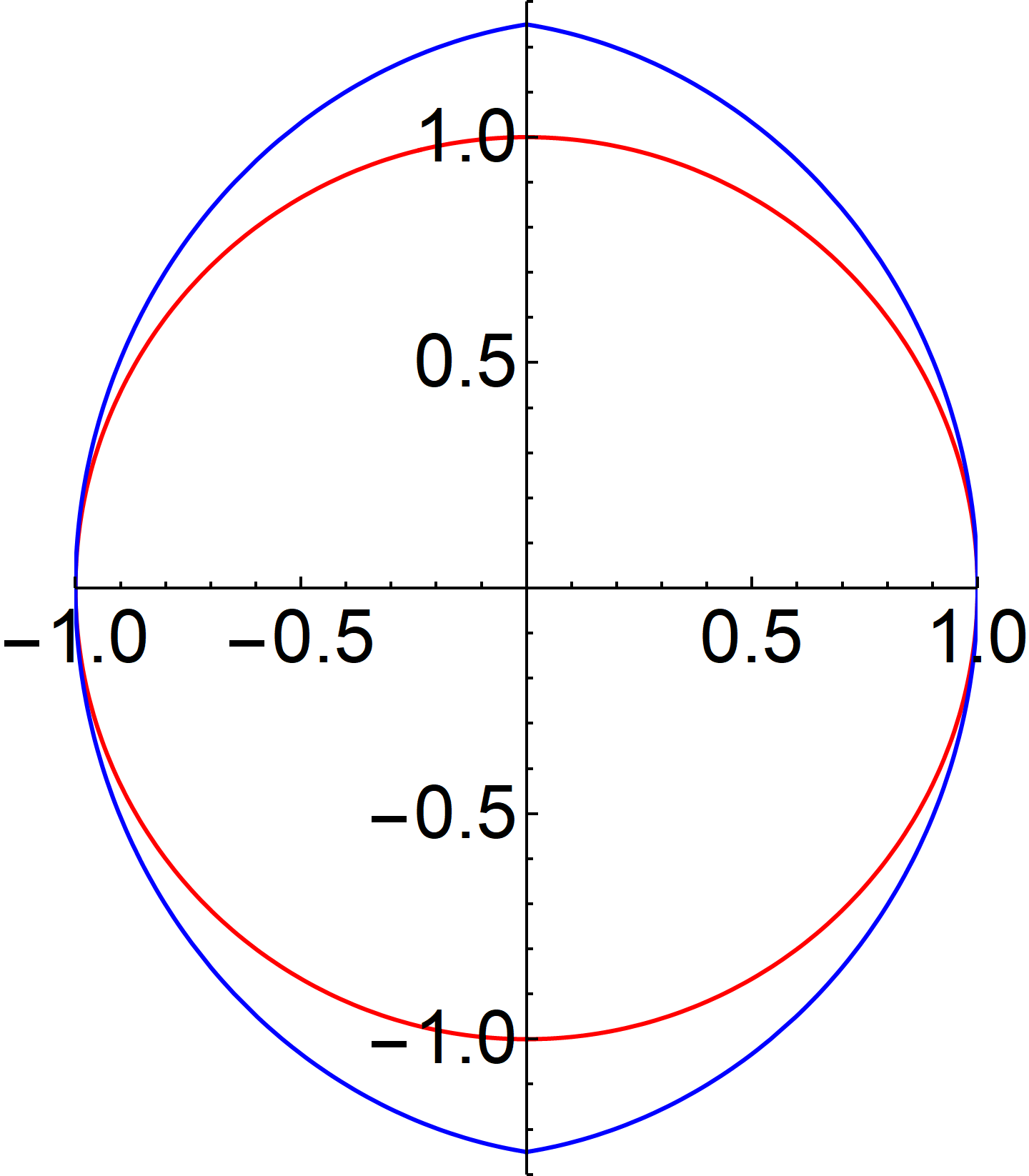}}
	}\quad
	\mbox{\subfigure[]
		{	\includegraphics[scale=0.2]{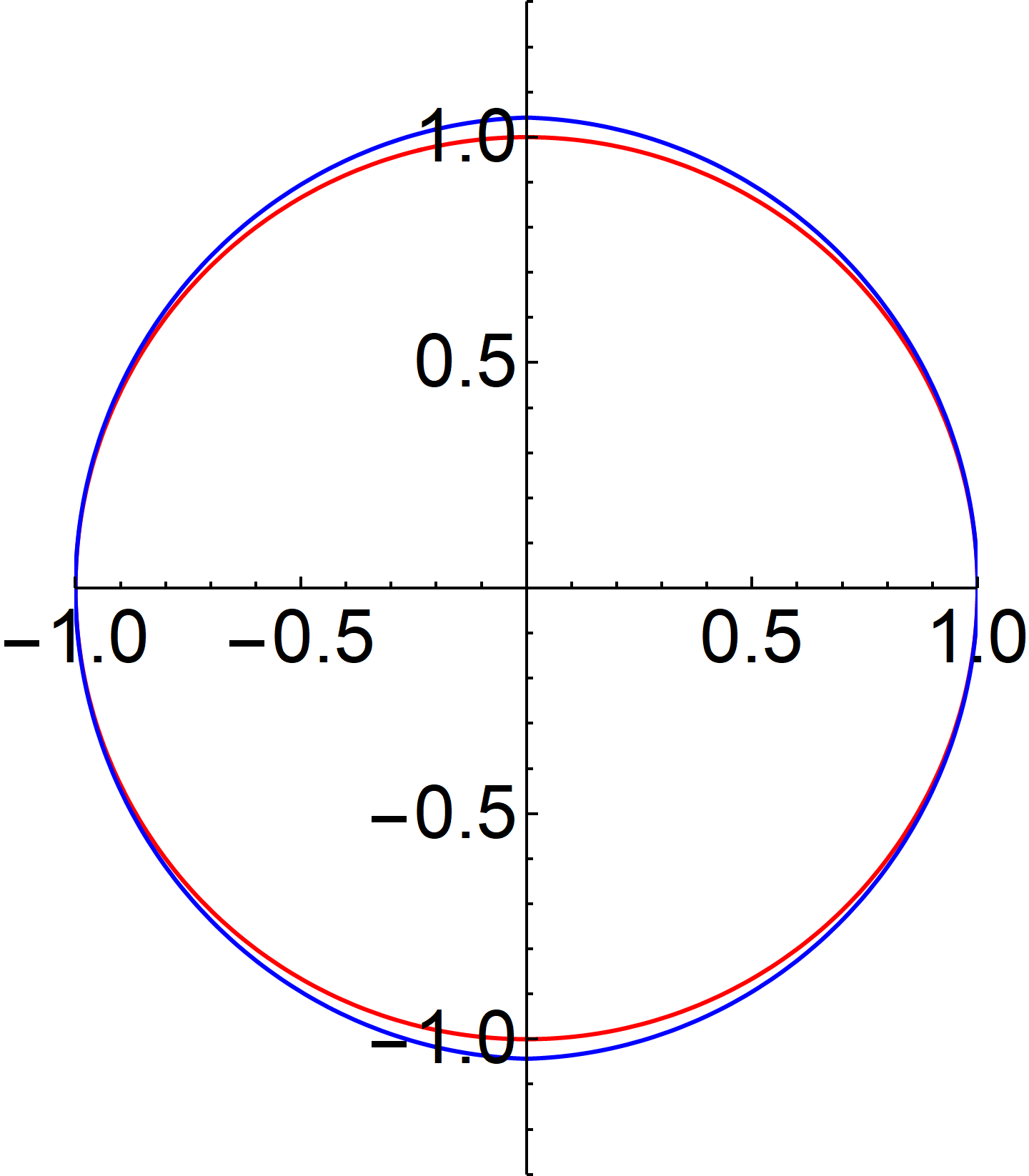}}
	}\quad
	\mbox{\subfigure[]
		{	\includegraphics[scale=0.2]{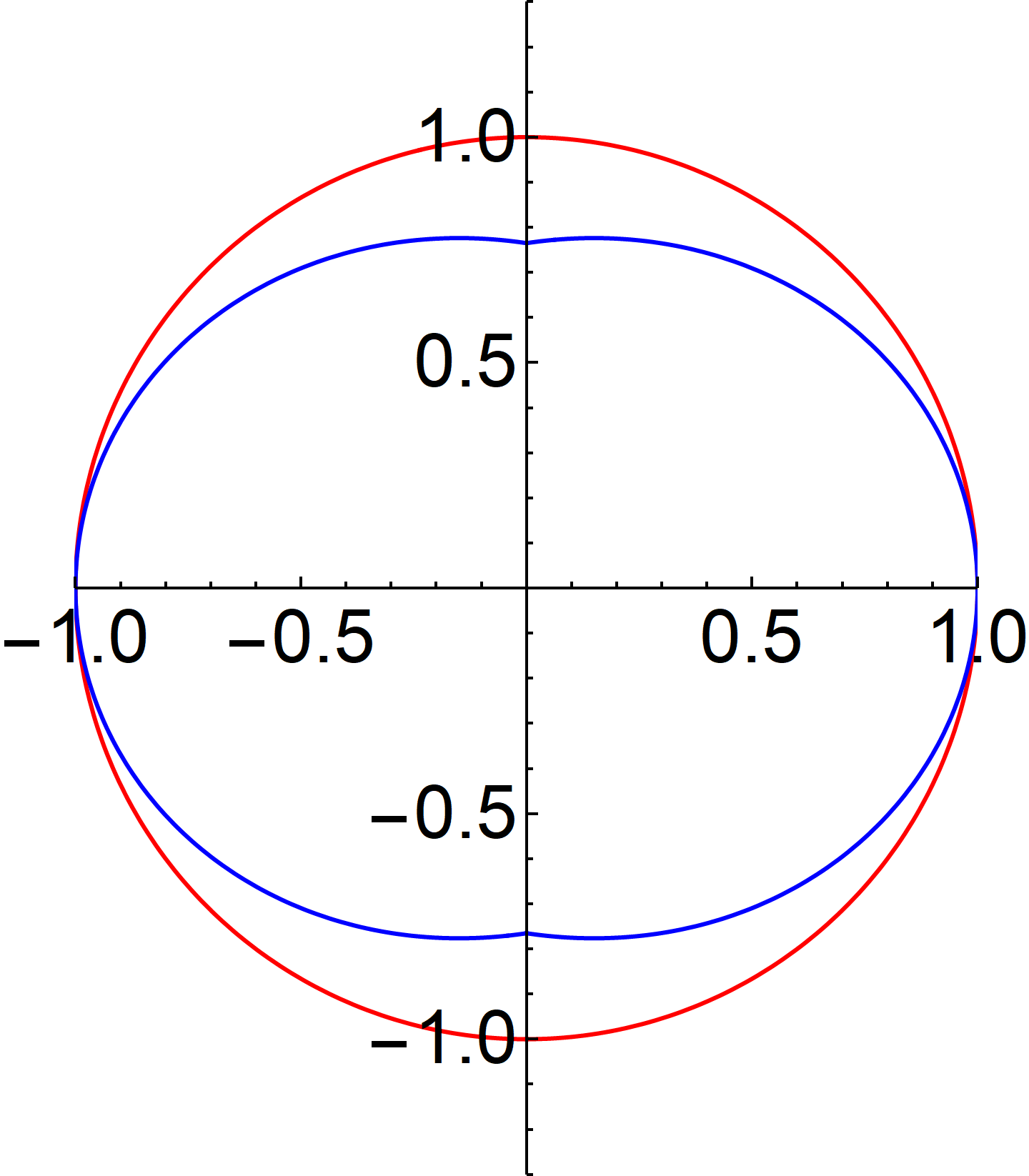}}
	}\quad
	\mbox{\subfigure[]
		{	\includegraphics[scale=0.2]{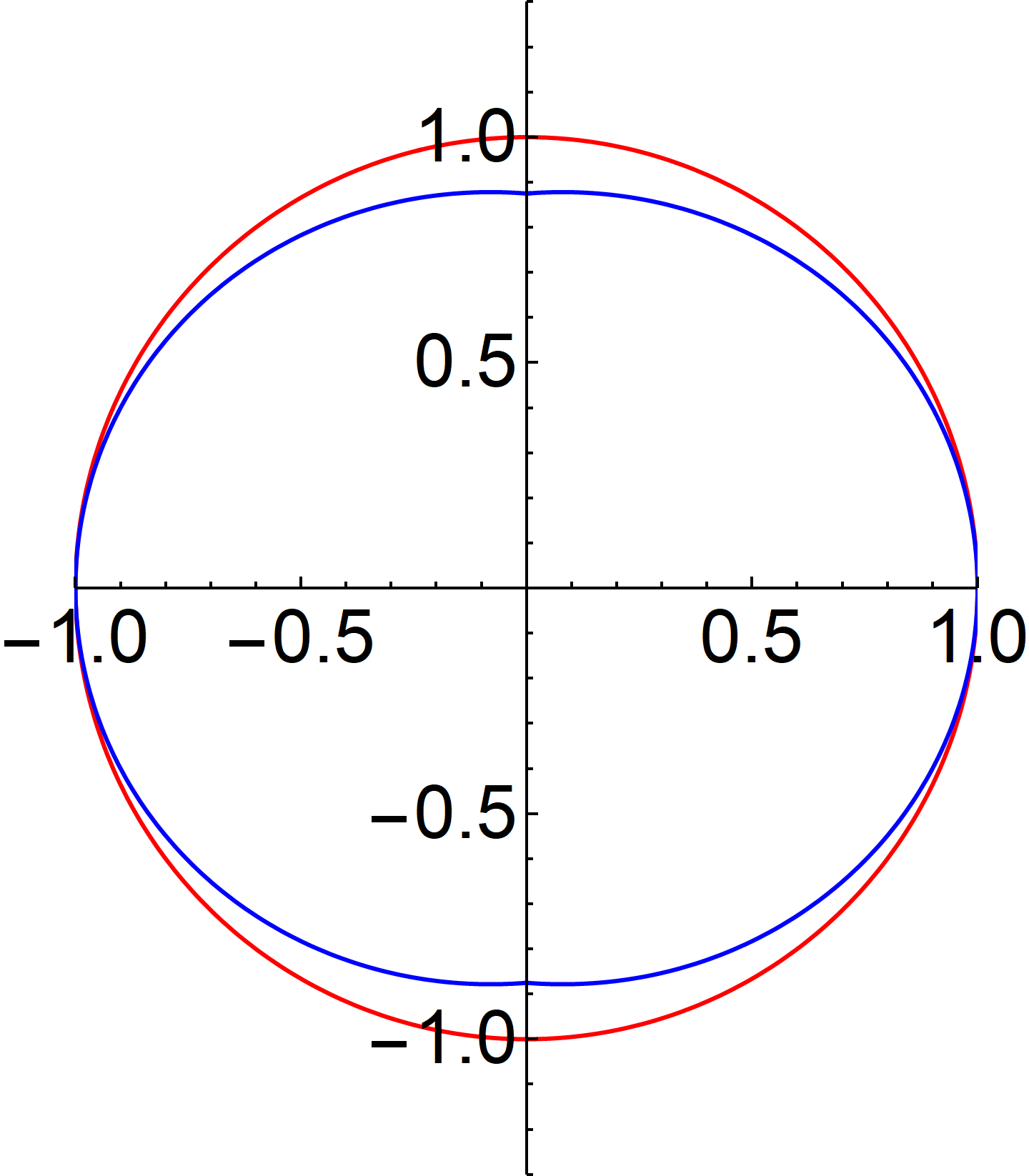}}
	}\quad
	\mbox{\subfigure[]
		{	\includegraphics[scale=0.2]{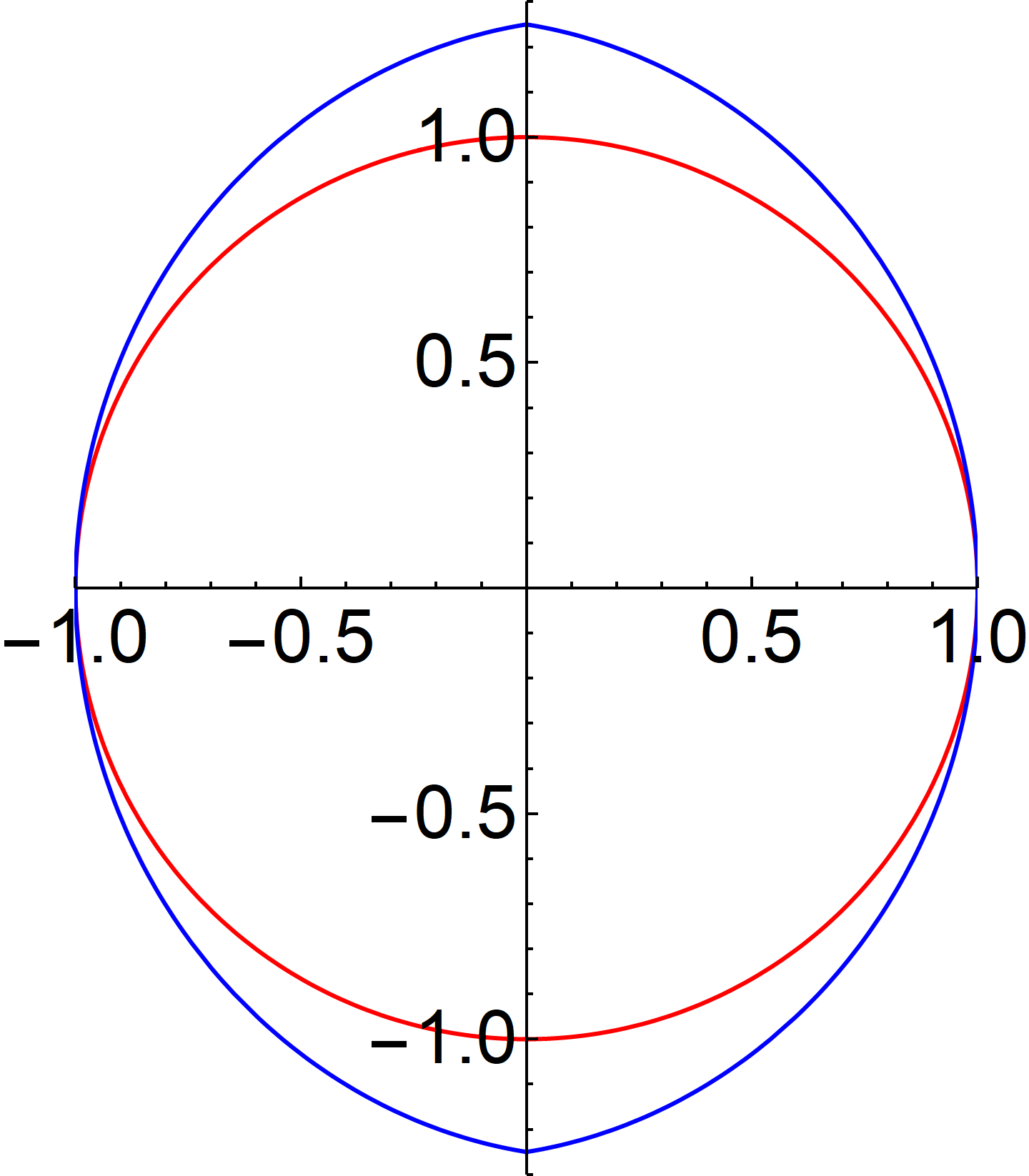}}
	}
	\vspace{0.5cm}
	\caption{Real part of the solution  $u(t)$ of the linear  Schr\"odinger equation \eqref{eq:slinear} with $g=\varphi_2/4$ and $f=0$. The graphs are ploted with $t=k\pi^2/2$, $k=0,2,4,6,9$.}
	\label{schr}
\end{figure}

\begin{figure}[H]
	\centering
	\mbox{\subfigure[]
		{	\includegraphics[scale=0.2]{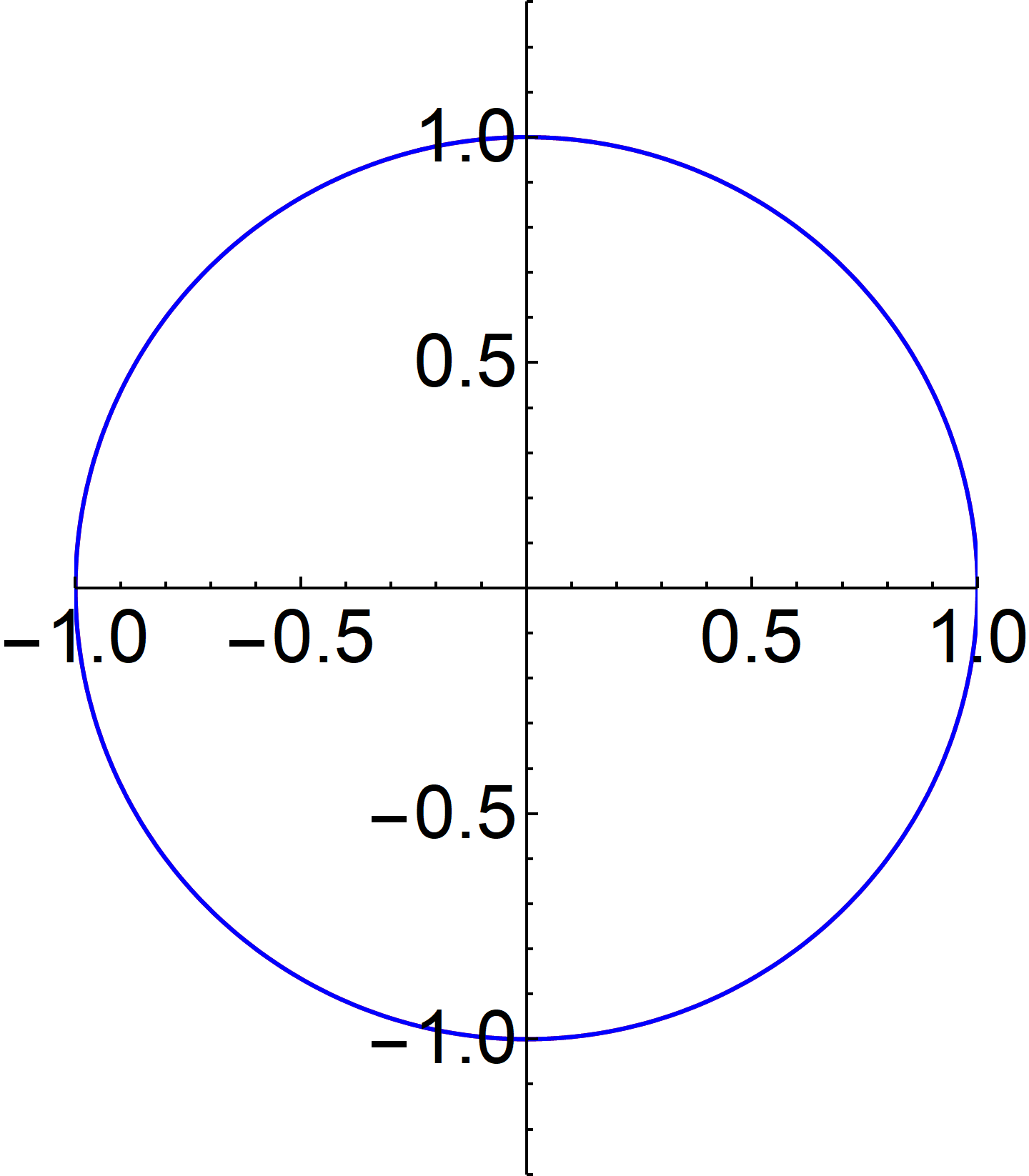}}
	}\quad
	\mbox{\subfigure[]
		{	\includegraphics[scale=0.2]{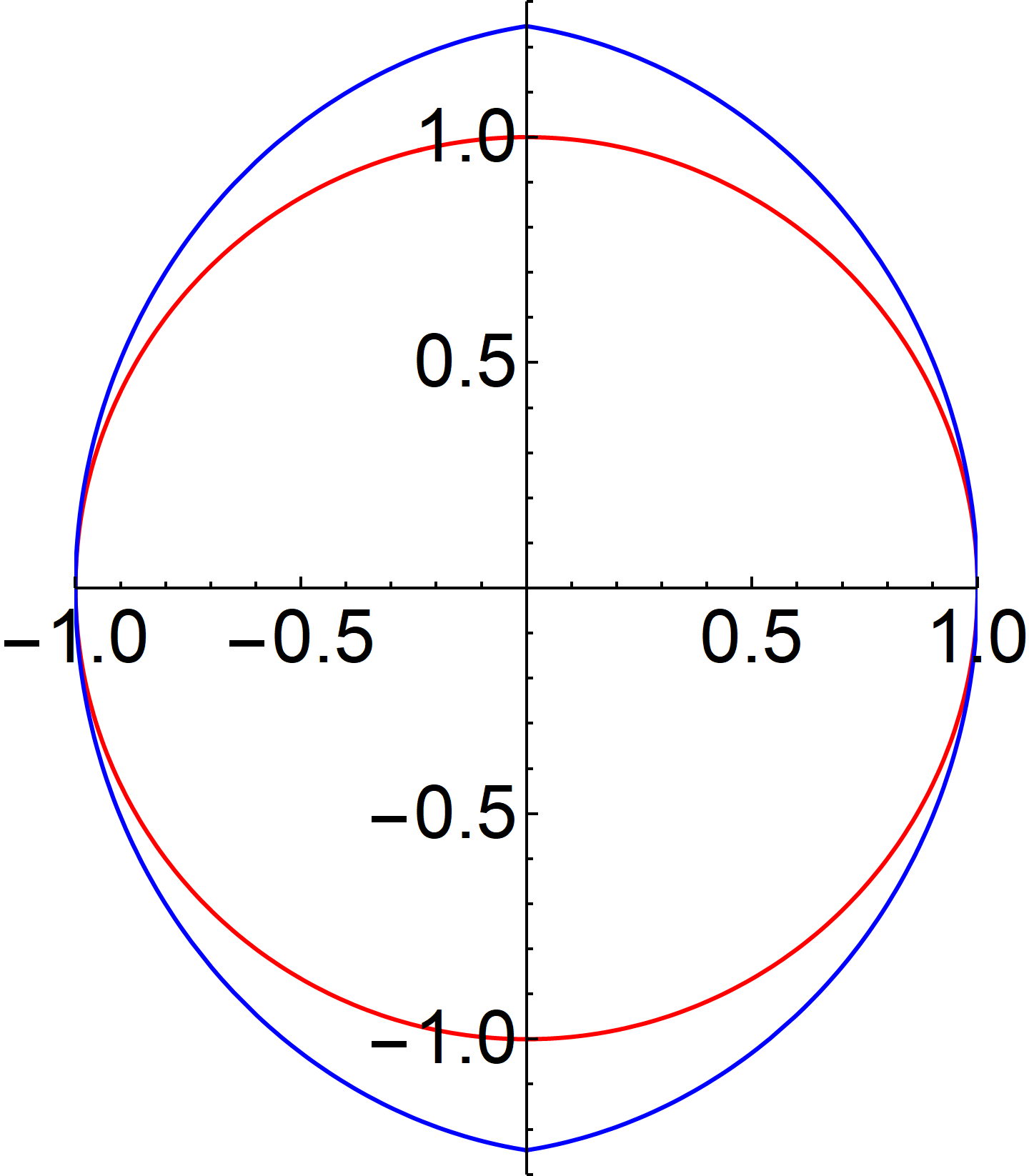}}
	}\quad
	\mbox{\subfigure[]
		{	\includegraphics[scale=0.2]{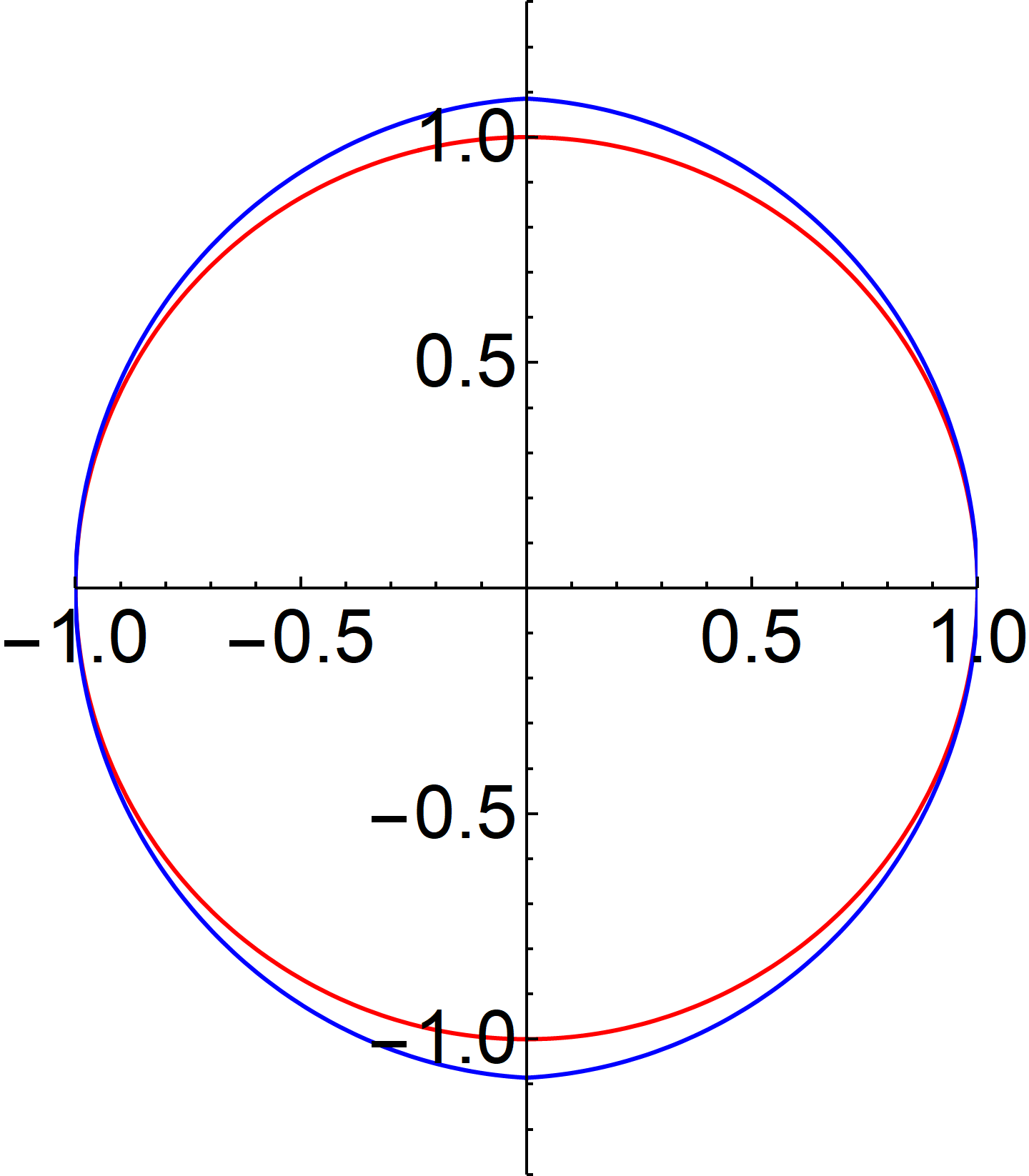}}
	}\quad
	\mbox{\subfigure[]
		{	\includegraphics[scale=0.2]{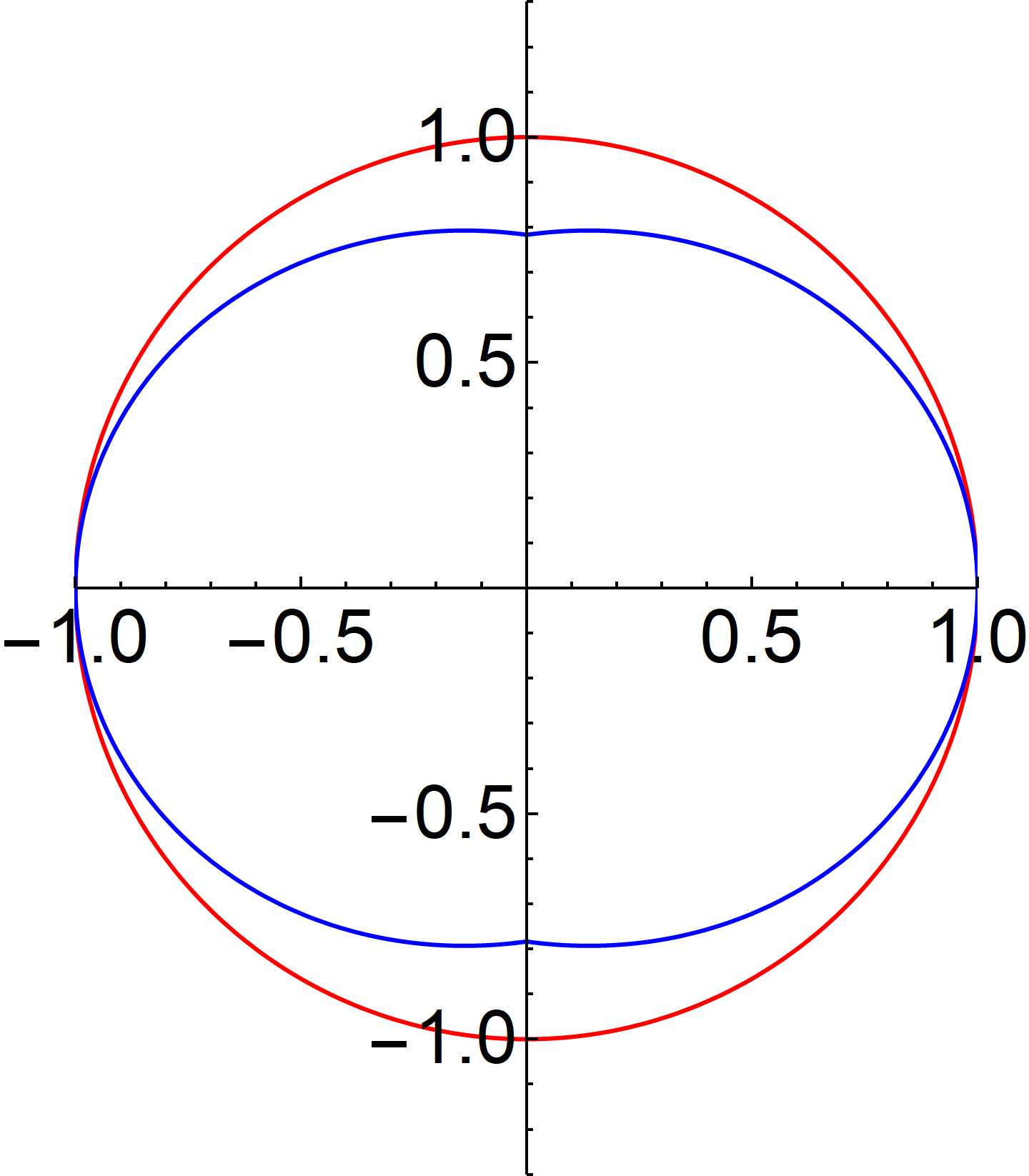}}
	}\quad
	\mbox{\subfigure[]
		{	\includegraphics[scale=0.2]{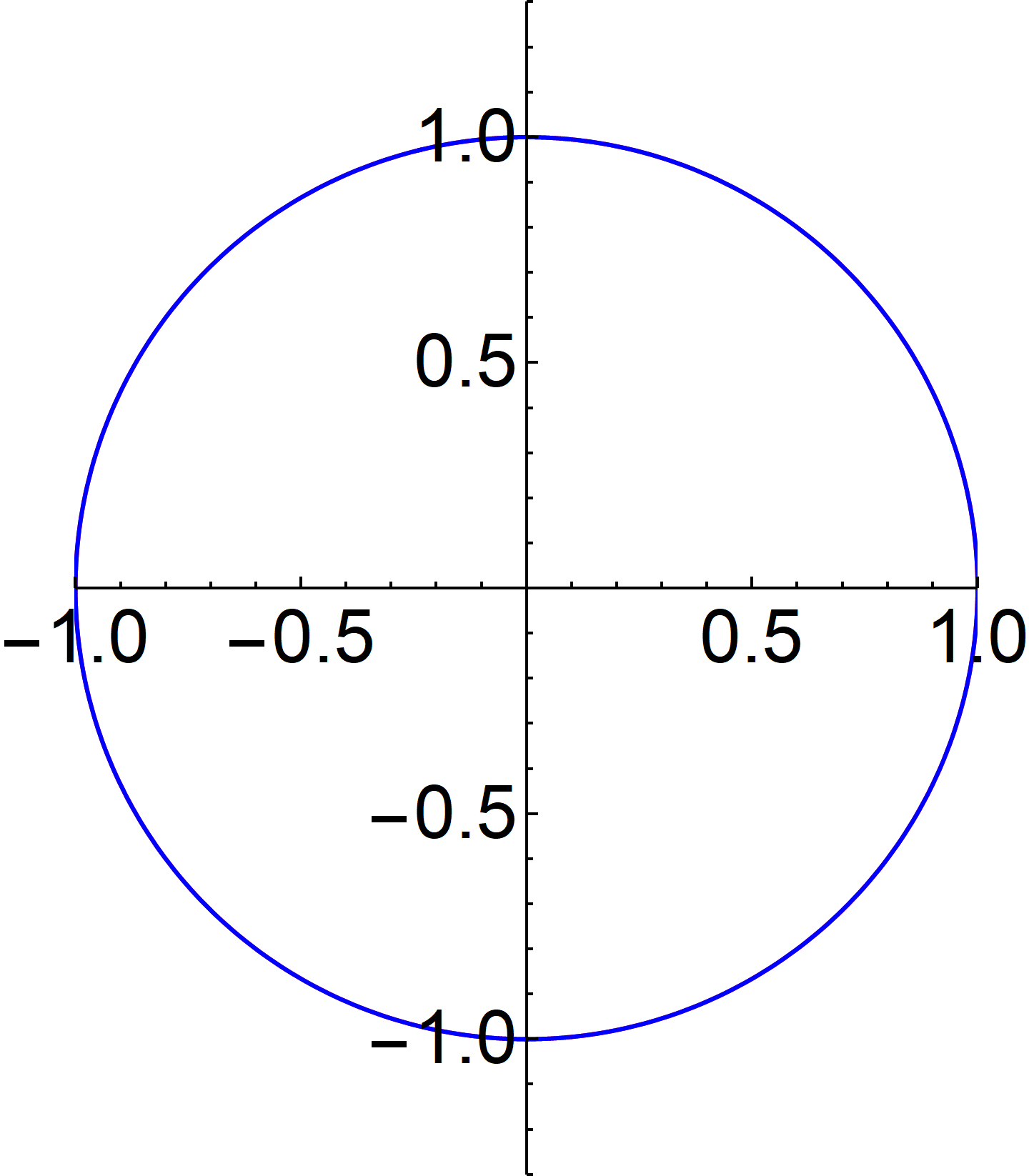}}
	}
	\vspace{0.5cm}
	\caption{Imaginary part of the solution  $u(t)$ of the linear  Schr\"odinger equation \eqref{eq:slinear}. The graphs are ploted by using the same initial data and $t$ values as in Figure.}
	\label{ima}
\end{figure}

\section*{Acknowledgement}
	
	Part of this work was carried out while the second author was visiting Beijing Institute of Mathematical Sciences and Applications (BIMSA). She thanks the institute for its hospitality and support.

\end{document}